\documentclass[11pt,letterpaper]{amsart}

\usepackage{amsmath,amscd}
\usepackage{amssymb}
\usepackage{amsthm}
\usepackage{mathtools}

\usepackage{graphicx}

\usepackage{mathrsfs}
\usepackage[ocgcolorlinks, linkcolor=blue]{hyperref}

\usepackage{calc}
             {\begin{list}{\arabic{enumi}.}{\usecounter{enumi}%
              \setlength{\labelsep}{0.5em}%
              \settowidth{\labelwidth}{(\arabic{enumi})}%
              \setlength{\leftmargin}{\labelwidth+\labelsep}}}%
             {\end{list}}

\usepackage{comment}

\usepackage{tikz}
\usetikzlibrary{arrows.meta}
\definecolor{region}{RGB}{200,200,212}   
\definecolor{pen}{RGB}{0,0,0}       

\DeclareRobustCommand{\SkipTocEntry}[5]{}

\usepackage{xcolor}
\definecolor{LOcolor}{RGB}{150,100,0}

\newtheorem{Theorem}{Theorem}[section]

\newtheorem{Lemma}[Theorem]{Lemma}
\newtheorem{Corollary}[Theorem]{Corollary}

\newtheorem{Proposition}[Theorem]{Proposition}

\theoremstyle{definition}
\newtheorem{Definition}[Theorem]{Definition}

\newtheorem{Example}[Theorem]{Example}

\newtheorem{Remark}[Theorem]{Remark}

\numberwithin{equation}{section}

\newcommand{\mR}{\mathbb{R}}                    
\newcommand{\norm}[1]{\lVert #1 \rVert}         

\newcommand{\ol}[1]{\overline{#1}}

\newcommand{\supp}{\mathrm{supp}}

\newcommand{\id}{\mathrm{id}}

\newcommand{\eps}{\varepsilon}

\newcommand{\p}{\partial}

\newcommand{\gm}{g_{\mathrm{Min}}}

\newcommand{\tbl}{\textcolor{blue}}

\def\p{\partial}
\def\R{\mathbb R}

\DeclareMathOperator{\WF}{WF}

\def\mid{\,:\,}

\newcounter{sidenote}

\begin{document}

\title{Semiglobal uniqueness for the Lorentzian Calder\'on problem} 

\author[L. Oksanen]{Lauri Oksanen}
\address{Department of Mathematics and Statistics, University of Helsinki, PO Box 68, 00014 Helsinki, Finland}
\email{lauri.oksanen@helsinki.fi}

\author{Rakesh}
\address{Department of Mathematical Sciences, University of Delaware, Newark, DE 19716, USA}
\email{rakesh@udel.edu}

\author[M. Salo]{Mikko Salo}
\address{Department of Mathematics and Statistics, University of Jyv\"askyl\"a, PO Box 35, 40014 Jyv\"askyl\"a, Finland}
\email{mikko.j.salo@jyu.fi}




\begin{abstract}
We prove uniqueness in the Lorentzian Calder\'on problem in a semiglobal setting, comparing a potentially large perturbation of the Minkowski metric against a small one. Our proof uses a reduced amount of data, solving a formally determined version of the problem. In contrast to traditional approaches, it employs neither control-theoretic arguments nor a reduction to a geometric inverse problem via microlocal methods or high-frequency solutions. Instead, it relies on distorted plane waves and weighted $L^2$-estimates. 
\end{abstract}

\maketitle

\section{Introduction} \label{sec_intro}

The anisotropic Calder\'on problem amounts to recovering the coefficients of the most fundamental generally covariant elliptic partial differential equation, that is, the Laplace equation on a Riemannian manifold. More precisely, if $M$ is a compact manifold with boundary and $g$ is a Riemannian metric on $M$, then $g$ should be recovered, up to the natural diffeomorphism gauge, given the Cauchy data set 
\[
C_g^{\mathrm{Ell}} = \{ (u|_{\p M}, \p_{\nu} u|_{\p M}) \mid u \in C^{\infty}(M) \text{ solves } \Delta_g u = 0 \text{ in $M$} \}.
\]
This has remained an open problem for more than 35 years if $\dim(M) \geq 3$ \cite{lee1989}. Results such as \cite{sylvester1987, lassas2001, dos-santos-ferreira2009} impose assumptions that are not satisfied by small perturbations of the Euclidean metric. 

We solve the Lorentzian analogue of the anisotropic Calder\'on problem for small perturbations of the Minkowski metric. In fact, formulated as a uniqueness statement comparing two metric tensors, our result requires only one of them to be close to Minkowski.

\subsection{Statement of the results}

If $M$ is a compact manifold with piecewise smooth boundary and $g$ is a smooth Lorentzian metric in $M$, we define the Cauchy data set 
\[
C_g^{\mathrm{Hyp}} = \{ (u|_{\p M}, \p_{\nu} u|_{\p M}) \mid u \in C^{\infty}(M) \text{ solves } \Box_g u = 0 \text{ in $M$} \}.
\]
Here $\Box_g$ is the Lorentzian wave operator on $(M,g)$ and $\p_{\nu} u$ is the normal derivative of $u$ with respect to some fixed reference Riemannian metric on $M$. We refer the reader to  \cite{oksanen2024} for the Lorentzian geometry facts and notation used in this article. 

We establish two uniqueness results for the Lorentzian Calder\'on problem: one of local character and one of semiglobal character. We begin with the local uniqueness result.

\begin{Theorem} \label{thm_main1}
Let $M$ be a compact domain in $\mR^{1+n}$, $n \geq 2$, with smooth boundary, such that $\mR^{1+n} \setminus M$ is connected. There is $\eps > 0$ such that for any smooth Lorentzian metrics $g_1, g_2$ in $M$ with 
\begin{gather*}
\norm{g_j-\gm}_{C^{2}(M)} \leq \eps, \\
g_j = \gm \text{ to infinite order on $\p M$},
\end{gather*}
the condition $C_{g_1}^{\mathrm{Hyp}} = C_{g_2}^{\mathrm{Hyp}}$ implies 
\[
g_2 = \Psi^* g_1
\]
for some diffeomorphism $\Psi: M \to M$ with $\Psi|_{\p M} = \id$.
\end{Theorem}

To formulate the semiglobal result, let $B$ be the open unit ball in $\mR^n$, let 
    \begin{align*}
Q_T = (-T,T) \times B, \quad M = \ol{Q}_T,
    \end{align*}
and consider a smooth Lorentzian metric $g$ in $\mR^{1+n}$ satisfying
\begin{gather}
g = \gm \text{ outside $M$}, \label{assumption1} \\
g \text{ is globally hyperbolic in $\mR^{1+n}$.} \label{assumption2}
\end{gather}

\begin{Theorem} \label{thm_main_semiglob}
Let $g, g'$ be smooth Lorentzian metrics  in $\mR^{1+n}$, $n \ge 2$,  satisfying \eqref{assumption1}--\eqref{assumption2}. There is $\eps > 0$ such that if
\begin{gather*}
\norm{g'-\gm}_{C^{2}(M)} \leq \eps,
\end{gather*}
the condition $C_{g}^{\mathrm{Hyp}} = C_{g'}^{\mathrm{Hyp}}$ implies 
\[
g = \Psi^* g'
\]
for some diffeomorphism $\Psi: \mR^{1+n} \to \mR^{1+n}$ with $\Psi = \id$ outside $M$.
\end{Theorem}

The essential difference between Theorems \ref{thm_main1} and \ref{thm_main_semiglob} lies in their assumptions on the two metrics: the former requires both to be close to Minkowski, whereas the latter imposes this condition on only one of them. 

Both results prove uniqueness on a compact domain $M$. The Lorentzian Calder\'on problem was recently shown to be unsolvable on certain non-compact domains $M$, even when $g$ satisfies \eqref{assumption1}--\eqref{assumption2} and $g' = \gm$ \cite{2604.20320}.  

\begin{Remark}
There is an additional conformal invariance in the case $n = 1$, since $\Box_{\kappa g} = \kappa^{-1} \Box_g$ for strictly positive $\kappa \in C^\infty(\R^{1+1})$. Our proof of Theorem~\ref{thm_main_semiglob} applies to this case, apart from the final step, that is, Proposition~\ref{prop_conformal_factor}. For $n=1$ we can conclude that $g$ and $g'$ coincide up to a diffeomorphism and a conformal factor in the sense of \eqref{eq_kappa_Psi}. Theorem \ref{thm_main1} holds for $n = 1$ as well, when adjusted to take into account the conformal invariance in the same manner.
\end{Remark}

We will derive both results from a more general theorem concerning a formally determined inverse problem. Consider a Lorentzian metric $g$
satisfying \eqref{assumption1}--\eqref{assumption2}. Let $U = U_{\omega,s}^g$ be the unique solution of 
\[
\Box_g U = 0 \text{ in $\mR^{1+n}$}, \qquad U = H(t-x \cdot \omega-s) \text{ for $t \ll 0$},
\]
which exists by global hyperbolicity (Lemma \ref{lemma_U_well_defined}). Here $H$ is the Heaviside function, $\omega$ is a unit vector, and $s \in \R$. We call $U_{\omega,s}^g$ the {\em distorted plane wave} propagating in direction $\omega$ with delay $s$. 
Write 
\[
\Sigma_T = (-T,T) \times \p B,
\]
and consider the measurement 
\[
\mathcal{F}_{\omega,T}(g) = (U_{\omega,s}^g|_{\Sigma_{T+2}})_{s \in [-T-1,T+1]}.
\]
That is, we measure the distorted plane waves on the lateral boundary $\Sigma_{T + 2}$ for delays $|s| \leq T+1$, visualized in Figure \ref{fig_intervals}. This measurement is well defined since the wave front set of $U_{\omega,s}^g$ is disjoint from the conormal bundle of $\Sigma_{T+2}$. If $\Omega \subset S^{n-1}$ is a finite set, we further consider the measurement 
\[
\mathcal{F}_{\Omega,T}(g) = (\mathcal{F}_{\omega,T}(g))_{\omega \in \Omega}.
\]

\begin{figure}
\begin{tikzpicture}[
  >=Stealth,
  every node/.style={font=\small}
]
 
\def\T{2}      
\def\W{1.5}    
\def\M{0.2}    
 
\fill[gray!20] (-\W,-\T) rectangle (\W,\T);
 
\draw[red, thick] (-\W,\T) -- (\W,{\T+2});
 
\draw[red, thick] (\W,-\T) -- (-\W,{-\T-2});
 
\draw[dashed, thick] (-\W,{\T+2}) -- (\W,{\T+2});
\draw[dashed, thick] (-\W,{-\T-2}) -- (\W,{-\T-2});
\draw[dashed, thick] (-\W,{\T}) -- (\W,{\T});
\draw[dashed, thick] (-\W,{-\T}) -- (\W,{-\T});

\draw[thick, blue] (-\W,{\T+2}) -- (-\W,{-\T-2});
\draw[thick, blue] (\W,{\T+2}) -- (\W,{-\T-2});

\draw[thick,->] (-3,0) -- (3,0) node[right] {$x$};
\draw[thick,->] (0,-4.5) -- (0,4.5) node[above] {$t$};
  
\node[right] at (\W + \M,{\T+2})   {$t = T+2$};
\node[right] at (\W + \M,\T)       {$t = T$};
\node[right] at (\W + \M,-\T)      {$t = -T$};
\node[right] at (\W + \M,{-\T-2})  {$t = {-T-2}$};
 
\end{tikzpicture}
\caption{The region of interest $Q_T$ is shown as the shaded rectangle. The measurement $\mathcal{F}_{\omega,T}(g)$ is given on the lateral boundary $\Sigma_{T+2}$, shown as two blue lines between the levels $t = T + 2$ and $t = -T - 2$. The delay $s$ parametrizes the level sets of the eikonal solution $\varphi_{\omega}$ between $-T-1$ and $T + 1$. The two extremal level sets are shown in red.}
\label{fig_intervals}
\end{figure}
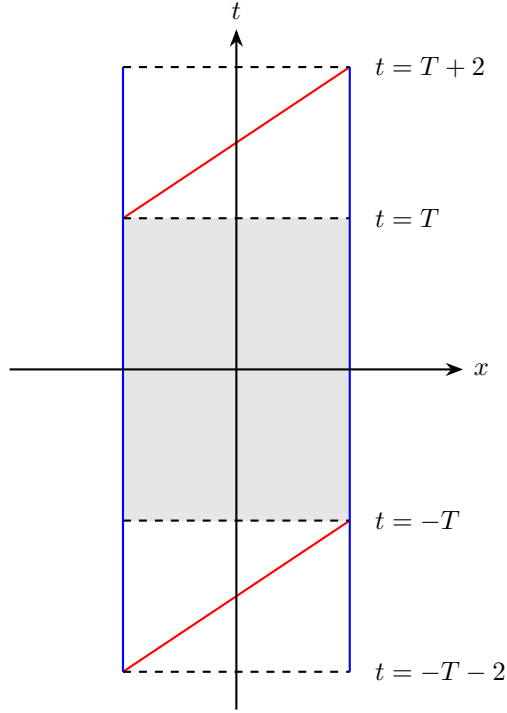

We say that $g$ is \emph{simple in direction $\omega$} if the null geodesics starting at $\{ (t,x) \,:\, x \cdot \omega = -1 \}$ with tangent vector $(1,\omega)$ smoothly parametrize a neighborhood of $\ol{Q}_{T+2}$, see Definition~\ref{simple_in_omega} for more details. If this holds, then the eikonal equation 
    \begin{align*} 
g(d\varphi, d\varphi) = 0, \qquad \varphi|_{\{ x \cdot \omega \leq -1 \}} = t-x \cdot \omega,
    \end{align*}
has a smooth solution $\varphi = \varphi_{\omega}$ near $\ol{Q}_{T+2}$ (i.e.\ in a neighborhood of $\ol{Q}_{T+2}$). If $g'$ is another such metric, we denote the corresponding solution of the eikonal equation by $\varphi_{\omega}'$. 

\begin{Theorem} \label{thm_main2}
Let $g, g'$ be smooth Lorentzian metrics  in $\mR^{1+n}$, $n \ge 2$, satisfying \eqref{assumption1}--\eqref{assumption2}, and suppose that for some finite set $\Omega \subset S^{n-1}$ the following conditions hold:
\begin{align}
 &\text{$g'$ is simple in any direction $\omega \in \Omega$}, \label{assumption3} \\[4pt]
 & \left\{ \begin{array}{c}
 \text{there are $\omega_0, \ldots, \omega_n \in \Omega$ such that $(\varphi_{\omega_0}', \ldots, \varphi_{\omega_n}')$ is} \\ 
 \text{a diffeomorphism from a neighborhood of $\ol{Q}_T$ onto} \\
 \text{a neighborhood of $[-T-1,T+1]^{1+n}$}, \end{array} \right. \label{assumption4}  \\[4pt]
  & \left\{ \begin{array}{c}
  \text{for any $z \in \ol{Q}_T$, the set $\{ g'(z) \} \cup \{ (d\varphi_{\omega}' \otimes d\varphi_{\omega}')(z) \}_{\omega \in \Omega}$} \\
  \text{spans the space of symmetric $2$-tensors at $z$.} \end{array} \right. \label{assumption5} 
\end{align}
If $\mathcal{F}_{\Omega,T}(g) = \mathcal{F}_{\Omega,T}(g')$, then  
\[
g = \Psi^* g'
\]
for some diffeomorphism $\Psi: \mR^{1+n} \to \mR^{1+n}$ with $\Psi = \id$ outside $\ol{Q}_T$.
\end{Theorem}

The directional simplicity condition \eqref{assumption3} ensures that eikonal solutions $\varphi_{\omega}$ are smooth and that the distorted plane waves $U_{\omega,s}^g$ are conormal distributions. Condition \eqref{assumption4} gives a coordinate system consisting of eikonal solutions, and the corresponding \emph{eikonal gauge} for the metric $g'$ is important for our method. The spanning condition \eqref{assumption5} is required in the proof to relate the metrics to eikonal solutions.

We note that the conditions \eqref{assumption3}--\eqref{assumption5} hold for $\gm$ and
    \begin{align*}
\Omega_{\mathrm{Min}} = \{ \pm e_1, \ldots, \pm e_n \} \cup \{ \frac{e_j+e_k}{\sqrt{2}} \mid 1 \leq j < k \leq n \}.
    \end{align*}
They are also stable under $C^{2}$-small perturbations of the metric, see Section~\ref{sec_proof_of_main_ths} for details. Thus Theorem \ref{thm_main2} implies Theorem \ref{thm_main_semiglob}. Related results for time-independent coefficients 
are given in the companion article \cite{rakesh2026}.

\subsection{Previous literature}

Barring the real analytic case \cite{lee1989, lassas2001}, the above three theorems give the first uniqueness results for a Calder\'on type problem comparing two fully anisotropic metric tensors. Previous results are largely confined to cases where the metric tensor has a translation symmetry up to a conformal factor, that is, the manifold admits a splitting $M \subset \R \times M_0$ and the metric tensor is a conformally rescaled product 
    \begin{align}\label{cta}
g(t, x) = c(t, x) (\pm dt^2 + h(x)), \quad (t, x) \in M,
    \end{align}
where $h$ is a Riemannian metric on $M_0$. Furthermore, most previous results recover one of the coefficients $c$ and $h$ while the other is known or fixed. 

When $c = 1$ identically and the sign is negative in \eqref{cta}, the coefficients are time-independent, and the problem can be formulated equivalently as an inverse boundary spectral problem \cite{katchalov2001}. These problems are well-understood, with a tradition going back to Gelfand and Levitan \cite{Gelfand1951, gelfand1957}. They are typically solved using the Boundary Control method \cite{belishev1988}, which relies on Tataru's unique continuation result across non-characteristic hypersurfaces \cite{tataru1995} and generalizes to time-dependent coefficients only insofar as the dependence is real-analytic \cite{eskin2007}. When the dependence is smooth, unique continuation may fail across non-pseudoconvex hypersurfaces \cite{alinhac1983}.

A complementary tradition concerns a fixed conformal class, that is, $h$ is assumed to be known and only $c$ is determined. The seminal result \cite{sylvester1987} recovers $c$ when $h$ is the Euclidean metric and the sign is positive in \eqref{cta}. The most general results with the positive sign \cite{dos-santos-ferreira2009,dos-santos-ferreira2016,carstea2023} impose certain geometric conditions on a known metric $h$. Similar results with the negative sign are available as well \cite{kian2019a, feizmohammadi2021b}. 

There are a handful of results that go beyond the conformally rescaled product form \eqref{cta}, however, they are all restricted in a fixed conformal class. Two such results are based on a reduction to inversion of the light ray transform. Global hyperbolicity suffices for the reduction \cite{stefanov2018}, but the transform has been inverted only in real analytic or stationary spacetimes \cite{stefanov2017, oksanen2025b}. Furthermore, a conformal rescaling of a Lorentzian metric satisfying a curvature bound can be determined by combining elements of the Boundary Control method with unique continuation along a pseudoconvex foliation \cite{alexakisfeizmohammadioksanen2022}. 

Finally, our earlier work \cite{oksanen2024} established a rigidity version of Theorem~\ref{thm_main_semiglob} that assumes $g' = \gm$. A rigidity result for the formally determined problem was proven as well. The method was somewhat inspired by the fixed angle scattering results in \cite{rakesh2020a, rakesh2020b}, which adapt the Bukhgeim-Klibanov method for formally determined inverse problems based on Carleman estimates (see \cite{bukhgeim1981, bellassoued2017, klibanov2021} and for time-dependent coefficients \cite{takase2023}). The proof in \cite{oksanen2024} uses special features of the Minkowski metric, in particular, the fact that solutions $\varphi_{\omega}$ to the eikonal equation give wave gauge coordinates. The present setting lacks this feature, and a key element in our proof is to use an exponentially weighted $L^2$-estimate with weight given by a temporal function (Proposition~\ref{prop_carleman_psi}).

Our weighted $L^2$-estimate is similar in form to Carleman estimates, but semiclassically elliptic rather than subelliptic, and it plays a role analogous to the unique continuation arguments in \cite{oksanen2024}. The correspondence is not exact, however. The estimate is applied repeatedly, and the first $n+1$ applications bound a wave gauge defect (the quantity $\ol \theta$ in Proposition~\ref{prop_thetabar_estimate}). The remaining applications, the number of which scales like the cardinality of $\Omega$ in Theorem~\ref{thm_main2}, 
bound differences of eikonal solutions by the difference of the metrics (here we use the eikonal gauge given by \eqref{assumption4}) as well as by the wave gauge defect. Finally we invoke the spanning condition \eqref{assumption5} to estimate the difference of the metrics by differences of eikonal solutions, and prove that the metrics agree in the eikonal gauge.

\subsection{Organization of article}

Section \ref{sec_intro} is the introduction and includes the statement of the main results. In Section \ref{sec_proof_main_theorem} we prove Theorem \ref{thm_main2} based on a sequence of propositions to be proved in later sections. Section \ref{sec_directional_simplicity} proves the geometric fact that directional simplicity can be detected from boundary measurements. The required facts on distorted plane waves and eikonal solutions are proved in Section \ref{sec_plane_waves}. Section \ref{sec_carleman} states the key weighted estimate, Proposition \ref{prop_carleman_psi}, and gives an outline of its proof. In Section \ref{sec_proof_of_main_ths} we consider metrics close to $\gm$ and prove Theorems \ref{thm_main1}--\ref{thm_main_semiglob}. There are four appendices: Appendix \ref{lorentzian_integration} discusses integration on Lorentzian manifolds, Appendix \ref{degree_theory} recalls certain degree theory facts needed in the proofs, Appendix \ref{sec_carleman_appendix} gives a detailed proof of Proposition \ref{prop_carleman_psi}, and Appendix \ref{sec_diffeo_appendix} proves a result relating directional simplicity and a restricted scattering relation.

\subsection*{Acknowledgements}

L.O.\ and M.S.\ were supported by the Research Council of Finland (Centre of Excellence in Inverse Modelling and Imaging and FAME Flagship, grants 353096, 359182, 353091 and 359208, as well as a project grant 347715). L.O.\ was also supported by the European Research Council of the European Union, grant 101086697 (LoCal). Rakesh's work was partly funded by the grants DMS 1908391 and DMS 2307800 from the National Science Foundation of USA. Views and opinions expressed are those of the authors only and do not
necessarily reflect those of the European Union or the other funding
organizations.

\section{Proof of Theorem \ref{thm_main2}} \label{sec_proof_main_theorem}

We will now prove Theorem \ref{thm_main2}, based on a sequence of propositions whose proofs will be given in later sections. We assume that $g, g'$ are smooth Lorentzian metrics in $\mR^{1+n}$ satisfying \eqref{assumption1}--\eqref{assumption2}, that the metric $g'$ satisfies \eqref{assumption3}--\eqref{assumption5}, and that 
\[
\mathcal{F}_{\Omega,T}(g) = \mathcal{F}_{\Omega,T}(g').
\]
The beginning of the proof follows the same approach as in \cite{oksanen2024}, and the way of using weighted estimates is motivated by \cite{krishnan2023}. The first step is to show that also $g$ must satisfy the simplicity condition \eqref{assumption3}.

\begin{Proposition} \label{prop_simplicity_detection}
Let $g, g'$ be smooth Lorentzian metrics  in $\mR^{1+n}$ satisfying \eqref{assumption1}--\eqref{assumption2}, and suppose that $g'$ is simple in direction $\omega$. If 
\[
\mathcal{F}_{\omega,T}(g) = \mathcal{F}_{\omega,T}(g'),
\]
then also $g$ is simple in direction $\omega$.
\end{Proposition}

Now we know that both $g$ and $g'$ satisfy the simplicity condition \eqref{assumption3}. As in \cite{oksanen2024}, this ensures that for any $\omega \in \Omega$ there is a smooth solution $\varphi_{\omega}$ of the eikonal equation, and that the distorted plane waves $U_{\omega,s} = U^g_{\omega,s}$ are conormal distributions that have an explicit representation in terms of the solutions $\varphi_{\omega}$. The amplitudes of the distorted plane waves satisfy transport equations involving the vector field $Z_{g,\varphi} := -d\varphi^{\sharp}$ that is tangent to level sets $\{ \varphi = s \}$, i.e.\ 
\[
Z_{g,\varphi} v = -g(d\varphi, dv).
\]
From now on, we will refer to distorted plane waves simply as plane waves.

\begin{Proposition} \label{prop_plane_wave_structure}
Assume that $g$ satisfies \eqref{assumption1}--\eqref{assumption2} and is simple in direction $\omega$.

\begin{enumerate}
\item[(a)] 
There is a unique smooth solution $\varphi = \varphi_{\omega}$ near $\mR \times \ol{B}$ of the eikonal equation 
\[
\text{$g(d\varphi, d\varphi) = 0$}, \quad  \varphi|_{\{x \cdot \omega \leq -1\}} = t - x \cdot \omega.
\]
Moreover, $\varphi = t - x \cdot \omega$ in $\{ |t - x \cdot \omega| \geq T+1 \}$ and 
\[
\varphi(\ol{Q}_T) = \varphi(\{ |t - x \cdot \omega| \leq T+1 \} \cap (\mR \times \ol{B})) = [-T-1, T+1].
\]

\item[(b)] 
The plane wave $U = U_{\omega,s}$ has the form 
\[
U = u H(\varphi - s) \text{ near $\mR \times \ol{B}$}
\]
where $u = u_{\omega,s}$ is smooth near $\mR \times \ol{B}$ and depends smoothly on $s$.

\item[(c)]
The function $u$ in {\rm (b)} can be chosen to satisfy near $\mR \times \ol{B}$ 
\begin{align*}
\Box_g u &= (\varphi-s) a, \\
2 Z_{g,\varphi} u + (\Box_g \varphi) u &= (\varphi-s) b,
\end{align*}
where $a = a_{\omega,s}$ and $b = b_{\omega,s}$ are smooth near $\mR \times \ol{B}$ and depend smoothly on $s$, and $a$ is supported in $\{ \varphi \leq s \}$. In particular, $u$ satisfies 
\begin{align*}
\Box_g u &= 0 \text{ in $\{ \varphi \geq s \}$}, \\
2 Z_{g,\varphi} u + (\Box_g \varphi) u &= 0 \text{ on $\{ \varphi = s \}$,} \\
u &> 0 \text{ on $\{ \varphi = s \}$.}
\end{align*}
\end{enumerate}
\end{Proposition}

We denote by $\varphi_{\omega}'$ the corresponding eikonal solution for $g'$, and $U_{\omega,s}' = u_{\omega,s}' H(\varphi_{\omega,s}')$ is the corresponding plane wave for $g'$. By comparing the wave front sets of $U_{\omega,s}$ and $U_{\omega,s}'$ outside $\ol{Q}_T$ we see that $\varphi_{\omega} = \varphi_{\omega}'$ outside $\ol{Q}_T$.

\begin{Proposition} \label{prop_eikonal_detection}
Let $g, g'$ satisfy \eqref{assumption1}--\eqref{assumption2}, and suppose that $g$ and $g'$ are simple in direction $\omega$. If 
\[
\mathcal{F}_{\omega,T}(g) = \mathcal{F}_{\omega,T}(g'),
\]
then $U_{\omega,s} = U_{\omega,s}'$ and $\varphi_{\omega}  = \varphi_{\omega}'$ in $W \setminus \ol{Q}_T$ for some neighborhood $W$ of $\mR \times \ol{B}$.
\end{Proposition}

Next we use the diffeomorphism assumption \eqref{assumption4} and fix the diffeomorphism gauge in Theorem \ref{thm_main2} via eikonal solutions. Let 
\[
\Phi(z) = (\varphi_{\omega_0}(z), \ldots, \varphi_{\omega_n}(z)) \text{ near $\ol{Q}_T$},
\]
and define $\Phi'$ similarly in terms of the solutions $\varphi_{\omega_j}'$. Then $\Phi$ is a smooth map, and by \eqref{assumption4} $\Phi'$ is a diffeomorphism from a neighborhood of $\ol{Q}_T$ onto a neighborhood of $[-T-1,T+1]^{1+n}$. Moreover, by Proposition \ref{prop_eikonal_detection} we have $\Phi = \Phi'$ slightly outside $\ol{Q}_T$. 

If also $\Phi$ were a diffeomorphism, we would like to prove that the metrics agree in the \emph{eikonal gauge}, i.e.\ that 
\begin{equation} \label{eikonal_gauge_phi_inverse}
(\Phi^{-1})^* g = ((\Phi')^{-1})^* g'
\end{equation}
if $g$ and $g'$ are considered as $(0,2)$-tensors. However, at this point $\Phi$ is only a smooth map, and therefore $(\Phi^{-1})^* g$ is not well defined. The fact that $\Phi$ is not yet known to be a diffeomorphism will require considerable care in how the proof is arranged. It turns out that it will be more natural to consider $g$ and $g'$ as $(2,0)$-tensors via raising indices.

From now on we will make the interpretation that $g$ and $g'$ are $(2,0)$-tensors, so in local coordinates $g = (g^{jk})$ and $g' = ((g')^{jk})$. Instead of proving \eqref{eikonal_gauge_phi_inverse}, our objective is to show the analogous statement 
\begin{equation*} 
(\Psi_* g)_{\Psi(z)} = g'_{\Psi(z)} \text{ for all $z \in \ol{Q}_T$,}
\end{equation*}
where $\Psi = (\Phi')^{-1} \circ \Phi$. 
Here the pushforward of a $(2,0)$-tensor $h$ by a smooth map $F$ is 
\[
(F_* h)_{F(z)}(\zeta, \tilde{\zeta}) = h_z(F^* \zeta, F^* \tilde{\zeta}), \qquad \zeta, \tilde{\zeta} \in T_{F(z)}^* \mR^{1+n}.
\]
In local coordinates 
\[
(F_* h)^{ab}(F(z)) = ((DF) h (DF)^t)^{ab}(z) = \p_p F^a (z) h^{pq}(z) \p_q F^b(z).
\]

We will show next that $\Psi = (\Phi')^{-1} \circ \Phi$ is indeed well defined.

\begin{Proposition} \label{prop_psi_well_defined}
If $g$ and $g'$ satisfy \eqref{assumption1}--\eqref{assumption2}, $g'$ satisfies \eqref{assumption3}--\eqref{assumption4}, and $\mathcal{F}_{\omega_j,T} = \mathcal{F}_{\omega_j,T}$ for $0 \leq j \leq n$, then the map 
\[
\Psi := (\Phi')^{-1} \circ \Phi
\]
is smooth near $\ol{Q}_T$, $\Psi = \id$ outside $\ol{Q}_T$, and $\Psi(\ol{Q}_T) = \ol{Q}_T$.
\end{Proposition}

We extend $\Psi$ as identity to $\mR^{1+n}$, so that $\Psi$ is a smooth map on $\mR^{1+n}$ with $\Psi = \id$ outside $\ol{Q}_T$. The functions $\Psi^* \varphi_{\omega}'$ satisfy 
\[
\Psi^* \varphi_{\omega}'|_{\{ x \cdot \omega \leq -1 \}} = t - x \cdot \omega.
\]
The construction of $\Psi$ gives the important gauge fixing property that for $0 \leq j \leq n$, 
    \begin{align}\label{phi_omega_j}
\varphi_{\omega_j} = \Psi^* \varphi_{\omega_j}' \text{ near $\ol{Q}_T$.}
    \end{align}
In particular, the interfaces $\{ \varphi_{\omega_j} = s \}$ and $\{ \Psi^* \varphi_{\omega_j}' = s \}$ match.

We will now state a weighted $L^2$-estimate that will be used several times in the argument. This estimate allows us to control $Z_{g,\varphi_{\omega}} w$ on the interface 
\[
\Gamma_{\omega} := Q_T \cap \{ \varphi_{\omega} = s \}
\]
in terms of $\Box_g w$ in $Q_T \cap \{ \varphi_{\omega} > s \}$.

\begin{Proposition} \label{prop_carleman_psi}
Let $g$ be a smooth Lorentzian metric in $\mR^{1+n}$ satisfying \eqref{assumption1}--\eqref{assumption3}, and 
fix the time-orientation so that $dt$ is future-directed outside $M$.
Let $\psi \in C^{\infty}(\ol{Q}_T)$ satisfy $g(d\psi, d\psi) < 0$ with $d\psi$ future-directed. There exist $C, \sigma_0 > 0$ such that for any $\omega \in \Omega$ and $|s| \leq T+1$, the estimate 
\begin{multline*}
\sigma^3 \norm{e^{\sigma \psi} w}_{L^2(\Gamma_{\omega})}^2 + \sigma \norm{e^{\sigma \psi} Z_{g,\varphi_{\omega}} w}_{L^2(\Gamma_{\omega})}^2  + \sigma^4 \norm{e^{\sigma \psi} w}_{L^2(Q_T \cap \{ \varphi_{\omega} > s \})}^2 \\
  + \sigma^2 \norm{e^{\sigma \psi} d w}_{L^2(Q_T \cap \{ \varphi_{\omega} > s \})}^2 \leq C \norm{e^{\sigma \psi} \Box_g w}_{L^2(Q_T \cap \{ \varphi_{\omega} > s \})}^2
\end{multline*}
holds whenever $\sigma \geq \sigma_0$ and $w \in C^{\infty}(\ol{Q}_T \cap \{ \varphi_{\omega} \geq s \})$ satisfies $w = \partial_{\nu} w = 0$ on $\p Q_T \cap \{ \varphi_{\omega} \geq s \}$.
\end{Proposition}

A weaker estimate, with $\sigma^4$ and $\sigma^2$ replaced by $\sigma^3$ and $\sigma$, would follow from the Carleman estimate in \cite{rakesh2020a}. The stronger Proposition \ref{prop_carleman_psi} is a semiclassically elliptic $L^2$ estimate with boundary terms (instead of subelliptic, as in standard Carleman estimates) for the conjugated operator $e^{\sigma \psi} \circ \Box_g \circ e^{-\sigma \psi}$. The temporal weight $\psi$ ensures that the boundary terms on the interface $\Gamma_{\omega}$ come with the right sign. The estimate will be discussed in Section \ref{sec_carleman}, with further details given in Appendix \ref{sec_carleman_appendix}. A similar estimate in the Minkowski case with $\psi(t,x) = t$ was employed in \cite{krishnan2023}.

Our next step is to use a first batch of weighted estimates for the differences $u_{\omega_j,s} - \Psi^* u_{\omega_j,s}'$ in the region $Q_T \cap \{ \varphi_{\omega_j} > s \}$ where $0 \leq j \leq n$. From now on we write 
\[
\varphi_{\omega}'' := \Psi^* \varphi_{\omega}', \qquad u_{\omega,s}'' := \Psi^* u_{\omega,s}'.
\]
With this notation, for the special vectors $\omega_j$ with $0 \leq j \leq n$ we have 
\[
\text{$\varphi_{\omega_j} = \varphi_{\omega_j}''$ near $\ol{Q}_T$,}
\]
and for any $\omega \in \Omega$ we have 
\begin{gather*}
\text{$\varphi_{\omega} = \varphi_{\omega}''$ slightly outside $Q_T$}, \\
\Box_g u_{\omega,s} = 0 \text{ in $Q_T \cap \{ \varphi_{\omega} \geq s \}$}, \qquad \Psi^*(\Box_{g'} u_{\omega,s}') = 0 \text{ in $Q_T \cap \{ \varphi_{\omega}'' \geq s \}$}, \\
Z_{g,\varphi_{\omega}} u_{\omega,s} + \frac{1}{2} (\Box_g \varphi_{\omega}) u_{\omega,s}  = 0 \text{ on $Q_T \cap \{ \varphi_{\omega} = s \}$}, \\
\Psi^* \left[ Z_{g',\varphi_{\omega}'} u_{\omega,s}'+ \frac{1}{2} (\Box_{g'} \varphi_{\omega}') u_{\omega,s}' \right]  = 0 \text{ on $Q_T \cap \{ \varphi_{\omega}'' = s \}$}.
\end{gather*}

\begin{Remark}
It would be tempting to shorten notations further and write $g'' = \Psi^* g'$ and $\Box_{g''} u_{\omega,s}'' = 0$ etc. However, the expression $\Psi^* g'$ is only well defined when $g'$ is a $(0,2)$-tensor, and as discussed above, we need to think of $g$ and $g'$ as $(2,0)$-tensors (acting on cotangent vectors) in the proof. 
\end{Remark}

We will write 
\[
\bar{w}_j := u_{\omega_j,s} - u_{\omega_j,s}'', \qquad 0 \leq j \leq n.
\]
Since the boundary measurements for $g$ and $g'$ agree, Proposition \ref{prop_eikonal_detection} ensures that $\bar{w}_j = 0$ outside $Q_T$ and therefore the weighted estimate in Proposition~\ref{prop_carleman_psi} can be applied to $\bar{w}_j$. This gives 
\begin{multline} \label{barwj_carleman_estimate}
\sigma \norm{e^{\sigma \psi} Z_{g,\varphi_{\omega_j}} \bar{w}_j}_{L^2(\Gamma_{\omega_j})}^2 + \sigma^3 \norm{e^{\sigma \psi} \bar{w}_j}_{L^2(\Gamma_{\omega_j})}^2  \\
 \lesssim \norm{e^{\sigma \psi} \Box_g \bar{w}_j}_{L^2(Q_T \cap \{ \varphi_{\omega_j} > s \})}^2.
\end{multline}

The term on the right of \eqref{barwj_carleman_estimate} satisfies in $Q_T \cap \{ \varphi_{\omega_j} > s \}$ 
\[
\Box_g \bar{w}_j = -\Box_g(\Psi^* u_{\omega_j,s}').
\]
We wish to add a term involving $\Box_{g'} u_{\omega_j,s}' = 0$ to obtain a quantity involving $g - g'$. However, it turns out that at first we can only determine the metric up to an unknown conformal factor $\kappa$ (this is because the spanning condition \eqref{assumption5} naturally involves trace free tensors). Such a conformal factor needs to be built into the argument.

We take $\kappa$ to be a smooth function (not necessarily positive) that will be fixed later, and write the term on the right of \eqref{barwj_carleman_estimate} as 
\begin{equation} \label{boxg_barwj_estimate_first}
\Box_g \bar{w}_j = - \left[ \Box_g(\Psi^* u_{\omega_j,s}') - \kappa \Psi^*(\Box_{g'} u_{\omega_j,s}') \right].
\end{equation}
We next give a simple result on differences of this type. Below $f(\Psi)$ is a shorthand for the function $z \mapsto f(\Psi(z))$.

\begin{Lemma} \label{lemma_psi_differences}
Let $z = (z^0, \ldots, z^n)$ be global coordinates in $\mR^{1+n}$, and write $g = (g^{jk})$ and $g' = ((g')^{jk})$ in these coordinates. If $v$ is a smooth function, then 
\[
-\Box_g v = g^{jk} \p_{jk} v + \theta^k_g \p_k v
\]
where $\theta^k_g = -\Box_g(z^k)$. If $\Psi$ is a smooth map from $\mR^{1+n}$ to itself and $\varphi$, $\kappa$ are smooth functions on a subset of $\mR^{1+n}$, then 
\begin{align*}
-\Box_g (\Psi^* v)  + \kappa  \Psi^* (  \Box_{g'} v) &= \bar{g}^{ab} \p_{ab} v(\Psi) + \bar{\theta}^k \p_k v(\Psi), \\
Z_{g,\Psi^* \varphi}(\Psi^* v) - \kappa \Psi^*( Z_{g',\varphi} v) &= -\bar{g}^{ab} \p_a \varphi(\Psi) \p_b v(\Psi)
\end{align*}
where 
\begin{align*}
\bar{g}_{z} &= (\Psi_* g)_{\Psi(z)} - \kappa(z) (g')_{\Psi(z)}, \\
\bar{\theta}^k &= -\Box_g (\Psi^* z^k)  + \kappa \Psi^* (\Box_{g'} z^k).
\end{align*}
More explicitly, 
\begin{align*}
\bar{g}^{ab} &= ((D\Psi) g (D\Psi)^t)^{ab} - \kappa (g')^{ab}(\Psi), \\
\bar{\theta}^k &= g^{lm} \p_{lm} \Psi^k + \theta^l_g \p_l \Psi^k -  \kappa \theta_{g'}^k(\Psi).
\end{align*}
\end{Lemma}

\begin{Remark}
Above, $\bar{g}$ is a smooth $(2,0)$-tensor field, but the quantities $\bar{\theta}^k$ depend on the choice of coordinates. One can view $\theta_g^k$, $k=0,\dots,n$, as measuring how far the coordinates $z = (z^0, \dots, z^n)$ are from \emph{wave coordinates} that would satisfy $\Box_g z^k = 0$ (these are the Lorentzian analogue of harmonic coordinates). We call $\bar{\theta} = (\bar{\theta}^0, \ldots, \bar{\theta}^n)$ the \emph{wave gauge defect}. Our objective will be to prove that $\bar{g} = 0$.
\end{Remark}

At this point we fix a global coordinate system $z = (z^0, \ldots, z^n)$ in $\mR^{1+n}$, and $\bar\theta$ is defined with respect to these coordinates. Lemma~\ref{lemma_psi_differences} and \eqref{boxg_barwj_estimate_first} imply that 
\begin{equation} \label{boxgbarwj_estimate}
|\Box_g \bar{w}_j| \lesssim |\bar{g}| + |\bar{\theta}|.
\end{equation}
Here and below, the constants are uniform over $|s| \leq T+1$ due to the smooth dependence in $s$ in Proposition \ref{prop_plane_wave_structure}.

The estimate \eqref{barwj_carleman_estimate} also involves $Z_{g,\varphi_{\omega_j}} \bar{w}_j$ on $\Gamma_{\omega_j}$. 
The gauge fixing property \eqref{phi_omega_j} together with the transport equations for $u$ and $u'$ yield that on $\Gamma_{\omega_j}$ 
\begin{align*}
 &Z_{g,\varphi_{\omega_j}} \bar{w}_j = Z_{g,\varphi_{\omega_j}} u_{\omega_j,s} - Z_{g,\varphi_{\omega_j}} u_{\omega_j,s}'' \\
 &= Z_{g,\varphi_{\omega_j}} u_{\omega_j,s} - Z_{g,\Psi^* \varphi_{\omega_j}'} \Psi^* u_{\omega_j,s}' \\
 &= Z_{g,\varphi_{\omega_j}} u_{\omega_j,s} - \kappa \Psi^* (Z_{g',\varphi_{\omega_j}'} u_{\omega_j,s}') - (Z_{g,\Psi^* \varphi_{\omega_j}'} \Psi^* u_{\omega_j,s}' - \kappa \Psi^* (Z_{g',\varphi_{\omega_j}'} u_{\omega_j,s}')) \\
 &=-\frac{1}{2} (\Box_g \varphi_{\omega_j}) u_{\omega_j,s} + \frac{1}{2} \kappa \Psi^*( (\Box_{g'} \varphi_{\omega_j}') u_{\omega_j,s}' ) \\
 &\qquad  \qquad -(Z_{g,\Psi^* \varphi_{\omega_j}'} \Psi^* u_{\omega_j,s}' - \kappa \Psi^* ( Z_{g',\varphi_{\omega_j}'} u_{\omega_j,s}')) \\
 &=-\frac{1}{2} (\Box_g (\Psi^* \varphi_{\omega_j}') - \kappa \Psi^*( \Box_{g'} \varphi_{\omega_j}')) u_{\omega_j,s} - \frac{1}{2} \kappa \Psi^*( \Box_{g'} \varphi_{\omega_j}') \bar{w}_j \\
 &\qquad  \qquad -(Z_{g,\Psi^* \varphi_{\omega_j}'} \Psi^* u_{\omega_j,s}' - \kappa \Psi^* ( Z_{g',\varphi_{\omega_j}'} u_{\omega_j,s}')).
\end{align*}
In view of Lemma \ref{lemma_psi_differences} and the fact that $u_{\omega_j,s} > 0$ on $\Gamma_{\omega_j}$ (see Proposition \ref{prop_plane_wave_structure}), this gives 
\begin{equation} \label{zbarwj_estimate}
|Z_{g,\varphi_{\omega_j}} \bar{w}_j| \geq c |\bar{\theta}^k \p_k \varphi_{\omega_j}'(\Psi)| - C(|\bar{w}_j| + |\bar{g}|).
\end{equation}

Using \eqref{boxgbarwj_estimate}--\eqref{zbarwj_estimate} in \eqref{barwj_carleman_estimate} and absorbing one term (with $\sigma$ chosen large enough) gives that 
\[
\sigma \norm{e^{\sigma \psi} \bar{\theta}^k \p_k \varphi_{\omega_j}'(\Psi)}_{L^2(\Gamma_{\omega_j})}^2 
 \lesssim \sigma \norm{e^{\sigma \psi} |\bar{g}|}_{L^2(\Gamma_{\omega_j})}^2 + \norm{e^{\sigma \psi} (|\bar{g}| + |\bar{\theta}|)}_{L^2(Q_T \cap \{ \varphi_{\omega_j} > s \})}^2.
\]
Integrating over $|s| \leq T+1$, and using that $\bar{g}$ and $\bar{\theta}$ vanish outside $\ol{Q}_T$, yields 
\[
\sigma \norm{e^{\sigma \psi} \bar{\theta}^k \p_k \varphi_{\omega_j}'(\Psi)}_{L^2(Q_T)}^2 
 \lesssim \sigma \norm{e^{\sigma \psi} |\bar{g}|}_{L^2(Q_T)}^2 + \norm{e^{\sigma \psi} |\bar{\theta}|}_{L^2(Q_T)}^2.
\]
Summing over $0 \leq j \leq n$ and using \eqref{assumption4} implies 
\[
\sigma \norm{e^{\sigma \psi} |\bar{\theta}|}_{L^2(Q_T)}^2 
 \lesssim \sigma \norm{e^{\sigma \psi} |\bar{g}|}_{L^2(Q_T)}^2 + \norm{e^{\sigma \psi} |\bar{\theta}|}_{L^2(Q_T)}^2.
\]
Choosing $\sigma$ large, we obtain the first consequence of the weighted estimate in Proposition~\ref{prop_carleman_psi}.

\begin{Proposition} \label{prop_thetabar_estimate}
Suppose that $g$ and $g'$ satisfy \eqref{assumption1}--\eqref{assumption2} and $g'$ satisfies \eqref{assumption3}--\eqref{assumption4}. If $\mathcal{F}_{\omega_j,T}(g) = \mathcal{F}_{\omega_j,T}(g')$ for $0 \leq j \leq n$, then 
\begin{equation*}
\norm{e^{\sigma \psi} |\bar{\theta}|}_{L^2(Q_T)}
 \lesssim \norm{e^{\sigma \psi} |\bar{g}|}_{L^2(Q_T)}
\end{equation*}
when $\sigma$ is sufficiently large.
\end{Proposition}

In Proposition \ref{prop_thetabar_estimate} we have only used measurements from $n+1$ directions $\omega_j$ and have obtained control on the wave gauge defect $\bar{\theta}$ in terms of $\bar{g}$. The next step is to use a second batch of weighted estimates related to measurements from directions $\omega \in \Omega \setminus \{ \omega_0, \ldots, \omega_n \}$. We will now also assume the spanning condition \eqref{assumption5}. With a suitable choice of $\kappa$, this will imply that a power of $\sigma$ times $\bar{g}$ can be controlled in terms of $\bar{g}$ and $\bar{\theta}$, or just in terms of $\bar g$ in view of Proposition~\ref{prop_thetabar_estimate}. The proof of the following result is based on taking $\sigma$ large enough in such an argument.

\begin{Proposition} \label{prop_gbar_zero}
Suppose that $g$ and $g'$ satisfy \eqref{assumption1}--\eqref{assumption2} and $g'$ satisfies \eqref{assumption3}--\eqref{assumption5}. If $\mathcal{F}_{\Omega,T}(g) = \mathcal{F}_{\Omega,T}(g')$, and if $\kappa$ is chosen so that $\bar{g}$ is $g'$-trace free, that is, 
\[
\kappa(z) = \frac{1}{n+1} \mathrm{tr}_{g'}(\Psi_* g)(\Psi(z)),
\]
then $\bar{g} = 0$.
\end{Proposition}
\begin{proof}
Fix $\omega \in \Omega \setminus \{ \omega_0, \ldots, \omega_n \}$ and write 
\[
\varphi_{\omega}'' = \Psi^* \varphi_{\omega}', \qquad u_{\omega,s}'' = \Psi^* u_{\omega,s}'.
\]
As above, we would like to apply the weighted estimate to the difference of $u_{\omega,s}$ and $u_{\omega,s}''$. However, now $u_{\omega,s}$ and $u_{\omega,s}''$ satisfy transport equations on interfaces $\{ \varphi_{\omega} = s \}$ and $\{ \varphi_{\omega}'' = s \}$, which may be different. Furthermore, the sets $\{ \varphi_{\omega}'' = s \}$ may not even be regular hypersurfaces, since $d \varphi_{\omega}'' = \Psi^* d \varphi_{\omega}'$ may vanish at points where $D\Psi$ vanishes (though $\{ \varphi_{\omega}'' = s \}$ is regular for almost every $s$ by Sard's theorem). To resolve these issues, we will use the properties of the smooth function $u'$ in $\ol{Q}_T$ given in Proposition~\ref{prop_plane_wave_structure}.

From now on we will drop the subscripts $\omega$ and $s$ for simplicity.  Define 
\[
\bar{w} = u - \Psi^* u'.
\]
This is smooth in $\ol{Q}_T \cap \{ \varphi \geq s \}$. Since the boundary measurements from direction $\omega$ agree for $g$ and $g'$ and $\Psi = \id$ outside $\ol{Q}_T$, we have that the Cauchy data of $\bar{w}$ vanishes on $\p Q_T \cap \{ \varphi \geq s \}$. Thus Proposition \ref{prop_carleman_psi} can be applied to $\bar{w}$ in $Q_T \cap \{ \varphi > s \}$, which gives 
\begin{gather} \label{barwomegas_carleman_estimate}
\sigma \norm{e^{\sigma \psi} Z_{g,\varphi} \bar{w}}_{L^2(\Gamma)}^2 + \sigma^3 \norm{e^{\sigma \psi} \bar{w}}_{L^2(\Gamma)}^2 
 \lesssim \norm{e^{\sigma \psi} \Box_g \bar{w}}_{L^2(Q_T \cap \{ \varphi > s \})}^2.
\end{gather}
Here $\Gamma = Q_T \cap \{ \varphi = s \}$.

In $\ol{Q}_T \cap \{ \varphi \geq s \}$, the right hand side of \eqref{barwomegas_carleman_estimate} satisfies 
\[
\Box_g \bar{w} = - \Box_g (\Psi^* u') = - (\Box_g (\Psi^* u') - \kappa \Psi^*( \Box_{g'} u')) - \kappa \Psi^* (\Box_{g'} u').
\]
From Lemma \ref{lemma_psi_differences} we obtain the bound 
\[
|\Box_g (\Psi^* u') - \kappa \Psi^* (\Box_{g'} u')| \lesssim |\bar{g}| + |\bar{\theta}|.
\]
On the other hand, by Proposition \ref{prop_plane_wave_structure} we have 
\[
\Psi^* ( \Box_{g'} u') = (\varphi''-s) \Psi^* a_{\omega,s}'
\]
where $a_{\omega,s}'$ is supported in $\{ \varphi' \leq s \}$. In the set $\{ \varphi \geq s \}$, the function $\Psi^* a_{\omega,s}'$  can only be nonzero if $\varphi'' \leq s \leq \varphi$. Thus in $\{ \varphi \geq s \}$ we have 
\[
| \kappa \Psi^* ( \Box_{g'} u')| \leq |\varphi'' - s| | \kappa \Psi^* a_{\omega,s}'| \lesssim |\varphi - \varphi''|.
\]
We introduce the notation 
\[
\bar{\varphi} := \varphi - \varphi''.
\]
Combining the above estimates leads to the bound 
\begin{equation}  \label{boxgbarwomega_estimate}
|\Box_g \bar{w}| \lesssim |\bar{g}| + |\bar{\theta}| + |\bar{\varphi}|
\end{equation}
in the set $\ol{Q}_T \cap \{ \varphi \geq s \}$.

To estimate the left hand side of \eqref{barwomegas_carleman_estimate}, we compute  
\begin{align*}
Z_{g,\varphi} \bar{w} &= Z_{g,\varphi} u - Z_{g,\varphi} (\Psi^* u') \\
  &= Z_{g,\varphi} u - Z_{g,\varphi''} (\Psi^* u') - (Z_{g,\varphi} - Z_{g,\varphi''}) (\Psi^* u') \\
  &= Z_{g,\varphi} u - \kappa \Psi^* ( Z_{g',\varphi'} u') - (Z_{g,\varphi} - Z_{g,\varphi''}) (\Psi^* u') \\
  & \qquad -  (Z_{g,\Psi^* \varphi'} (\Psi^* u') - \kappa \Psi^* ( Z_{g',\varphi'} u')).
\end{align*}
We combine this with Lemma \ref{lemma_psi_differences} and recall that $Z_{g,\varphi} = -d\varphi^\sharp$ to obtain 
\[
Z_{g,\varphi} \bar{w} = Z_{g,\varphi} u - \kappa \Psi^* ( Z_{g',\varphi'} u') + O(| d \bar{\varphi}| + |\bar{g}|).
\]
On the set $\{ \varphi = s \}$, from the transport equation for $u$ we obtain that $2 Z_{g,\varphi} u = -(\Box_g \varphi) u$. On the other hand, still on $\{ \varphi = s \}$, Proposition \ref{prop_plane_wave_structure} implies that 
\[
\Psi^* (2 Z_{g',\varphi'} u') = -\Psi^* ( (\Box_{g'} \varphi') u') - \bar{\varphi} \Psi^* b_{\omega,s}'.
\]
This yields that on $\{ \varphi = s \}$  
\begin{align*}
Z_{g,\varphi} \bar{w} &= -\frac{1}{2} \left[ (\Box_g \varphi) u - \kappa \Psi^* ( (\Box_{g'} \varphi') u' ) \right] + O(|\bar{\varphi}| + |d \bar{\varphi}| + |\bar{g}|) \\
 &= -\frac{1}{2} (\Box_g \bar{\varphi}) u -\frac{1}{2} \left[ (\Box_g \varphi'') \Psi^* u' - \kappa \Psi^* ( (\Box_{g'} \varphi') u' ) \right] \\
 &\qquad + O(|\bar{w}| + |\bar{\varphi}| + |d \bar{\varphi}| + |\bar{g}|).
\end{align*}
Finally, one more application of Lemma \ref{lemma_psi_differences} gives on $\{ \varphi = s \}$ the estimate 
\begin{equation}  \label{zgbarwomega_estimate}
|Z_{g,\varphi} \bar{w}| \geq c |\Box_g \bar{\varphi}| - C(|\bar{w}| + |\bar{\varphi}| + |d \bar{\varphi}| + |\bar{g}| + |\bar{\theta}|).
\end{equation}
for some $C, c > 0$. Inserting the estimates \eqref{boxgbarwomega_estimate} and \eqref{zgbarwomega_estimate} in \eqref{barwomegas_carleman_estimate} and absorbing one $\bar{w}$ term yields 
\begin{multline*}
\sigma \norm{e^{\sigma \psi} \Box_g \bar{\varphi}}_{L^2(\Gamma)}^2 \lesssim \norm{e^{\sigma \psi} (|\bar{g}| + |\bar{\theta}| + |\bar{\varphi}| )}_{L^2(Q_T \cap \{ \varphi > s \})}^2 \\
 + \sigma \norm{e^{\sigma \psi} (|\bar{\varphi}| + |d \bar{\varphi}| + |\bar{g}| + |\bar{\theta}|)}_{L^2(\Gamma)}^2.
\end{multline*}
Integrating both sides of this estimate from $-T-1$ to $T+1$ with respect to $s$, and using that $\bar{\varphi}$, $\bar{g}$, and $\bar{\theta}$ vanish outside $Q_T$, gives that 
\[
\norm{e^{\sigma \psi} \Box_g \bar{\varphi}}_{L^2(Q_T)}^2 \lesssim \norm{e^{\sigma \psi} (|\bar{\varphi}| + |d \bar{\varphi}| + |\bar{g}| + |\bar{\theta}|)}_{L^2(Q_T)}^2.
\]
For the $\bar{\theta}$ term on the right we can use Proposition \ref{prop_thetabar_estimate}. On the other hand, since $\bar{\varphi}$ vanishes outside $\ol{Q}_T$, on the left hand side we may apply the $L^2$ estimate in Proposition \ref{prop_carleman_psi} on $Q_S$ with $s = -S - 1$ and $S > 0$ large enough so that $Q_T \subset Q_S \cap \{\varphi_\omega > s\}$ for some unit vector $\omega$. This implies, after absorbing terms, that 
\begin{equation} \label{barvarphi_barg_estimate}
\sigma^4 \norm{e^{\sigma \psi} \bar{\varphi}}_{L^2(Q_T)}^2 + \sigma^2 \norm{e^{\sigma \psi} d \bar{\varphi}}_{L^2(Q_T)}^2 \lesssim \norm{e^{\sigma \psi} |\bar{g}|}_{L^2(Q_T)}^2.
\end{equation}

It remains to relate $\bar{g}$ to $\bar{\varphi}$. For this we will use the spanning condition \eqref{assumption5}. Now the subspace of $g'$-trace free tensors is a codimension $1$ subspace of the symmetric tensors, and each $d \varphi_{\omega}' \otimes d \varphi_{\omega}'$ is trace free by the eikonal equation, so \eqref{assumption5}  implies that $\{ d \varphi_{\omega}' \otimes d \varphi_{\omega}' \}_{\omega \in \Omega}$ spans the set of $g'$-trace free symmetric $2$-tensors at any point. Moreover, the choice of $\kappa$ guarantees that $\bar g$ is $g'$-trace free.
Hence \eqref{assumption5} and the eikonal equations for $\varphi_{\omega}$ and $\varphi_{\omega}'$ imply 
\[
|\bar{g}| \lesssim \sum_{\omega \in \Omega} |\bar{g}(d \varphi_{\omega}'(\Psi), d \varphi_{\omega}'(\Psi))| = \sum_{\omega \in \Omega} |g(d \varphi_{\omega}'', d \varphi_{\omega}'')| \lesssim \sum_{\omega \in \Omega} |d \bar{\varphi}_{\omega}|.
\]
Summing \eqref{barvarphi_barg_estimate} over $\omega \in \Omega$ and inserting the previous estimate, we obtain that  
\[
\sigma^2 \norm{e^{\sigma \psi} |\bar{g}|}_{L^2(Q_T)}^2 \lesssim \norm{e^{\sigma \psi} |\bar{g}|}_{L^2(Q_T)}^2.
\]
Choosing $\sigma$ large enough, this finally implies that $\bar{g} = 0$ in $Q_T$.
\end{proof}

The equation $\bar{g} = 0$ means in coordinates that 
\begin{equation} \label{gbar_zero_in_coordinates}
(D\Psi) g (D\Psi)^t = \kappa(g')(\Psi).
\end{equation}
Next we show that $\kappa$ is nowhere vanishing and hence a conformal factor.

\begin{Proposition}
Assume the conditions in Proposition \ref{prop_gbar_zero}. Then $\kappa > 0$ in $\mR^{1+n}$ with $\kappa = 1$ outside $Q_T$, and $\Psi$ is a diffeomorphism from $\mR^{1+n}$ onto itself.
\end{Proposition}
\begin{proof}
If $u$ and $v$ are smooth functions, the equation \eqref{gbar_zero_in_coordinates} implies that 
\[
g_z(d\Psi^* u, d\Psi^* v) = \kappa(z) (g')_{\Psi(z)}(du, dv).
\]
Choosing $u = \varphi_{\omega_0}'$ and $v = \varphi_{\omega_1}'$ and using the gauge fixing property \eqref{phi_omega_j} shows that 
\[
g_z(d \varphi_{\omega_0}, d\varphi_{\omega_1}) = \kappa(z) (g')_{\Psi(z)}(d\varphi_{\omega_0}', d\varphi_{\omega_1}').
\]
Now if $\kappa(z) = 0$ for some $z$, we obtain $g_z(d \varphi_{\omega_0}, d\varphi_{\omega_1}) = 0$, and since these are null covectors in the same causal cone we must have 
\[
d\varphi_{\omega_0}(z) = \lambda d\varphi_{\omega_1}(z)
\]
for some $\lambda > 0$. However, since both $\varphi_{\omega_0}$ and $\varphi_{\omega_1}$ are eikonal solutions with $\varphi_{\omega_j} = t - x \cdot \omega_j$ for $x \cdot \omega_j \leq -1$, it follows that $\omega_0 = \omega_1$ (see \cite[Proposition 4.1]{oksanen2024}). This contradicts \eqref{assumption4}, and therefore $\kappa$ is nowhere vanishing. The formula for $\kappa$ in Proposition \ref{prop_gbar_zero}, together with \eqref{assumption1} and the fact that $\Psi = \id$ outside $Q_T$, implies that $\kappa = 1$ outside $Q_T$. Since $\mR^{1+n}$ is connected, we see that $\kappa > 0$ in $\mR^{1+n}$.

Now taking determinants in \eqref{gbar_zero_in_coordinates} and using that $g$ and $g'$ are nondegenerate and $\kappa$ is nowhere vanishing implies that $\det(D\Psi)$ is never zero. Hence $\Psi$ is a local diffeomorphism with $\Psi = \id$ outside $\ol{Q}_T$. Hadamard's global inverse function theorem \cite[Theorem 6.2.8]{krantz2002} implies that $\Psi$ is a diffeomorphism from $\mR^{1+n}$ to itself.
\end{proof}

Now we know that $\Psi$ is indeed a global diffeomorphism and $\kappa$ is positive. We can return to considering $g$ and $g'$ as $(0,2)$-tensors (i.e.\ with indices down). Then \eqref{gbar_zero_in_coordinates} may be rewritten as 
    \begin{align}\label{eq_kappa_Psi}
\kappa g = \Psi^* g'.
    \end{align}
Thus we have proved that $g$ and $g'$ agree up to a conformal transformation that is identity outside $\ol{Q}_T$.

It remains to show that $\kappa \equiv 1$. Writing $c = \kappa^{-1} \circ \Psi^{-1}$ and using that $\Psi = \id$ outside $Q_T$, for any $\omega \in \Omega$ we have 
\[
\mathcal{F}_{\omega,T}(c g') = \mathcal{F}_{\omega,T}(\Psi^*(c g')) = \mathcal{F}_{\omega,T}(\kappa^{-1} \Psi^* g') = \mathcal{F}_{\omega,T}(g).
\]
Since we assumed that $\mathcal{F}_{\omega,T}(g) = \mathcal{F}_{\omega,T}(g')$, this gives 
\[
\mathcal{F}_{\omega,T}(c g') = \mathcal{F}_{\omega,T}(g').
\]
The proof of Theorem \ref{thm_main2} is now concluded by the following result, which recovers an unknown conformal factor from $n+1$ measurements.

\begin{Proposition}\label{prop_conformal_factor}
Let $g'$ be a smooth Lorentzian metric in $\mR^{1+n}$, $n \ge 2$, satisfying \eqref{assumption1}--\eqref{assumption2}, and suppose that \eqref{assumption3}--\eqref{assumption4} hold for $\omega_0, \ldots, \omega_n \in S^{n-1}$. Let $c$ be a smooth positive function in $\mR^{1+n}$ with $c = 1$ outside $Q_T$, and suppose that \eqref{assumption2} holds for $c g'$. If 
\[
\mathcal{F}_{\omega_j,T}(c g') = \mathcal{F}_{\omega_j,T}(g'), \quad 0 \leq j \leq n,
\]
then $c = 1$.
\end{Proposition}
\begin{proof}
Write $g = c g'$. Now $g$ satisfies \eqref{assumption1}--\eqref{assumption2}, and by Proposition \ref{prop_simplicity_detection} also $g$ is simple in each direction $\omega_j$. By Proposition \ref{prop_plane_wave_structure} there are smooth eikonal solutions $\varphi = \varphi_{\omega_j}$ and $\varphi' = \varphi'_{\omega_j}$, and since $g = c g'$ we have $\varphi = \varphi'$. Again by Proposition \ref{prop_plane_wave_structure} the plane wave solutions for $g$ and $g'$ have the form 
\[
U = u_s H(\varphi-s), \qquad U' = u_s' H(\varphi-s),
\]
respectively, where $u_s$ and $u'_s$ are smooth in $\{ \varphi \geq s \}$. On $\Gamma = \{ \varphi = s \}$, the functions $u_s$ and $u_s'$ satisfy transport equations 
\begin{align*}
2 Z_{g,\varphi} u_s + (\Box_g \varphi) u_s &= 0, \\
2 Z_{g',\varphi} u_s' + (\Box_{g'} \varphi) u_s' &= 0.
\end{align*}
Proposition \ref{prop_eikonal_detection} ensures that $U_s = U_s'$ slightly outside $Q_T$.

Let $\bar{w} = u_s - u_s'$. Applying Proposition \ref{prop_carleman_psi} to $\bar{w}$ yields, for $\sigma$ sufficiently large,  
\begin{equation} \label{carleman_conformal_factor}
\sigma \norm{e^{\sigma \psi} Z_{g,\varphi} \bar{w}}_{L^2(\Gamma)}^2 + \sigma^3 \norm{e^{\sigma \psi} \bar{w}}_{L^2(\Gamma)}^2 \lesssim \norm{e^{\sigma \psi} \Box_g \bar{w}}_{L^2(Q_T \cap \{ \varphi > s \})}^2.
\end{equation}
We write the right hand side of \eqref{carleman_conformal_factor} in $Q_T \cap \{ \varphi > s \}$ as 
\[
\Box_g \bar{w} = -\Box_g u_s' = (-\Box_g + c^{-1} \Box_{g'}) u_s'.
\]
Since $g = c g'$ and $-\Box_g = g^{jk} \p_{jk} + \theta_g^k \p_k$, for any smooth $v$ we have 
\[
(-\Box_g + c^{-1} \Box_{g'}) v = (\theta_g^k - c^{-1} \theta_{g'}^k) \p_k v.
\]
Now $\theta_g^k = |g|^{-1/2} \p_j(|g|^{1/2} g^{jk}) = c^{-1} \theta_{g'}^k + c^{-\frac{1+n}{2}} \p_j(c^{\frac{1+n}{2}-1}) (g')^{jk}$, so 
\begin{equation} \label{boxg_cinvboxgprime}
(-\Box_g + c^{-1} \Box_{g'}) v  = \frac{n-1}{2} c^{-2} g'(dc, dv).
\end{equation}
In particular 
\[
\Box_g \bar{w} = \frac{1-n}{2} g'(du_s', d(c^{-1})).
\]

On the other hand, for the left hand side of \eqref{carleman_conformal_factor}, since $Z_{g,\varphi} = c^{-1} Z_{g',\varphi}$ we have 
\begin{align*}
Z_{g,\varphi} \bar{w} &= Z_{g,\varphi} u_s - c^{-1} Z_{g',\varphi} u_s' \\
 &= -\frac{1}{2} ( (\Box_g \varphi) u_s - c^{-1} (\Box_{g'} \varphi) u_s' ) \\
 &= \frac{1}{2} ( -\Box_g \varphi + c^{-1} \Box_{g'} \varphi) u_s + O(|\bar{w}|).
\end{align*}
From \eqref{boxg_cinvboxgprime} we obtain 
\[
Z_{g,\varphi} \bar{w} = \frac{1-n}{4} g'(d\varphi, d(c^{-1})) u_s + O(|\bar{w}|).
\]
Inserting these estimates in \eqref{carleman_conformal_factor} and absorbing one term, with $\sigma$ large enough, we obtain  
\[
\sigma \norm{e^{\sigma \psi} g'(d\varphi, d(c^{-1}))}_{L^2(\Gamma)}^2 \lesssim \norm{e^{\sigma \psi} d(c^{-1})}_{L^2(Q_T \cap \{ \varphi > s \})}^2.
\]
Integrating from $-T-1$ to $T+1$ with respect to $s$, this becomes 
\[
\sigma \norm{e^{\sigma \psi} g'(d\varphi, d(c^{-1}))}_{L^2(Q_T)}^2 \lesssim \norm{e^{\sigma \psi} d(c^{-1})}_{L^2(Q_T)}^2.
\]

The previous estimate is true for $\varphi = \varphi_{\omega_j}$. Summing over $0 \leq j \leq n$ and applying the diffeomorphism property \eqref{assumption4}, so that $\{ d \varphi_{\omega_0}, \ldots, d \varphi_{\omega_n} \}$ is a basis at each point, implies that 
\[
\sigma \norm{e^{\sigma \psi} d(c^{-1})}_{L^2(Q_T)}^2 \lesssim \norm{e^{\sigma \psi} d(c^{-1})}_{L^2(Q_T)}^2.
\]
Choosing $\sigma$ large enough gives $d(c^{-1}) = 0$, so $c$ is constant. Since $c = 1$ outside $Q_T$, we obtain $c \equiv 1$ as required.
\end{proof}

\section{Directional simplicity} \label{sec_directional_simplicity}

In this section we prove Proposition \ref{prop_simplicity_detection}.
Let $g$ be a smooth Lorentzian metric  in $\mR^{1+n}$ satisfying \eqref{assumption1}.
We begin by formulating the directional simplicity condition in more detail.
Fix $\omega \in S^{n-1}$ and write 
    \begin{align}\label{def_Sigma_minus}
\Sigma_- = \{ (t,x) \in \R^{1+n} \,:\, x \cdot \omega = - 1 \},
    \end{align}
Given $z \in \Sigma_-$, let $\gamma_z(r) = \gamma(r; z, (1, \omega))$ be the geodesic satisfying 
    \begin{align*}
\gamma_z(0) = z, \quad \dot{\gamma}_z(0) = (1,\omega)
    \end{align*}
that is defined in the maximal interval $(-\infty,\rho^g(z))$ where $\rho^g(z) \in (0,\infty]$. Occasionally, we write $\gamma_z = \gamma_z^g$ to indicate the choice of the metric. Furthermore, we define
    \begin{align}\label{Phig}
\Phi^g : \{(z,r) \in \Sigma_- \times \R \mid r < \rho^g(z) \}
\to \R^{1+n}, \quad \Phi^g(z, r) = \gamma_z^g(r).
    \end{align}

\begin{Definition}\label{simple_in_omega}
A smooth Lorentzian metric $g$ in $\mR^{1+n}$ satisfying \eqref{assumption1} is simple in direction $\omega \in S^{n-1}$ if there is an open set $U \subset \Sigma_- \times \R$ and a neighborhood $V$ of $\overline{Q_{T+2}}$ such that $\Phi^g$ is a diffeomorphism from $U$ to $V$. 
\end{Definition}

We write $\Phi^g = \Phi^g_\omega$ and $\Sigma_- = \Sigma_-^\omega$ when emphasizing the dependence on the direction $\omega$. Our first lemma shows that $\Phi^g$ is a diffeomorphism on a larger set. We write $B_r = \{x \in \R^n \mid |x| < r\}$.

\begin{figure}
\begin{tikzpicture}[line cap=round, line join=round]
 
  \def\xL{1}          
  \def\R{2}           
  \def\cy{3}        
  \pgfmathsetmacro{\cx}{\xL+\R}       
  \def\xmin{0}  \def\xmax{5.5}  \def\ymin{0}  \def\ymax{6}
  \pgfmathsetmacro{\ytop}{\cy+\R}     
  \pgfmathsetmacro{\ybot}{\cy-\R}     
  \pgfmathsetmacro{\xout}{\xmax+1}    
 
  \begin{scope}
    \clip (\xmin,\ymin) rectangle (\xmax,\ymax);
    \fill[region, even odd rule]
        (\xmin-1,\ymin-1) rectangle (\xout,\ymax+1)
        (\cx,\ytop) arc[start angle=90, end angle=270, radius=\R]
                    -- (\xout,\ybot) -- (\xout,\ytop) -- cycle;
  \end{scope}
 
  \draw[pen, line width=1pt] (\cx,\cy) circle (\R);
  \draw[pen, line width=1pt] (\xL,\ymin) -- (\xL,\ymax);
 
  \node[pen, anchor=south west, font=\large] at (\xL,\ymax+0.1) {$\Sigma_-$};
  \node[pen, font=\large] at (\cx,\cy) {$\mathbb R\times B$};
\end{tikzpicture}
\caption{The cylinder $\R \times B$ and $\Sigma_-$ projected in space $\R^n$. The shaded region is foliated by the geodesics $\gamma_z$, $z \in \Sigma_-$. The geodesics are Minkowski geodesics in this region.}
\label{fig_shadow}
\end{figure}
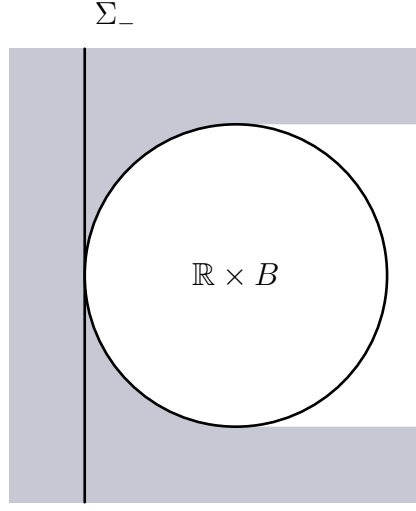

\begin{Lemma}\label{lem_simple_ext}
Suppose that $g$ satisfying \eqref{assumption1}--\eqref{assumption2} is simple in direction $\omega \in S^{n-1}$. Then
there are open sets $\tilde U \subset \Sigma_- \times \R$ and $\tilde V \subset \R^{1+n}$ and a constant $\epsilon > 0$ such that 
\begin{itemize}
\item[(i)] $\Phi^g : \tilde U \to \tilde V$ is a diffeomorphism,
\item[(ii)] $\R \times B_{1+\epsilon} \subset \tilde V$,
\item[(iii)] $\Sigma_- \times \{0\} \subset \tilde U$,
\item[(iv)] $(z, r) \in \tilde U$ implies $(z,s) \in \tilde U$ for all $s < r$,
\item[(v)] $(\Sigma_- \setminus \pi((\Phi^g)^{-1}(\ol Q_{T+2}))) \times \R \subset \tilde U$ where $\pi : \Sigma_- \times \R \to \Sigma_-$ is the natural projection.
\end{itemize}
Furthermore, for all $z \in \Sigma_-$, the inverse image $(\gamma_z)^{-1}(\R \times B_{1+\epsilon})$ is a bounded interval, if it is nonempty, and $\gamma_z$ extends beyond the interval on both ends. 
\end{Lemma}
\begin{proof}
Compactness of $\overline{Q_{T+2}}$ implies that there is $\epsilon > 0$ such that
    \begin{align*}
V_0 = \{(t, x) \in \R^{1+n} \mid |t| < T+2 + \epsilon,\ |x| < 1 + \epsilon \}
    \end{align*}
is contained in $V$. We write $U_0 = (\Phi^g)^{-1}(V_0)$.
The geodesic $\gamma_z$ with $z \in \Sigma_-$ is a Minkowski null geodesic as long as it stays outside $V_0$. In particular, 
$(\gamma_z)^{-1}(\R \times B_{1+\epsilon})$ is an interval if it is nonempty. 
By \eqref{assumption2} there is a Cauchy temporal function $\tau$ on $(\mR^{1+n}, g)$, see e.g \cite[Proposition 2.3]{oksanen2024}. Compactness of $\overline{V_0}$ implies that $\tau$ attains its minimum and maximum values on $\overline{V_0}$. On the other hand, $\tau$ is strictly increasing and takes every value in $\R$ along $\gamma_z$. Thus 
$(\gamma_z)^{-1}(\overline{V_0})$ is bounded and $\gamma_z$ extends beyond it on both ends.

The geodesic $\gamma_z$ with $z \in \{t > T\} \cap \Sigma_-$ is a Minkowski null geodesic. (In Figure \ref{fig_intervals} such geodesics stay above the upper red line segment.) The directional simplicity implies that $\Phi^g$ is a diffeomorphism from $U_1$ to $V_1$ where
    \begin{align*}
U_1 = ((\{t > T\} \cap \Sigma_-) \times \R) \cup U_0,
\quad
V_1 = \{t - x \cdot \omega> T + 1\} \cup V_0.
    \end{align*}
Analogously $\Phi^g$ is a diffeomorphism from $U_2$ to $V_2$ where
    \begin{align*}
U_2 = ((\{t < -T - 2\} \cap \Sigma_-) \times \R) \cup U_1,
\quad
V_2 = \{t - x \cdot \omega < -T - 1\} \cup V_1.
    \end{align*}
Let us show that 
    \begin{align}\label{V2_bound}
\R \times \overline B \subset V_2.
    \end{align}
Let $x \in \overline B$ and $t \in \R$. If $|t| \le T + 2$ then 
    \begin{align*}
(t,x) \in \overline{Q_{T+2}} \subset V_0 \subset V_2.
    \end{align*}
On the other hand, if $t > T + 2$ then $t - x \cdot \omega > T + 1$, and if $t < -T - 2$ then $t - x \cdot \omega < -T - 1$. Thus \eqref{V2_bound} holds.

Consider the projection $p(x) = x - (x \cdot \omega) \omega$ and define the region
    \begin{align*}
V_3 = \{ |x| > 1 \} \cap (\{ x \cdot \omega < 0 \} \cup \{|p(x)| > 1\}),
    \end{align*}
visualized in Figure \ref{fig_shadow}. Write $U_3 = (\Phi^g)^{-1}(V_3)$, $\tilde U = U_2 \cup U_3$, and $\tilde V = V_2 \cup V_3$.
The geodesic $\gamma_z$ with $z \in \Sigma_-$ is a Minkowski null geodesic as long as it stays in $V_3$, and the directional simplicity implies that $\Phi^g$ is a diffeomorphism from $\tilde U$ to $\tilde V$. Now (ii) holds due to \eqref{V2_bound} and (iii) follows from the choice of $V_3$ and $U_3$. 

Write $I = (\gamma_z)^{-1}(V_0)$ and recall that $I$ is a bounded interval if nonempty. We have $\gamma_z(s) \in V_3$ for $s \le \inf I$ whenever $I$ is nonempty. On the other hand, when $I$ is empty, $\gamma_z$ is a Minkowski geodesic contained in $\tilde V$. We conclude that (iv) holds using these two facts. 

Consider $z = (t,x) \in \Sigma_-$ such that the geodesic $\gamma_z$ does not intersect $\ol Q_{T+2}$. If $|p(x)| > 1$ then $\gamma_z$ is a Minkowski geodesic contained in $V_3$. If $|p(x)| \le 1$ and $t \ge 0$ then $t > T + 1$ and $\gamma_z$ is a Minkowski geodesic contained in $V_1$. Finally, if $|p(x)| \le 1$ and $t \le 0$ then $t < -T - 3$ and $\gamma_z$ is a Minkowski geodesic contained in $V_2$. Thus (v) holds.

We have already shown that $(\gamma_z)^{-1}(\overline{V_0})$ is a bounded interval, if nonempty, and that $\gamma_z$ extends beyond it on both ends.
If $\gamma_z$ does not intersect $\overline{V_0}$ then $\gamma_z$ is a Minkowski geodesic, and it is straightforward to see that the preimage $(\gamma_z)^{-1}(\R \times B_{1+\epsilon})$ has also the above properties.
\end{proof}

\begin{Lemma}\label{lem_r_plus}
Suppose that $g$ satisfying \eqref{assumption1}--\eqref{assumption2} is simple in direction $\omega \in S^{n-1}$. Then there is a smooth function $r_+ : \Sigma_- \to (0, \infty)$ and a constant $\eps > 0$ such that $r_+(z) = 2 + \eps$ when $|z|$ is large and, writing
    \begin{align*}
\mho_\delta &= \{(z, r) \mid r \in (-\infty, r_+(z) + \delta),\ z \in \Sigma_- \}, \quad \delta \in \R, 
    \end{align*}
it holds that $\R \times B \subset \Phi^g(\mho_{-\eps})$ and that
$\Phi^g : \mho_\eps \to \Phi^g(\mho_\eps)$ is a diffeomorphism.
\end{Lemma}

The proof uses the following elementary lemma, shown in Appendix \ref{sec_diffeo_appendix}. 

\begin{Lemma}\label{lem_sc}
Let $e_n = (0, \ldots, 0, 1) \in \mathbb{R}^n$.
For a set $S \subset \mathbb{R}^n$ and a point $x \in \mathbb{R}^n$, define
\[
f_S : \mathbb{R}^n \to \mathbb{R} \cup \{-\infty, +\infty\},
\quad
f_S(x) \;=\; \sup \{\, t \in \mathbb{R} \mid x + t\,e_n \in S \,\},
\]
with the convention $\sup \emptyset = -\infty$.
If $K \subset \mathbb{R}^n$ is compact, then $f_K$ is upper semicontinuous. Moreover, if $U \subset \mathbb{R}^n$ is open, then $f_U$ is lower semicontinuous.
\end{Lemma}
\begin{proof}[Proof of Lemma \ref{lem_r_plus}.]
Let $\tilde U \subset \Sigma_- \times \R$ and $\pi : \Sigma_- \times \R \to \Sigma_-$ be as in Lemma~\ref{lem_simple_ext}.
We define $C = \pi((\Phi^g)^{-1}(\ol Q_{T+2}))$ and
    \begin{align*}
m &: \Sigma_- \to (0,\infty), 
\quad
m(z) = \sup \{ r \in (0, \infty) \mid (z, r) \in \Phi^{-1}(\overline{Q_{T+4}})\},
\\
R &: \Sigma_- \to (0,\infty], \quad R(z) = \sup \{ r \in (0, \infty) \mid (z, r
) \in \tilde U \}.
    \end{align*}
Then $m(z) < R(z)$ for all $z \in C$ due to Lemma~\ref{lem_simple_ext}. 
Writing 
    \begin{align*}
\iota : \Sigma_- \to \Sigma_- \times \R,
\quad \iota(z) = (z, 0),
    \end{align*}
we have $R = f_{\tilde U} \circ \iota$. 
Hence $R$ is lower semicontinuous by Lemma \ref{lem_sc}. Similarly, $m$ is upper semicontinuous.
Therefore $R - m$ is lower semicontinuous and strictly positive on the compact set $C$, and there is a constant $\eps > 0$ such that $R - m > 2\eps$ on $C$. For each $z \in C$ there is $r_z$ with 
    \begin{align*}
m(z) + \eps < r_z < R(z) - \eps.
    \end{align*}
Semicontinuity implies that there is a neighborhood $U_z \subset \Sigma_-$ of $z$ such that $m(w) + \eps < r_z < R(w) - \eps$ for $w \in U_z$. Choose a finite subcover $U_{z_1}, \dots, U_{z_N}$ of $C$ and a smooth partition of unity $\chi_j \in C^\infty(\Sigma_-)$ subordinate to this cover. Then the function $r_+ = \sum_{j=1}^N r_{z_j} \chi_j$ is smooth on $\Sigma_-$ and satisfies $m + \eps < r_+ < R - \eps$ in $C$. By Lemma~\ref{lem_simple_ext} there is room to modify $r_+$ on $\Sigma_- \setminus C$ so that it satisfies $\R \times \overline{B} \subset \Phi^g(\mho_{-\eps})$ and that $r_+(z) = 2 + \eps$ when $|z|$ is large.
\end{proof}

When $g$ and $r_+$ are as in Lemma~\ref{lem_r_plus}, we define 
    \begin{align*}
\Sigma_+^g &= \{ \Phi^g(z, r_+(z)) \mid z \in \Sigma_- \}.
    \end{align*}

\begin{Remark}\label{rem_alpha_diffeo}
If $g$ satisfying \eqref{assumption1}--\eqref{assumption2} is simple in direction $\omega$, then in the coordinates given by $\Phi^g$, the geodesic $\gamma_z$, $z \in \Sigma_-$, is the curve $r \mapsto (z, r)$. In particular, $\gamma_z$ is transverse to $\Sigma_+^g$, and the map $\alpha(z) = \gamma_z(r_+(z))$ is a diffeomorphism from $\Sigma_-$ to $\Sigma_+^g$.
\end{Remark}

For later use, we define 
    \begin{align*}
\Sigma_{-,\sigma} = \Sigma_- \cap \{t = \sigma\}, \quad \sigma \in \R,
    \end{align*}
and when $g$ simple in direction $\omega$ and $\alpha$ is as in Remark \ref{rem_alpha_diffeo},
    \begin{align}\label{def_Sigma_plus_sigma}
\Sigma_{+,\sigma}^g = \alpha(\Sigma_{-,\sigma}).
    \end{align}
Then $\Sigma_{+,\sigma}^g$, $\sigma \in \R$, is a family of disjoint smooth manifolds of codimension two that foliates the smooth manifold $\Sigma_+^g$ of codimension one.

\begin{Definition} \label{def_spw_scattering_relation}
Suppose that $g'$ satisfying \eqref{assumption1}--\eqref{assumption2} is simple in direction $\omega \in S^{n-1}$.
Let $g$ be a Lorentzian metric on $\R^{1 + n}$.
For any $z \in \Sigma_-$ let 
    \begin{align}\label{def_r_g}
r_+^{g}(z) = \inf \,\{ 0 < r < \rho^g(z) \,:\, \gamma_z^{g}(r) \in \Sigma_+^{g'} \}
    \end{align}
with the convention that $r_+^{g}(z) = \rho^g(z)$ if $\gamma_z^{g}(r) \notin \Sigma_+^{g'}$ for all $0 < r < \rho^g(z)$. Let $\mathcal{T}_{g} = \{ z \in \Sigma_- \,:\, r_+^{g}(z) = \rho^{g}(z) \}$ be the trapped set. Given any function $\lambda : \Sigma_+^{g'} \to \R$, the \emph{single plane wave (SPW) scattering relation} is the map 
\[
\beta_{g}^{\lambda}: \Sigma_- \setminus \mathcal{T}_{g} \to T \R^{1+n}, \ z \mapsto (p, \lambda(p) q)
\]
where $(p, q) = (\gamma_z^{g}(r), \dot \gamma_z^{g}(r))$ is evaluated at $r = r_+^{g}(z)$.
We write $\beta_{g} = \beta_{g}^1$ when $\lambda \equiv 1$.
\end{Definition}

Lemma \ref{lem_simple_ext} implies that $\mathcal{T}_{g'} = \emptyset$. The following variant of \cite[Theorem 1.6]{oksanen2024} can be proven using the same strategy as in \cite{oksanen2024}. For the convenience of the reader we give a detailed proof in Appendix \ref{sec_diffeo_appendix}.

\begin{Theorem}\label{th_diffeo}
Let two Lorentzian metrics $g$ and $g'$ on $\R^{1+n}$ both satisfy \eqref{assumption1} and \eqref{assumption2}. 
Suppose that $g'$ is simple in direction $\omega \in S^{n-1}$ and that there is positive $\lambda \in C^\infty(\Sigma_+^{g'})$ such that for all $\sigma \in \R$
    \begin{align}\label{eq_beta}
\beta_{g}(\Sigma_{-,\sigma} \setminus \mathcal T_{g})
= 
\beta_{g'}^{\lambda}(\Sigma_{-,\sigma}),
    \end{align}
where $\Sigma_{-,\sigma} = \Sigma_- \cap \{t = \sigma\}$.
Then $g$ is simple in direction $\omega$.
\end{Theorem}

\begin{Lemma}\label{lem_nontrapping}
Let two Lorentzian metrics $g$ and $g'$ on $\R^{1+n}$ both satisfy \eqref{assumption1} and \eqref{assumption2}. 
Suppose that $g'$ is simple in direction $\omega \in S^{n-1}$ and that there is $\lambda : \Sigma_+^{g'} \to \R$ such that \eqref{eq_beta} holds for all $\sigma \in \R$.
Then 
$\mathcal{T}_{g} = \emptyset$, $\gamma_z^{g}$ intersects $\Sigma_+^{g'}$ transversally for all $z \in \Sigma_-$,
$r_+^{g}$ is smooth on $\Sigma_-$, the map
taking $z$ to $\gamma_z^{g}(r_+^{g}(z))$ is a diffeomorphism from $\Sigma_{-}$ to $\Sigma_{+}^{g'}$, and $\lambda$ is smooth and positive.
\end{Lemma}
\begin{proof}
Since $g$ and $g'$ are globally hyperbolic, they are time-orientable, and we fix their time-orientations so that $dt$ is future-directed outside $M$ for both the metrics.
Let $w \in \Sigma_+^{g'}$.
Using the directional simplicity of $g'$, Remark \ref{rem_alpha_diffeo} implies that
    \begin{align}\label{def_line_bundle}
\{(\gamma_z^{g'}(r), \lambda_z \dot \gamma_z^{g'}(r)) \mid r=r_+^{g'}(z),\ \lambda_z \ne 0, \ z \in \Sigma_{-}\}
    \end{align}
is a line bundle over $\Sigma_{+}^{g'}$.
In particular, there is unique $z' \in \Sigma_-$ such that $w = \gamma_{z'}^{g'}(r_+^{g'}(z'))$. It follows from \eqref{eq_beta} that there
are $z \in \Sigma_{-}$, $s > 0$ and $\lambda \in \R$ such that $\gamma_z^{g}(s) = w$ and 
    \begin{align}\label{def_lambda}
\dot \gamma_z^{g}(s) = \lambda \xi, \quad \xi = \dot \gamma_z^{g'}(r_+^{g'}(z')).
    \end{align}
The geodesics $\gamma_z^{g}$ and $\gamma_{z'}^{g'}$ are future pointing at $z$ and $z'$, respectively. Thus they are future pointing everywhere, and therefore $\lambda > 0$. 

Consider the largest number $\ell \in (0, s]$ such that $\gamma_z^{g}(r)$ is outside the cylinder $\R \times B$ for $r \in (s-\ell, s)$. Then $\gamma_z^{g}|_{(s-\ell, s)}$ coincides with a segment of $\gamma_{z'}^{g'}$, up to reparametrization. This segment of $\gamma_{z'}^{g'}$ does not intersect $\Sigma_+^{g'}$ because of Lemma \ref{lem_r_plus}. Moreover, if $\gamma_z^{g}$ enters into the cylinder $\R \times B$, then it can not intersect $\Sigma_+^{g'}$ before it exits. Since $B$ is convex, it follows from \eqref{assumption1} that if $\gamma_z^g$ exits the cylinder then it does not re-enter. We conclude that $s = r_+^{g}(z)$.

The geodesics through $(w, \lambda \xi)$, for $\lambda \ne 0$, coincide up to reparametrization. This implies that $z \in \Sigma_{-}$ is unique, and we write $\kappa(w) = z$. Then $w = \alpha(\kappa(w))$ where $\alpha(z) = \gamma_z^{g}(r_+^{g}(z))$. By \eqref{def_lambda} and Remark \ref{rem_alpha_diffeo} the geodesic $\gamma_z^{g}$ is transverse to $\Sigma_+^{g'}$ for $z \in \kappa(\Sigma_+^{g'})$, and therefore $r_+^{g}$ is smooth near any $z \in \kappa(\Sigma_+^{g'})$. Hence $\alpha$ is smooth in $\kappa(\Sigma_+^{g'}) \subset \Sigma_-$. 

Write $|\cdot|$ for the Euclidean norm and $\alpha'(z) = \gamma_z^{g'}(r_+^{g'}(z))$. Let us show that $\kappa$ is smooth. 
Using $\lambda > 0$, the function 
    \begin{align*}
\xi(w) 
= 
\frac{\dot \gamma_{\kappa(w)}^{g}(r_+^{g}(\kappa(w)))}
{|\dot \gamma_{\kappa(w)}^{g}(r_+^{g}(\kappa(w)))|} 
= 
\frac{\dot \gamma^{g'}_{z'}(r_+(z'))}
{|\dot \gamma^{g'}_{z'}(r_+(z'))|},
\quad 
z' = (\alpha')^{-1}(w),
    \end{align*}
is smooth. Define
    \begin{align*}
r_-(z) = \inf \,\{ r > 0 \,:\, \gamma^{g}(-r; w, \xi(w)) \in \Sigma_- \}.
    \end{align*}
Then
    \begin{align*}
\kappa(w) = \gamma^{g}(-r_-(w); w,\xi(w)).
    \end{align*}
Moreover, $\gamma^{g}(\cdot; w,\xi(w))$ is transverse to $\Sigma_-$ since
    \begin{align*}
\dot \gamma(-r_-(w); w,\xi(w))
=
\tilde \lambda \dot \gamma_{\kappa(w)}^{g}(0)
    \end{align*}
for some $\tilde \lambda > 0$.
The transversality implies that $r_-$ is smooth. Therefore $\kappa$ is smooth.

Since $\alpha$ is a smooth left inverse of $\kappa$, $\kappa$ is a local diffeomorphism. In the coordinates given by $\Phi^{g'}$, we can identify $z$ and $(z, 0)$. Then it holds for $z \in \Sigma_-$ with large $|z|$ that $\kappa(z, 2) = z$. Hence $\kappa$ is proper. It follows from Hadamard's global inverse function theorem that $\kappa$ is a diffeomorphism from $\Sigma_+^{g'}$ to $\Sigma_-$. This again implies that $\mathcal T_{g} = \emptyset$ and that $\alpha$ is a diffeomorphism from $\Sigma_-$ to $\Sigma_+^{g'}$.

Finally, the smoothness of $\lambda$ follows from 
    \begin{align*}
\lambda(w) = |\lambda(w)| = \frac{|\dot \gamma_{z'}^{g'}(r_+^{g'}(z'))|}{|\dot \gamma_z^g(r_+^{g}(z))|},
    \end{align*}
where $z = \alpha^{-1}(w)$ and $z' = (\alpha')^{-1}(w)$.
\end{proof}

Next we move to boundary measurements of plane waves. First we show that the plane waves are well defined.

\begin{Lemma} \label{lemma_U_well_defined}
Let $g$ be a Lorentzian metric in $\R^{1+n}$ satisfying \eqref{assumption1}--\eqref{assumption2}. Given any $t_0 < -T$, there is a unique $U \in L^2_{\mathrm{loc}}(\mR^{1+n})$ satisfying 
\[
\Box_g U = 0 \text{ in $\mR^{1+n}$}, \qquad U = H(t - x \cdot \omega - s) \text{ in $\{ t < -t_0 \}$}.
\]
\end{Lemma}
\begin{proof}
By \eqref{assumption2} there is a Cauchy temporal function $\tau$ in $(\mR^{1+n}, g)$. Using \cite[Proposition 2.6]{oksanen2024} there is $r_- \in \mR$ such that 
\[
\{ \tau \leq r_- \} \subset \{ t < t_0 + \max(|x|-1, 0) \}.
\]
Let $U_1$ be the unique $L^2_{\mathrm{loc}}$ solution of $\Box_g U_1 = 0$ in $\mR^{1+n}$ with $U_1|_{\{\tau < r_-\}} = H(t - x \cdot \omega - s)$ as in \cite[Proposition 4.1]{oksanen2024}. By \eqref{assumption1}, the function $V = U_1 - H(t - x \cdot \omega - s)$ solves 
\[
\Box_g V = f \text{ in $\mR^{1+n}$}, \qquad V|_{\{ \tau < r_- \}} = 0.
\]
Finite speed of propagation \cite[Proposition 2.5]{oksanen2024} yields $\supp(V) \subset J_+(\supp(f)) \subset J_+(\ol{Q}_T)$, and thus $V = 0$ in $\{ t < t_0 + \max(|x|-1,0) \}$. In particular 
\[
\Box_g U_1 = 0 \text{ in $\mR^{1+n}$}, \qquad U_1 = H(t - x \cdot \omega - s) \text{ in $\{ t < -t_0 \}$}.
\]
If $U_2$ is another solution of the above problem, then $\Box_g(U_1-U_2) = 0$ in $\mR^{1+n}$ and $U_1-U_2=0$ in $\{ t < -t_0 \}$. By \eqref{assumption1} we have $U_1-U_2=0$ in $\{ t < t_0 + \max(|x|-1, 0) \}$ and thus also in $\{ \tau \leq r_- \}$. Uniqueness in the Cauchy problem implies that $U_1-U_2=0$ in $\mR^{1+n}$.
\end{proof}

\begin{Lemma}\label{lem_from_F_to_U}
Let two Lorentzian metrics $g$ and $g'$ on $\R^{1+n}$ both satisfy \eqref{assumption1} and \eqref{assumption2}. Let $\omega \in S^{n-1}$ and suppose that 
$\mathcal{F}_{\omega,T}(g) = \mathcal{F}_{\omega,T}(g')$.
Then $U_{\omega, s}^{g} = U_{\omega, s}^{g'}$ in $\R^{1+n} \setminus \ol{Q}_T$ for $s \in \R$.
\end{Lemma}
\begin{proof}
Observe that $U_{\omega, s}^{\gm}$ is a constant in $Q_T$ when the delay $s$ satisfies $|s| \ge T + 1$. It follows from \eqref{assumption1} that $\Box_g U_{\omega, s}^{\gm} = 0$ for $|s| > T + 1$, and therefore $U_{\omega, s}^{g} = U_{\omega, s}^{\gm}$ for these delays. The same argument shows $U_{\omega, s}^{g'} = U_{\omega, s}^{\gm}$ as well, and $U_{\omega, s}^{g} = U_{\omega, s}^{g'}$ in $\R^{1+n}$ for $|s| > T + 1$. 

Let us consider the case $s = T + 1$. Writing $\phi(t,x) = t - x \cdot \omega - s$ we have $U_{\omega, s}^{\gm} = \phi^* H$ and 
    \begin{align*}
\Box_g U_{\omega, s}^{\gm} 
= 
g(d\phi, d\phi) \phi^* H'' + (\Box_g \phi) \phi^* H'.
    \end{align*}
Both terms on the right-hand side vanish as distributions. Indeed, $\phi^* H''$ and $\phi^* H'$ are supported on the plane $\{\phi = 0\}$, this plane intersects $\overline{Q}_T$ only at the point $p = (T,-\omega)$, it holds in $\{\phi = 0\} \setminus \{p\}$ that
    \begin{align*}
g(d\phi, d\phi) = \gm(d\phi, d\phi) = 0,
\quad
\Box_g \phi = \Box_{\gm} \phi = 0,
    \end{align*}
and this vanishing extends to $p$ due to smoothness of $g$.
Similarly to the case $|s| > T + 1$, we conclude that $U_{\omega, s}^{g} = U_{\omega, s}^{g'}$ in $\R^{1+n}$ for $s = T + 1$, and the case $s = -T - 1$ is analogous.
It remains to consider the case $|s| < T + 1$.

Let $V = U_{\omega,s}-U_{\omega,s}'$. By \eqref{assumption1}, $\Box_{\gm} V = 0$ in $\{t < -T \}$ with $V = 0$ when $t \ll 0$. This implies that $V = 0$ in $\{ t < -T \}$. Since $\mathcal{F}_{\omega,T}(g) = \mathcal{F}_{\omega,T}(g')$, we see that $V$ solves the exterior Dirichlet problem 
\[
\Box_{\gm} V = 0 \text{ in $\{ t < T+2 \} \times (\mR^n \setminus \ol{B})$}, \ \ V|_{\{ t < T+2 \} \cap \p B} = 0.
\]
Since $V = 0$ in $\{ t < -T \}$, uniqueness in the exterior Dirichlet problem (see \cite[Lemma 6.1]{oksanen2024}) implies that 
\begin{equation} \label{v_vanishing}
V = 0 \text{ in $\{ t < T+2 \} \times (\mR^n \setminus \ol{B})$}.
\end{equation}
Now $\Box_{\gm} V = 0$ in $\{ t > T \}$ by \eqref{assumption1}. Since $|s| < T + 1$, propagation of singularities implies that $V$ can be singular in $\{ t > T \} \times B$ only along Minkowski null geodesics intersecting $\{ t = T \} \times B$. But these geodesics intersect $\{ t < T+2 \} \times (\mR^n \setminus \ol{B})$ and it follows from \eqref{v_vanishing} that $V$ is smooth in $\{ T < t < T+2 \}$. Then unique continuation from $(T, T+2) \times \p B$ implies that 
\[
V|_{t=T+1} = \p_t V|_{t=T+1} = 0.
\]
Solving the Cauchy problem for $\Box_{\gm}$ with this initial data yields $V = 0$ in $\{ t > T \}$ as required.
\end{proof}

\begin{proof}[Proof of Proposition \ref{prop_simplicity_detection}]
Lemma \ref{lem_from_F_to_U} implies for all $s \in \R$ and $\omega \in \Omega$
    \begin{align*}
\WF(U_{\omega, s}^{g}) = \WF(U_{\omega, s}^{g'})
\quad \text{on $T^* (\R^{1+n} \setminus Q_T)$}.
    \end{align*}
It follows from propagation of singularities that for all $\sigma \in \R$ and $\omega \in \Omega$
    \begin{align*}
&\{(\gamma_z^g(r), \lambda_z \dot \gamma_z^g(r)) \mid r=r_+^g(z),\ \lambda_z \ne 0, \ z \in \Sigma_{-,\sigma} \setminus \mathcal{T}_{g}\}
\\&\quad=
\{(\gamma_z^{g'}(r), \lambda_z \dot \gamma_z^{g'}(r)) \mid r=r_+^{g'}(z),\ \lambda_z \ne 0, \ z \in \Sigma_{-,\sigma}\}.
    \end{align*}
Since \eqref{def_line_bundle} is a line bundle over $\Sigma_{+}^{g'}$, there is $\lambda : \Sigma_+^{g'} \to \R$ such that \eqref{eq_beta} holds.
We conclude by combining Lemma \ref{lem_nontrapping} and Theorem \ref{th_diffeo}.
\end{proof}

\section{Plane waves and eikonal solutions} \label{sec_plane_waves}

In this section we will prove Propositions \ref{prop_plane_wave_structure}--\ref{prop_psi_well_defined} and Lemma \ref{lemma_psi_differences}.

\begin{proof}[Proof of Proposition \ref{prop_plane_wave_structure}]
(a): It follows from Lemma \ref{lem_simple_ext} and \cite[Lemma A.1]{oksanen2024} that there is a unique smooth solution $\varphi$ near $\mR \times \ol{B}$. Observe that $(T, -\omega) \in \ol{Q}_T$ and that $\varphi(T, -\omega) = T + 1$. Similarly, $(-T, \omega) \in \ol{Q}_T$ and \eqref{assumption1} implies that $\varphi(-T, \omega) = -T - 1$.
Since $\varphi(\ol{Q}_T)$ is a compact and connected subset of $\mR$, it is a closed interval containing $[-T-1,T+1]$.

Now suppose that $|\varphi(z)| > T+1$ for some $z \in \ol{Q}_T$. Let $\gamma(r)$ be the null geodesic through $z$ such that $\gamma(-r_0) \in \{ x \cdot \omega = -1 \}$ and $\dot{\gamma}(-r_0) = (1,\omega)$ for some $r_0$. Since $\varphi$ is constant along $\gamma$ (see \cite[Lemma A.1]{oksanen2024}), we have $|\varphi(\gamma(-r_0))| > T+1$. But $\gamma(-r_0) = (t_0,x_0)$ where $x_0 \cdot \omega = -1$, so $\varphi(\gamma(-r_0)) = t_0+1$ and therefore $|t_0+1| > T+1$. This means that $t_0 > T$ or $t_0 < -T-2$. However, by \eqref{assumption1} the null geodesic through such $(t_0,x_0)$ can never reach $\ol{Q}_T$, which is a contradiction. We have proved that $\varphi(\ol{Q}_T) = [-T-1,T+1]$. The same argument shows that in fact $\varphi(\{|t - x \cdot \omega| \leq T+1 \} \cap (\mR \times \ol{B})) = [-T-1,T+1]$.

(b) and (c): Following \cite[proof of Lemma 5.3]{oksanen2024}, for any $N$ the plane wave $U$ is given by 
\[
U = U^{(N)} + R^{(N)}
\]
where $U^{(N)}$ has the form 
\[
U^{(N)} = \left[ \sum_{j=0}^N u_j(\varphi-s)^j \right] H(\varphi-s).
\]
Here $u_j$ are functions obtained by solving transport equations 
\begin{align*}
Lu_0 &= 0, \qquad &u_0|_{\{x\cdot \omega \leq -1\}} &= 1, \\
j L u_j + \Box_g u_{j-1} &= 0, \qquad &u_j|_{\{ x\cdot \omega \leq -1 \}} &= 0 \ \ (j \geq 1)
\end{align*}
near $\mR \times \ol{B}$, where $Lv = 2Z_{g,\varphi} v + (\Box_g \varphi)v$. Since the null geodesics smoothly parametrize a neighborhood of $\mR \times \ol{B}$, it follows from the explicit formulas for $u_j$ in \cite[proof of Lemma 5.3]{oksanen2024} that each $u_j$ is smooth near $\mR \times \ol{B}$ and is independent of $s$, and that $u_0 > 0$. Moreover, $R^{(N)}$ is the solution of the Cauchy problem 
\[
\Box_g R^{(N)} = -(\Box_g u_N) (\varphi-s)^N H(\varphi-s), \qquad R^{(N)}|_{\{ \tau < \tau_0 \}} = 0,
\]
where $\tau$ is a Cauchy temporal function for $(\mR^{1+n},g)$ and $\tau_0$ is a suitable constant. As in \cite[proof of Lemma 5.3]{oksanen2024}, $R^{(N)} \in H^{N-1}_{\mathrm{loc}}(\mR^{1+n})$ and $R^{(N)}$ is supported in $\{ \varphi \geq s \}$. Since $R^{(N)} = G F$ where $G$ is the solution operator for $\Box_g$ with vanishing Cauchy data in $\{ \tau < \tau_0 \}$ and $F = -(\Box_g u_N) (\varphi-s)^N H(\varphi-s)$, the fact that $\p_s^j F \in H^{N-j}_{\mathrm{loc}}(\mR^{1+n})$ implies that $\p_s^j R^{(N)} \in H^{N-1-j}_{\mathrm{loc}}(\mR^{1+n})$ for $j \leq N-1$. Thus $U = u H(\varphi-s)$ where 
\[
u = \sum_{j=0}^N u_j(\varphi-s)^j + R^{(N)}.
\]
In particular, $u > 0$ on $\{ \varphi = s \}$.

Given any $k$, choosing $N$ large enough yields $u(z)$ that is $C^{k+3}$ jointly in $s$ and $z$. Then $\Box_g u$ is $C^{k+1}$ jointly in $s$ and $z$, and since $u=U$ in $\{ \varphi > s \}$ we have that $\Box_g u$ is supported in $\{ \varphi \leq s \}$. Taking $\varphi-s$ as a new coordinate, it follows that 
\[
\Box_g u = (\varphi-s) a
\]
where $a$ is $C^k$ jointly in $s$ and $z$ and supported in $\{ \varphi \leq s \}$. Similarly, since $Lu_0 = 0$, $Z_{g,\varphi}((\varphi-s)^j) = 0$ and $R^{(N)}$ vanishes to high order on $\{ \varphi=s \}$, we see that $Lu$ vanishes on $\{ \varphi = s \}$ and is $C^{k+2}$ in $s$ and $z$. Taking $\varphi-s$ as a new coordinate, it follows that $Lu = (\varphi-s) b$ where $b$ is $C^{k+1}$ in $s$ and $z$.

The argument above proves (b) and (c), except that $u$, $a$ and $b$ are only $C^k$ in $s$ and $z$ where $k$ can be any large number. This would already be enough for the subsequent proof. By a Borel summation argument (see \cite[Theorem 1.2.6]{hormander}), we can further take 
\[
\tilde{u} = \sum_{j=0}^{\infty} u_j(\varphi-s)^j \chi((\varphi-s)/\eps_j)
\]
where $\chi \in C^{\infty}_c((-1,1))$ satisfies $\p_t^j(\chi(t)-1)|_{t=0} = 0$ for all $j$ and $\eps_j \to 0$ sufficiently rapidly. Writing $u = \tilde{u} + R$ where $R$ is the smooth solution of a wave equation as above, we obtain  $U = u H(\varphi-s)$ where $u$ is smooth jointly in $s$ and $z$. Moreover, $\Box_g u = (\varphi-s) a$ and $2 Z_{g,\varphi} u + (\Box_g \varphi) u = (\varphi-s) b$ with $a$ and $b$ smooth in $s$ and $z$.
\end{proof}

\begin{proof}[Proof of Proposition \ref{prop_eikonal_detection}]
By Lemma \ref{lem_from_F_to_U} we have 
\begin{equation} \label{u_uprime_equality_exterior}
U_{\omega,s} = U_{\omega,s}' \text{ outside $\ol{Q}_T$}.
\end{equation}
By Proposition \ref{prop_plane_wave_structure}, there is a neighborhood $W$ of $\mR \times \ol{B}$ such that $\varphi_{\omega}$ is smooth in $W$ and $\WF(U_{\omega,s}) \cap T^* W = N^*(\{ \varphi_{\omega} =s \}) \cap T^*W$, and similar statements hold for $\varphi_{\omega}'$ and $U_{\omega,s}'$. Then from \eqref{u_uprime_equality_exterior} we obtain that 
\[
N^*(\{ \varphi_{\omega} =s \}) \cap T^*(W \setminus \ol{Q}_T) = N^*(\{ \varphi_{\omega}'=s \}) \cap T^*(W \setminus \ol{Q}_T).
\]
Thus if $z \in W \setminus  \ol{Q}_T$ and $|\varphi_{\omega}(z)| \leq T+1$ or $|\varphi_{\omega}'(z)| \leq T+1$, then we must have $\varphi_{\omega}(z) = \varphi_{\omega}'(z)$. On the other hand, if $|\varphi_{\omega}(z)| > T+1$ or $|\varphi_{\omega}'(z)| > T+1$, then $z$ must be in $\{ |t - x \cdot \omega| \geq T+1 \}$ by Proposition~\ref{prop_plane_wave_structure}, and in this set we have $\varphi_{\omega}(z) = \varphi_{\omega}'(z) = t - x \cdot \omega$. It follows that 
\[
\varphi_{\omega} = \varphi_{\omega}' \text{ in $W \setminus \ol{Q}_T$}. \qedhere
\]
\end{proof}

\begin{proof}[Proof of Proposition \ref{prop_psi_well_defined}]
By Proposition \ref{prop_plane_wave_structure} we know that 
\[
\varphi_{\omega_j}(\ol{Q}_T) = [-T-1,T+1]
\]
for $0 \leq j \leq n$, which implies that $\Phi(\ol{Q}_T) \subset [-T-1,T+1]^{1+n}$. It follows that for any neighborhood $\mathcal U$ of $[-T-1,T+1]^{1+n}$ there is a neighborhood $\mathcal V$ of $\ol{Q}_T$ satisfying $\Phi(\mathcal V) \subset \mathcal U$. Then by \eqref{assumption4} the map $\Psi = (\Phi')^{-1} \circ \Phi$ is well defined and smooth near $\ol{Q}_T$. From Proposition \ref{prop_eikonal_detection} we know that $\Phi = \Phi'$ slightly outside $\ol{Q}_T$, which implies that $\Psi = \id$ slightly outside $\ol{Q}_T$. By Lemma \ref{lemma_idext_surjective}, we also have $\Psi(\ol{Q}_T) = \ol{Q}_T$.
\end{proof}

\begin{proof}[Proof of Lemma \ref{lemma_psi_differences}]
We have 
\[
-\Box_g v = |g|^{-1/2} \p_j (|g|^{1/2} g^{jk} \p_k v) = g^{jk} \p_{jk} v + \theta_g^k \p_k v,
\]
where $\theta_g^k = |g|^{-1/2} \p_j (|g|^{1/2} g^{jk}) = - \Box_g z^k$. Since $\p_k(\Psi^* v) = \p_b v(\Psi) \p_k \Psi^b$, $\p_{jk}(\Psi^* v) = \p_{ab} v(\Psi) \p_j \Psi^a \p_k \Psi^b + \p_b v(\Psi) \p_{jk} \Psi^b$, a computation yields 
\begin{align*}
 &- \Box_g (\Psi^* v)  + \kappa \Psi^* (\Box_{g'} v) \\
 &= g^{jk} \p_{jk}(\Psi^* v) + \theta^k_g \p_k(\Psi^* v) - \kappa \big[g'(\Psi)^{jk} \p_{jk} v(\Psi) + \theta_{g'}^k(\Psi) \p_k v(\Psi) \big] \\
 &= \bar{g}^{ab} \p_{ab} v(\Psi) + \bar{\theta}^k \p_k v(\Psi).
\end{align*}
Similarly, 
\begin{align*}
Z_{g,\Psi^* \varphi}(\Psi^* v) - \kappa \Psi^*( Z_{g',\varphi} v) &= -g(d\Psi^* \varphi, d\Psi^* v) + \kappa \Psi^* ( g'(d\varphi,dv)) \\
 &= -\bar{g}^{ab} \p_a \varphi(\Psi) \p_b v(\Psi). \qedhere
\end{align*}
\end{proof}

\section{\texorpdfstring{$L^2$}{L2}-estimate with temporal weight} \label{sec_carleman}

In this section we will sketch a proof of Proposition \ref{prop_carleman_psi} (a detailed proof is given in Appendix \ref{sec_carleman_appendix}). In fact, we will prove an estimate of this type on a general Lorentzian manifold with rather precise constants.

Let $X$ be a smooth oriented manifold of dimension $1+n$ with smooth boundary $\p X$, and let $g$ be a smooth Lorentzian metric on $X$ with signature $(-,+,\ldots,+)$. Let $\psi \in C^{\infty}(X, \mR)$ be such that $g(d\psi, d\psi) < 0$. As in \cite[Section 24.1]{hormander}, we assume that $\p X$ is timelike and that $\psi$ is a proper map on $X$. Then $\psi$ is an analogue of a temporal function on the manifold $X$ with boundary. 

Let also $\varphi \in C^{\infty}(X,\mR)$ be such that $g(d\varphi,d\varphi) = 0$, $d\varphi$ is nowhere vanishing, and $d\psi$ and $d\varphi$ are in the same causal cone. We use the abbreviation 
\[
T := d\psi.
\]
Then $T$ is timelike and $d\varphi$ is null. We recall the fact that two vectors in the same causal cone have negative $g$-inner product unless they are null and parallel. Thus we have 
\[
g(T, T) < 0, \qquad g(d\varphi, d\varphi) = 0, \qquad g(d\varphi, T) < 0.
\]

Fix $s, R \in \mR$, and consider the set 
\[
Q = Q_{s,R} := \{ x \in X \,:\, \varphi \geq s, \ \psi \leq R \}.
\]
We also write 
\[
\Gamma = \{ \varphi = s \}, \qquad \Gamma_{R} = \{ \psi = R \},
\]
and assume that $\Gamma \subset \{ \psi < R \}$. Thus $Q$ is a cylinder-like set whose top $\Gamma_R$ is spacelike and bottom $\Gamma$ is characteristic (with suitable $s$ and $R$).

\begin{Example}
A basic example to keep in mind is when $X = \mR \times \ol{B}$, $g = \gm$, $\psi(t,x) = t$, and $\varphi(t,x) = t - x_1$. In this example 
\[
Q = \{ (t,x) \in \mR \times \ol{B} \,:\, t-x_1 \geq s, \ \ t \leq R \}.
\]
\end{Example}

We also remark that $\p X$ always meets both $\Gamma$ and $\Gamma_R$ transversally, since the conormals are spacelike, null and timelike, respectively. This implies that $Q$ is a manifold with corners and thus the standard integration by parts formulas hold in $Q$ \cite{lee2002}.

We will work in the context of semiclassical analysis \cite{zworski2012}, with $h > 0$ a small parameter. We wish to consider an $L^2$-estimate with boundary terms for the semiclassical operator 
\[
P := e^{\psi/h} \circ h^2 \Box_g \circ e^{-\psi/h},
\]
applied to functions in $Q$ whose Cauchy data vanishes on $\p X$. This estimate is given in terms of $L^2$ norms with respect to the canonical volume form $dV_g$ (see Appendix \ref{lorentzian_integration}). The estimate also involves $L^2$ norms of $1$-forms, and for this it is convenient to have a Riemannian metric on $X$. The next result states that if $g$ is a Lorentzian metric and $T$ is a timelike $1$-form, there is an associated Riemannian metric $g_T$ that satisfies $dV_{g_T} = dV_g$.

\begin{Lemma} \label{lemma_ht_riemannian}
Let $T$ be a smooth $1$-form on $X$ with $g(T,T) < 0$. Any $\xi \in T^* X$ has a unique decomposition $\xi = \lambda T + \xi_{\perp}$ where $g(T, \xi_{\perp}) = 0$, and then $\lambda = g(T,\xi)/g(T,T)$. If $\xi = \lambda T + \xi_{\perp}$, $\eta = \mu T + \eta_{\perp}$ and 
\[
g_T(\xi, \eta) := g(\xi_{\perp}, \eta_{\perp}) - g(T, T) \lambda \mu = g(\xi, \eta) - 2 \frac{g(T,\xi)g(T,\eta)}{g(T,T)},
\]
then $g_T$ is a Riemannian metric on $X$ with $dV_{g_T} = dV_g$.
\end{Lemma}

Lemma \ref{lemma_ht_riemannian} is proved in Appendix \ref{lorentzian_integration}. Note that if $g = \gm$ on $X = \mR \times \ol{B}$ and $\psi = \psi(t)$, then $g_T$ is the Euclidean metric on $\mR \times \ol{B}$.

To state the estimate, we fix $T = d\psi$ as above and define the inner products 
\[
(u, v) = \int_Q u \bar{v} \,dV_{g}, \qquad (du, dv) = \int_Q g_T(du, \ol{dv}) \,dV_{g},
\]
with associated norms $\norm{u}$ and $\norm{du}$. On $\Gamma$ and $\Gamma_R$ we will use the volume form $dS$ induced by $g_T$, with associated inner products $(u, v)_{\Gamma}$, $(du, dv)_{\Gamma}$ etc and corresponding norms. The main result in this section is the following estimate for the conjugated operator $P$.

\begin{Proposition} \label{prop_semiclassical_estimate}
Let $T = d\psi$ with $\psi$ as above, and let 
\[
m_T = |g(T,T)|^{1/2}.
\]
There is $C> 0$ 
such that when $0 < h \leq 1$ we have the estimate 
\begin{align*}
(1-Ch) \Big[ \norm{m_T^2 u}^2 + 2 \norm{m_T h du}^2 + \sqrt{2} h \norm{m_T^{3/2} u}_{\Gamma}^2 + \sqrt{2} h \norm{m_T^{1/2} h d_{\Gamma} u}_{\Gamma}^2 \Big] \\
 \leq \norm{Pu}^2 + (3+Ch) \left[ h \norm{m_T^{3/2} u}_{\Gamma_R}^2 + h \norm{m_T^{1/2} h d u}_{\Gamma_R}^2 \right]
\end{align*}
for any $u \in H^2(Q)$ with Cauchy data vanishing on $\p X$.
\end{Proposition}

Above $d_{\Gamma}$ and $d$ are the exterior derivatives on $\Gamma$ and in $X$. The proof of Proposition \ref{prop_semiclassical_estimate} is somewhat long. In this section we give a short sketch, which is hopefully sufficient to convince the reader that such an estimate holds. The detailed proof is left to Appendix \ref{sec_carleman_appendix}.

We explain the proof in three different cases:
\begin{enumerate}
\item[(a)] 
no boundary terms, i.e.\ $u \in H^2_0(Q)$,
\item[(b)]
boundary terms on $\Gamma_R$, and 
\item[(c)]
boundary terms on $\Gamma$.
\end{enumerate}

Case (a) amounts to computing the principal symbol of $P$ (all symbols are taken in the semiclassical sense), and observing that it is nowhere vanishing. Since $h^2 \Box_g$ has principal symbol $g(\xi, \xi)$, $P$ has principal symbol 
\[
p(x,\xi) = g(\xi + i T, \xi + i T) = g(\xi, \xi) - g(T, T) + 2i g(T, \xi).
\]
Then 
\begin{align*}
|p|^2 &= (g(\xi,\xi) - g(T,T))^2 + (2 g(T, \xi))^2 \\
 &= g(T,T)^2 + g(\xi,\xi)^2 - 2 g(T,T) \left[ g(\xi, \xi) - 2 \frac{g(T,\xi)^2}{g(T,T)} \right].
\end{align*}
The quantity in brackets is $g_T(\xi, \xi)$. Since $-g(T,T) = m_T^2$, we have 
\[
|p|^2 = m_T^4 + 2 m_T^2 g_T(\xi, \xi) + g(\xi, \xi)^2.
\]
This is $\geq m_T^4 + 2 m_T^2 g_T(\xi, \xi)$, suggesting an $H^1$ lower bound. This argument is carried out precisely in Lemma \ref{lemma_pu_interior_terms}, which shows that 
\[
\norm{Pu}^2 \geq  (1-Ch)((m_T^4 u, u) + 2 (m_T^2 hdu, hdu))
\]
when $u \in H^2_0(Q)$. This proves case (a) of Proposition \ref{prop_semiclassical_estimate}. Note that since $p$ is nowhere vanishing, Proposition \ref{prop_semiclassical_estimate} is a semiclassically elliptic estimate instead of being subelliptic as would be the case for a standard Carleman estimate.

Next we sketch a proof for (b). As in general Carleman estimates with boundary terms \cite{tataru1996, bellassoued2015}, the terms on $\Gamma_R$ can be understood by factoring $p$ into first order factors corresponding to the conormal direction of $\Gamma_R$. This conormal direction is just $T$. When $g_T(T, \xi) = 0$ (i.e.\ $g(T, \xi) = 0$), we compute 
\begin{align*}
p(x, \xi + \lambda T) &= g(\xi + \lambda T, \xi + \lambda T) - g(T, T) + 2i g(T, \xi + \lambda T) \\ 
 &= \lambda^2  g(T, T) + g(\xi, \xi) - g(T, T) + 2i \lambda g(T, T).
\end{align*}
We view the right hand side as a second order polynomial in $\lambda$. Its roots are 
\begin{align*}
\lambda &= \frac{-2i g(T, T) \pm \sqrt{-4 g(T, T)^2  - 4 g(T, T)( g(\xi, \xi)  - g(T, T)} )}{2 g(T, T)} \\
 &= -i \pm \sqrt{g(\xi, \xi)/|g(T, T)|}.
\end{align*}
Now $g(T, \xi) = 0$, so $g(\xi, \xi) = g_T(\xi, \xi)$. This proves that 
\begin{equation} \label{gammar_terms_symbol}
p(x, \xi + \lambda T) = g(T, T) (\lambda - \lambda_{+}(x, \xi)) (\lambda - \lambda_{-}(x,\xi))
\end{equation}
where $\lambda_{\pm}(x, \xi) = -i \pm \sqrt{g_T(\xi, \xi)/|g(T, T)|}$. If we consider coordinates $(x',x_n)$ so that $x'$ is tangential to $\Gamma_R$ and $x_n = R-\psi$, quantizing \eqref{gammar_terms_symbol} yields the pseudodifferential factorization identity 
\[
P = g(T,T) (hD_n - \lambda_+(x',x_n,hD_{x'})) (hD_n - \lambda_-(x',x_n,hD_{x'}))
\]
modulo a lower order term. Here $\lambda_{\pm}$ are first order semiclassical operators depending smoothly on $x_n$. Since $\mathrm{Im}(\lambda_{\pm}(x,\xi)) < 0$, both first order factors behave like heat equations in $x_n$ (see e.g.\ \cite[Chapter III]{treves1980}). This suggests that one can solve a Cauchy problem for $P$ near $\Gamma_R$, and therefore in Proposition \ref{prop_semiclassical_estimate} one should indeed have boundary terms involving $u$ and $\p_{\nu} u$ on $\Gamma_R$ on the right hand side. The boundary terms on $\Gamma_R$ are computed in detail in Lemma \ref{lemma_pu_gammar_terms}.

The sketch for (c) is similar as for (b), though slightly more involved due to the characteristic nature of $\Gamma$. Now $N = d\varphi$ is conormal to $\Gamma$ and satisfies $g(N, N) = 0$. If $g_T(\xi, N) = 0$ we have 
\begin{align*}
p(x, \xi + \lambda N) &= g(\xi+\lambda N, \xi + \lambda N) - g(T,T) + 2i g(T, \xi + \lambda N) \\
 &= 2 \lambda (g(\xi, N) + i g(T, N)) + p(x,\xi).
\end{align*}
This is a first order polynomial in $\lambda$. If $(x',x_n)$ are coordinates such that $x'$ is tangential to $\Gamma$ and $x_n = \varphi-s$ is the normal direction near $\Gamma$, quantizing yields the pseudodifferential identity 
\[
P = R(x',x_n,hD_{x'}) hD_n + S(x',x_n, hD_{x'})
\]
modulo a lower order term. Here $R$ has symbol $2 g(\xi, N) + 2 i g(T, N)$ which never vanishes since $g(N, T) < 0$. Thus $R$ is elliptic, and inverting $P$ near $\Gamma$ corresponds to inverting $hD_n + R^{-1} S$, where $R^{-1}$ has symbol $r^{-1}$.

The nature of $hD_n + R^{-1} S$ is determined by the sign of 
\[
\mathrm{Im}(r^{-1} s) = |r|^{-2} \mathrm{Im}(\bar{r} s).
\]
If $g_T(\xi, N) = 0$, we have 
\begin{align*}
\mathrm{Im}(\bar{r} s) &= 2 (g(N, \xi) \mathrm{Im}(p(x,\xi)) - g(N,T) \mathrm{Re}(p(x,\xi))) \\
 &=  2 ( 2 g(N, \xi) g(T, \xi) - g(N,T)(g(\xi,\xi) - g(T,T)) ) \\
 &= 2 ( E(\xi) + g(N,T) g(T,T))
\end{align*}
where $E(\xi) = 2 g(N, \xi) g(T, \xi) - g(N,T) g(\xi,\xi)$. This quadratic form is a variant of the one in \cite[Lemma 24.1.2]{hormander}. Since $g_T(N, \xi) = 0$, we have 
\begin{align*}
E(\xi) &= 4 \frac{g(T,N) g(T,\xi)}{g(T,T)} g(T,\xi) - g(N,T) g(\xi,\xi) \\
 &= |g(N,T)| (g_T(\xi,\xi) +  2 g_T(\hat{T},\xi)^2)
\end{align*}
where $\hat{T} = T/|T|_{g_T}$, and therefore 
\[
\mathrm{Im}(\bar{r} s) = 2 |g(N,T)| (|g(T,T)| + g_T(\xi,\xi) + 2 g_T(\hat{T},\xi)^2).
\]
In particular $\mathrm{Im}(r^{-1} s) > 0$, so $hD_n + R^{-1} S$ and thus $P$ behave like backward heat operators near $\Gamma$. This implies that in Proposition \ref{prop_semiclassical_estimate}, one should indeed have a boundary term involving $u$ (but not $\p_{\nu} u$) on $\Gamma$ on the left hand side. The precise form of the boundary terms is given in Lemma \ref{lemma_pu_gamma_terms}.

To have precise control over the constants, we will give a direct proof of Proposition \ref{prop_semiclassical_estimate} in Appendix \ref{sec_carleman_appendix} via integration by parts instead of using pseudodifferential factorizations as in the sketch above.

\section{Proof of Theorems \ref{thm_main1}--\ref{thm_main_semiglob}}
\label{sec_proof_of_main_ths}

In this section we prove that small compactly supported perturbations of the Minkowski metric satisfy the conditions \eqref{assumption1}--\eqref{assumption5} in Theorem \ref{thm_main2}. This will imply that Theorems \ref{thm_main1}--\ref{thm_main_semiglob} reduce to Theorem \ref{thm_main2}.

\subsection{The metric \texorpdfstring{$\gm$}{gMin}}

We first observe that $\gm$ satisfies \eqref{assumption1}--\eqref{assumption5} with a suitable set of directions $\Omega$.

\begin{Lemma} \label{lemma_minkowski_conditions}
The metric $g = \gm$ satisfies \eqref{assumption1}--\eqref{assumption2}. Moreover:
\begin{enumerate}
\item[(a)] 
The simplicity condition \eqref{assumption3} holds in any direction $\omega \in S^{n-1}$.
\item[(b)]
The map $(\varphi_{\omega_0}, \ldots, \varphi_{\omega_n})$ is a linear diffeomorphism from $\mR^{1+n}$ to itself, so in particular \eqref{assumption4} holds, whenever $\omega_1, \ldots, \omega_n$ are linearly independent and $\omega_0$ is not of the form $\sum_{j=1}^n a_j \omega_j$ for some $a_j \in \mR$ with $\sum a_j = 1$.
\item[(c)]
The spanning condition \eqref{assumption5} holds with $\Omega = \Omega_{\mathrm{Min}}$, where $\Omega_{\mathrm{Min}} = \{ \pm e_1, \ldots, \pm e_n \} \cup \{ \frac{e_j+e_k}{\sqrt{2}} \,:\, 1 \leq j < k \leq n \}$.
\end{enumerate}
\end{Lemma}
\begin{proof}
Let $(t,x)$ be global coordinates in $\mR^{1+n}$. Then $(\mR^{1+n}, \gm)$ satisfies \eqref{assumption1}, and \eqref{assumption2} follows since $t$ is a Cauchy temporal function.

(a) For $\gm$, the null geodesics $\gamma_z(r)$ for $z \in \Sigma_- = \{ x \cdot \omega = -1 \}$ with $\gamma_z(0) = z$ and $\dot{\gamma}_z(0) = (1,\omega)$, are the straight lines 
\[
\gamma_z(r) = z + r(1,\omega).
\]
Thus the map $\Sigma_- \times \mR \to \mR^{1+n}$, $(z,r) \mapsto \gamma_z(r)$ is a diffeomorphism onto $\mR^{1+n}$, which implies \eqref{assumption3} for any $\omega$. 

(b) The eikonal solutions for $\gm$ are $\varphi_{\omega}(t,x) = t - x \cdot \omega$. Let 
\[
\Phi = (\varphi_{\omega_0}, \ldots, \varphi_{\omega_n}) = \left( \begin{array}{rc} 1 & -\omega_0 \\ \vdots & \vdots \\ 1 & -\omega_n \end{array} \right) \begin{pmatrix} t \\ x \end{pmatrix}.
\]
Let $A$ be the transpose of the matrix on the right. Then $A$ has nontrivial kernel if there is $(c_1, \ldots, c_n) \in \mR^n \setminus \{0\}$ such that 
\[
\sum_{j=1}^n c_j \omega_j = (\sum c_j) \omega_0.
\]
If $\sum c_j = 0$, one would get a contradiction since $\omega_1, \ldots, \omega_n$ are linearly independent. Thus, after dividing by $\sum c_j$, $\Phi$ is a diffemorphism onto $\mR^{1+n}$ unless $\omega_0 = \sum_{j=1}^n a_j \omega_j$ with $\sum a_j = 1$.

(c) It is enough to prove that the set $\{ (1,-\omega) \otimes (1, -\omega) \}_{\omega \in \Omega_{\mathrm{Min}}}$ spans 
\[
S = \{ A = (a_{jk})_{j,k=0}^n \,:\, a_{jk} = a_{kj}, \ -a_{00} + a_{11} + \ldots + a_{nn} = 0 \}.
\]
Consider the matrix product $A : B = \sum_{j,k} a_{jk} b_{jk}$. Since $A : (v \otimes w) = Av \cdot w$, the spanning condition holds iff any $A \in S$ satisfying 
\[
A \begin{pmatrix} 1 \\ -\omega \end{pmatrix} \cdot \begin{pmatrix} 1 \\ -\omega \end{pmatrix} = 0, \qquad \omega \in \Omega_{\mathrm{Min}},
\]
must be zero. Writing $A = \begin{pmatrix} a & b \\ b^t & C \end{pmatrix}$ with $a \in \mR$, $b \in \mR^n$ and $C$ a symmetric $n \times n$ matrix, the above condition reads 
\[
a - 2 b \cdot \omega + C \omega \cdot \omega = 0, \qquad \omega \in \Omega_{\mathrm{Min}}.
\]
Adding and subtracting these equations for  $\omega = \pm e_j$ gives 
\[
2 b \cdot e_j = 0, \qquad a + c_{jj} = 0.
\]
Since this holds for all $j$, we get $b = 0$ and $c_{jj} = -a$. Then taking $\omega = (e_j+e_k)/\sqrt{2}$, $j < k$, gives 
\[
a + (c_{jj} + 2 c_{jk} + c_{kk})/2 = 0.
\]
Inserting $c_{jj} =  -a$ gives $c_{jk} = 0$ for $j < k$. It follows that  
\[
A = \begin{pmatrix} a & 0 \\ 0 & -a \mathrm{Id} \end{pmatrix}.
\]
Since $A \in S$, the trace condition yields $-(1+n)a = 0$, which implies $a=0$ and $A = 0$ as required.
\end{proof}

\subsection{Compactly supported perturbations of \texorpdfstring{$\gm$}{gMin}}

We first show that any $C^0$-small compactly supported perturbation of $\gm$ satisfies \eqref{assumption2}. Below, a Lorentzian manifold is said to be causally geodesically complete if every inextendible causal geodesic is defined on the whole real axis. 

\begin{Lemma}\label{lem_globhyp_near_Min}
Let $K \subset \R^{1+n}$ be compact. There is $\eps > 0$ such that any smooth Lorentzian metric $g$ that coincides with $\gm$ outside $K$ and satisfies
\begin{gather*}
\norm{g-\gm}_{C^0(\R^{1+n})} \leq \eps,
\end{gather*}
is globally hyperbolic and causally geodesically complete.
\end{Lemma}
\begin{proof}
Let $p \in K$ and $v = (v^0, v') \in T_p\R^{1 + n}$.
Then, writing $h = g-\gm$,
\[
|h(v,v)| \lesssim \eps (|v^0|^2 + |v'|^2).
\]
In particular, if $g(v,v) \leq 0$ then
\[
-(v^0)^2 + |v'|^2 = \gm(v,v) = g(v,v) - h(v,v) \lesssim \eps (|v^0|^2 + |v'|^2).
\]
For small enough $\eps > 0$, we obtain the cone bound 
\begin{equation}\label{cone_bd}
|v'|^2 \lesssim |v^0|^2.
\end{equation}
If, further, $v \ne 0$, then $v^0 \ne 0$. Hence $t$ is strictly monotone along every $g$-causal curve, and these curves can be parametrized by $t$. 

We fix the time-orientation so that $dt$ is future-directed.
Let $I \subset \R$ be an interval and let $\gamma : I \to \R^{1+n}$ be an inextendible future-directed $g$-causal curve parametrized by $t$. We write $\gamma(t) = (t, x(t))$ and show that $I = \R$. 
The cone bound \eqref{cone_bd} gives $|\dot x(t)| \lesssim 1$.
Let $a \in I$ and suppose for contradiction that $I \cap [a, \infty) = [a, b)$ for some $b < \infty$. Then for all $t, s \in [a, b)$,
\[
|x(t) - x(s)| \lesssim |t - s| \le b - a,
\]
so $x([a, b)) \subset K'$ for a compact set $K' \subset \R^{n}$. 
We see that $x : [a, b) \to K'$ is Lipschitz continuous. Hence $x(t) \to y$ for some $y \in K'$ as $t \to b$, contradicting inextendibility. The same argument applies in the past direction.

Therefore every inextendible $g$-causal curve crosses each level set of $t$ exactly once. In particular, $\{t = 0\}$ is a Cauchy surface on $(\R^{1+n}, g)$.

Consider now the case that $\gamma$ is a causal geodesic. As $t$ is bounded in the compact set $K$, we see using the non-affine reparametrization $\gamma(t) = (t, x(t))$, that $\gamma$ stays away from $K$ for $t \gg 1$ and $t \ll 1$. Thus $\gamma$ is an unbounded line segment for $s \gg 1$ and $s \ll 1$, where $s$ is the affine parameter of $\gamma$.
\end{proof}

Next we consider the stability of \eqref{assumption3}. Recall that the map $\Phi^g$ is defined by \eqref{Phig}.

\begin{Lemma}\label{lem_simple_near_Min}
Let $g$ be a smooth Lorentzian metric in $\mR^{1+n}$ satisfying \eqref{assumption1}. Let $K_1 \subset \Sigma_- \times \R$ and $K_2 \subset \R^{1+n}$ be compact. There is $\eps > 0$ such that if
\begin{gather*}
\norm{g-\gm}_{C^2(\R^{1+n})} \le \eps,
\end{gather*}
then there are neighborhoods $U \subset \Sigma_- \times \R$ of $K_1$ 
and $V \subset \R^{1+n}$ of $K_2$ such that the map
$\Phi^g : U \to V$ 
is a diffeomorphism. Moreover, for compact $K \subset U$
    \begin{align*}
\norm{\Phi^g - \Phi^{\gm}}_{C^1(K)} \to 0, \quad \eps \to 0.
    \end{align*}
\end{Lemma}
\begin{proof}
Choose $\eps > 0$ small enough so that $(\R^{1+n}, g)$ is causally geodesically complete (see Lemma \ref{lem_globhyp_near_Min}), and write $\mathcal B$ for the open ball of radius $\eps$ around $\gm$ in the space $\mathcal V \subset C^2(\R^{1+n})$ of $2$-tensors satisfying \eqref{assumption1}.
Consider the flow  
    \begin{align*}
\alpha(r; p, v, g) = (\gamma^g(r; p, v), \dot \gamma^g(r; p, v), g)
    \end{align*}
associated to the vector field
    \begin{align*}
f(x, \xi, g) = (\xi, \eta, 0), 
\quad \eta^j(x, \xi, g) = -\Gamma^j_{kl}(x; g) \xi^k \xi^l.
    \end{align*}
The Christoffel symbols $\Gamma^j_{kl}(x; g)$ depend on $g$ and its first derivatives. Thus $f \in C^1(\mathcal U; \mathcal W)$ 
where $\mathcal W = (\R^{1+n})^2 \times \mathcal V$ and  
$\mathcal U = (\R^{1+n})^2 \times \mathcal B$.
By \cite[Theorem 5.2]{lang1983} we have $\alpha \in C^1(\R \times \mathcal U; \mathcal U)$. 
Write
    \begin{align*}
\tilde \alpha_g(r,p,\omega) = \alpha(r; p,(1,\omega), g).
    \end{align*}
Let $I \subset \R$ and $K \subset \R^{1+n}$ be compact sets. Then for $g \in \mathcal B$, noting that $\alpha$ is $C^1$, 
    \begin{align*}
\norm{\tilde \alpha_g - \tilde \alpha_{\gm}}_{C^0(I \times K \times S^{n-1})}
\lesssim
\norm{g - \gm}_{C^2(\R^{1+n})} \le \eps.
    \end{align*}
Let $(r, p) \in I \times K$, $\omega \in S^{n-1}$, and $g \in \mathcal B$. It holds, in particular, that $|\tilde \alpha_g(r; p, \omega)| \lesssim 1$.
We write $q = (p, \omega)$ and 
    \begin{align*}
\beta_g(r; p, \omega) = \p_q \gamma^g(r; p, (1, \omega)),
    \end{align*}
Differentiating $\dot \alpha = f \circ \alpha$ in $q$ we get
    \begin{align*}
\ddot \beta_g^j = -2 \Gamma^j_{kl}(\gamma^g, g) (\dot \gamma^g)^k \dot \beta_g^l -\p_x \Gamma^j_{kl}(\gamma^g, g) \beta_g (\dot \gamma^g)^k (\dot \gamma^g)^l.
    \end{align*} 
In particular, $\ddot \beta_{\gm} = 0$. The difference $w = \beta_g - \beta_{\gm}$ satisfies $w(0) = 0$, $\dot w(0) = 0$ and 
    \begin{align*}
|\ddot w| 
\lesssim \eps |\dot \beta_g| + \eps |\beta_g|
\lesssim \eps |\dot w| + \eps |w| + \eps.
    \end{align*}
Thus $|w| \lesssim \eps$. Indeed,  
we have
\begin{equation*}
w(t) = \int_0^t \dot w(s) ds, 
\quad 
\dot w(t) = \int_0^t \ddot w(s) ds,
\end{equation*}
Thus, writing $u(t)=|w(t)|+|\dot w(t)|$,
    \begin{align*}
u(t) \lesssim \eps + \int_0^t u(s) ds.
    \end{align*}
and Gr\"onwall yields $u(t)\lesssim \eps$.
Writing $\tilde \gamma^g(r, p, \omega) = \gamma^g(r; p, (1, \omega))$, we have shown that 
    \begin{align*}
\norm{\tilde \gamma_g - \tilde \gamma_{\gm}}_{C^1(I \times K \times S^{n-1})}
\lesssim \eps, \quad g \in \mathcal B.
    \end{align*}
The claims that $\Phi^g : U \to V$ 
is a diffeomorphism and $\norm{\Phi^g - \Phi^{\gm}}_{C^1(K)} \to 0$ as $\eps \to 0$ follow from stability of diffeomorphisms (see Lemma \ref{lemma_diffeo_perturbation}).
\end{proof}

\begin{Lemma}\label{lem_diffeo_near_Min}
Let $g$ be a smooth Lorentzian metric in $\mR^{1+n}$ satisfying \eqref{assumption1}. Let $K_1, K_2 \subset \R^{1+n}$ be compact and let $\Omega_0 \subset S^{n-1}$ be finite. 
Write $\omega_j = e_j$, $j = 1,\dots,n$, $\omega_0 = -e_1$ and $\Omega = \{\omega_0, \dots, \omega_n\} \cup \Omega_0$.
There is $\eps > 0$ such that if
\begin{gather*}
\norm{g-\gm}_{C^2(\R^{1+n})} \le \eps,
\end{gather*}
then for $\omega \in \Omega$ the eikonal equation 
    \begin{align}\label{eikonal_aux}
g(d\varphi, d\varphi) = 0, \quad \varphi|_{\{ x \cdot \omega \leq -1 \}} = t-x \cdot \omega,
    \end{align}
has a smooth solution $\varphi = \varphi_{\omega}^g$ near $K_1$ and 
$F_g(x) = (\varphi_{\omega_0}^g(x), \dots, \varphi_{\omega_n}^g(x))$ is a diffeomorphism from a neighborhood of $K_1$ onto a neighborhood of $K_2$. Moreover, for $\omega \in \Omega$  
    \begin{align}\label{conv_phi}
\norm{\varphi_{\omega}^g - \varphi_{\omega}^{\gm}}_{C^1(K_1)} \to 0, \quad \eps \to 0.
    \end{align}
\end{Lemma}
\begin{proof}
By Lemma \ref{lemma_minkowski_conditions} (b), $F_{\gm}$ is a linear diffeomorphism from $\R^{1+n}$ to itself. Choose large enough compact $K \subset \R^{1+n}$ so that $K_1 \subset K$ and $K_2 \subset F_{\gm}(K)$. 

By Lemma~\ref{lem_simple_near_Min}, if $\eps > 0$ is small enough, then there are open $U_\omega \subset \Sigma_-^\omega \times \R$ and $V \subset \R^{1+n}$ such that $K \subset V$ and that the map $\Phi_\omega^g : U_\omega \to V$ is a diffeomorphism for all $\omega \in \Omega$. By \cite[Proposition 4.1]{oksanen2024} and Lemma \ref{lem_globhyp_near_Min}
the solution $\varphi_\omega^g$ of \eqref{eikonal_aux} can be expressed in terms of the inverse of $\Phi_\omega^g$ and the projection to the time coordinate: 
\[
\varphi_\omega^g - 1 = \pi \circ (\Phi_\omega^g)^{-1} \in C^\infty(V),
\quad 
\pi((t, x), r) = t, \quad (t, x) \in \R^{1+n},\ r \in \R.
\]
The convergence \eqref{conv_phi} follows from 
    \begin{align*}
\norm{(\Phi_\omega^g)^{-1} - (\Phi_\omega^{\gm})^{-1}}_{C^1(K_1)} \to 0, \quad \eps \to 0,
    \end{align*}
which itself is proven by combining Lemmas \ref{lem_simple_near_Min} and \ref{lemma_diffeo_perturbation}. Moreover, for small $\eps > 0$, $F_g$ is a diffeomorphism from a neighborhood of $K_1$ onto a neighborhood of $K_2$ in view of Lemma \ref{lemma_diffeo_perturbation}.
\end{proof}

Finally we consider the stability of \eqref{assumption5}.

\begin{Lemma}\label{lem_spanning_near_Min}
Let $g$ be a smooth Lorentzian metric in $\mR^{1+n}$ satisfying \eqref{assumption1}. There is $\eps > 0$ such that if
\begin{gather*}
\norm{g-\gm}_{C^2(\R^{1+n})} \le \eps,
\end{gather*}
then $g$ and
$d\varphi_\omega^g \otimes d\varphi_\omega^g$, $\omega \in \Omega_{\mathrm{Min}}$, span all symmetric $2$-tensors at each $x$.
\end{Lemma}
\begin{proof}
By Lemma \ref{lemma_minkowski_conditions} (c), the matrices $\gm$ and $d\varphi_{\omega} \otimes d\varphi_{\omega}$, $\omega \in \Omega_{\mathrm{Min}}$, with $\varphi_{\omega} = t - x \cdot \omega$, span the set of symmetric $(1+n) \times (1+n)$ matrices. If $m = (n+1)(n+2)/2$, writing these matrices as column vectors in $\mR^{m}$ gives a basis of $\mR^m$.

Now for $g$ as in the statement, write $g$ and $d\varphi_\omega^g \otimes d\varphi_\omega^g$, $\omega \in \Omega_{\mathrm{Min}}$, similarly as columns of a matrix $A(x, g) \in \R^{m \times m}$. 
Then $\det(A(x,\gm))$ is nonzero. By Lemma~\ref{lem_diffeo_near_Min}, for $K$ compact and for each $\omega \in \Omega_{\mathrm{Min}}$,
\[
\norm{d\varphi_\omega^g - d\varphi_\omega^{\gm}}_{C^0(K)} \to 0 \quad \text{as } \eps \to 0.
\]
It follows that
\[
\norm{\det(A(x,g)) - \det(A(x,\gm))}_{C^0(K)} \to 0 \quad \text{as } \eps \to 0.
\]
This proves the claim when $\eps$ is small enough.
\end{proof}

\subsection{Reduction of Theorems \ref{thm_main1}--\ref{thm_main_semiglob} to Theorem \ref{thm_main2}}

We give two lemmas before proving Theorem \ref{thm_main_semiglob}, and then turn to a proof of Theorem~\ref{thm_main1}.

\begin{Lemma}
Let $g$ be a smooth Lorentzian metric on $\R^{1+n}$ satisfying \eqref{assumption1}--\eqref{assumption2}. 
Let $t_0 > T$ or $t_0 < -T$. Then $\{t=t_0\}$ is a Cauchy surface on $(\R^{1+n}, g)$.
\end{Lemma}
\begin{proof}
By \cite[Lemma 2.8]{oksanen2024} any timelike curve intersects $\{t=t_0\}$. The same proof shows that any causal curve intersects $\{t=t_0\}$. Since $g = \gm$ near $\{t \le t_0\}$, the function $t$ is strictly monotonous along any causal curve near $\{t \le t_0\}$. Thus any causal curve can intersect $\{t \le t_0\}$ at most once.
\end{proof}

\begin{Lemma}\label{lem_cauchy_to_plane}
Let $g_j$, $j=1,2$, be smooth Lorentzian metrics on $\R^{1+n}$ satisfying \eqref{assumption1}--\eqref{assumption2}. Let $M$ be a domain with piecewise smooth boundary and suppose that $\overline M \subset Q_T$ and that $g_1 = g_2$ outside $M$. Then $C_{g_1}^{\mathrm{Hyp}} = C_{g_2}^{\mathrm{Hyp}}$ implies $\mathcal F_{\omega, \infty}(g_1) = \mathcal F_{\omega, \infty}(g_2)$ for all $\omega \in S^{n-1}$.
\end{Lemma}
\begin{proof}
Choose $\varphi_j \in C^{\infty}_c(\mR)$ so that $\supp(\varphi_j) \subset [-1,\infty)$ and $\varphi_j \to H$ in the sense of distributions, and let $u_j^{\gm} = \varphi_j(t-x\cdot \omega - s)$. Then $u_j^{\gm}$ solves $\Box_{\gm} u_j^{\gm} = 0$ in $\mR^{1+n}$, $\supp(u_j^{\gm}) \subset \{ t - x \cdot \omega \geq s -1 \}$, and $u_j^{\gm} \to U_{\omega,s}^{\gm} = H(t - x \cdot \omega - s)$ in the sense of distributions. For $t_0 \ll -T$ 
    \begin{align}\label{aux_supp_cauchy_to_plane}
\supp(u_j^{\gm}) \cap \{t < t_0\} \cap Q_\infty = \emptyset.
    \end{align}

Let $g = g_1, g_2$.
As $\{t = t_0\}$ is a Cauchy surface on $(\R^{1+n}, g)$, there is a unique solution $u_j^g \in C^\infty(\R^{1+n})$ of 
    \begin{align}\label{aux_wave_eq_cauchy_to_plane}
\begin{cases}
\Box_g u = 0, & \text{in $\R^{1+n}$}
\\
u|_{t=t_0} = u_j^{\gm}|_{t=t_0},
\\
\p_t u|_{t=t_0} = \p_t u_j^{\gm}|_{t=t_0}.
\end{cases}
    \end{align}
It follows from \eqref{aux_supp_cauchy_to_plane} that $u_j^g = u_j^{\gm}$ for $t < t_0$.
We define 
\[
u_j = \left\{ \begin{array}{cl} u_j^{g_1} & \text{in $M$}, \\ u_j^{g_2} & \text{outside $M$}. \end{array} \right.
\]
The assumption $C_{g_1}^{\mathrm{Hyp}} = C_{g_2}^{\mathrm{Hyp}}$ implies
$u_j \in H^2_{\mathrm{loc}}(\mR^{1+n})$, and in view of $g_1 = g_2$ outside $M$, $u_j$ solves \eqref{aux_wave_eq_cauchy_to_plane} with $g=g_1$. Due to uniqueness of the solution $u_j = u_j^{g_1}$. 
It follows from \cite[Lemma 4.1]{bar2015} that 
$u_j^{g_k} \to U_{\omega,s}^{g_k}$, $k=1,2$, in the sense of distributions as $j \to \infty$.
By taking distributional traces on $\mR \times \p B$ using \cite[Theorem 8.2.4 and Corollary 8.2.7]{hormander}, which is possible since the wave front sets of $U_{\omega,s}^{g_k}$, $k=1,2$, are disjoint from $N^*(\mR \times \p B)$, we obtain 
\[
U_{\omega,s}^{g_1}|_{\mR \times \p B} 
= \lim_{j \to \infty} u_j^{g_1}|_{\mR \times \p B} 
= \lim_{j \to \infty} u_j|_{\mR \times \p B} 
= \lim_{j \to \infty} u_j^{g_2}|_{\mR \times \p B} 
= U_{\omega,s}^{g_2}|_{\mR \times \p B}.
\]
\end{proof}

In view of Lemma \ref{lem_cauchy_to_plane}, Theorem \ref{thm_main_semiglob} reduces to the following corollary of Theorem \ref{thm_main2} and Lemmas \ref{lem_simple_near_Min}--\ref{lem_spanning_near_Min}.

\begin{Corollary}\label{cor_main}
Let $g, g'$ be smooth Lorentzian metrics  in $\mR^{1+n}$ satisfying \eqref{assumption1}--\eqref{assumption2}. There is $\eps > 0$ such that if
\begin{gather*}
\norm{g'-\gm}_{C^2(\ol{Q}_T)} \leq \eps,
\end{gather*}
the condition
$\mathcal F_{\omega, \infty}(g) = \mathcal F_{\omega, \infty}(g')$,  $\omega \in S^{n-1}$, implies 
\[
g = \Psi^* g'
\]
for some diffeomorphism $\Psi: \mR^{1+n} \to \mR^{1+n}$ with $\Psi = \id$ outside $Q_T$.
\end{Corollary}

Let us reduce Theorem \ref{thm_main1} to Corollary \ref{cor_main}. We first state a basic property of Lorentzian isometries.

\begin{Proposition} \label{prop_conformal_identity}
Let $g$ be a smooth Lorentzian metric on $\mR^{1+n}$, and let $\Psi: (U,g) \to (\mR^{1+n},g)$ be a Lorentzian isometry where $U \subset \mR^{1+n}$ is connected. If there is $z \in U$ such that $\Psi(z) = z$ and $D\Psi(z) = \id$, then $\Psi = \id$ in $U$.
\end{Proposition}
\begin{proof}
Let $B = B(0,r) \subset T_z U$ be a ball such that $\exp_z$ is a diffeomorphism from $B$ onto the geodesic ball $\exp_z(B)$. If $v \in B$ and if $\gamma(t) = \exp_z(tv)$ for $t \in [0,1]$ is the radial geodesic starting at $z$, define $\eta(t) = \Psi(\gamma(t))$. Since $\Psi$ is an isometry, also $\eta$ is a geodesic. By the assumptions on $\Psi$ we have $\eta(0) = z = \gamma(0)$ and $\dot{\eta}(0) = D\Psi(\dot{\gamma}(0)) = \dot{\gamma}(0)$. By uniqueness of geodesics $\eta(t) = \gamma(t)$ for $t \in [0,1]$, i.e.\ $\Psi(\gamma(t)) = \gamma(t)$ for $t \in [0,1]$. This implies that 
\[
\Psi|_{\exp_z(B)} = \id.
\]
Since any two points in the connected Lorentzian manifold $U$ can be connected by a broken geodesic \cite{oneill1983}, and since any point is contained in a geodesic ball, iterating the above argument proves that $\Psi = \id$ on $U$.
\end{proof}

\begin{proof}[Proof of Theorem \ref{thm_main1}]
Write $h_j = g_j - \gm$, $j=1,2$. We extend $h_1$ to a smooth compactly supported 2-tensor on $\R^{1+n}$, denoted by $\tilde h_1$, so that 
    \begin{align*}
\norm{\tilde h_1}_{C^2(\R^{1+n})} \lesssim \eps.
    \end{align*}
We then extend $h_2$ by setting $h_2 = \tilde h_1$ in $\R^{1+n} \setminus M$. Since $g_1$ and $g_2$ coincide up to infinite order on $\p M$, this extension, denoted by $\tilde h_2$, is smooth. Writing $\tilde g_j = \gm + \tilde h_j$, we have $\tilde g_1 = \tilde g_2$ outside $M$, $\tilde g_j = \gm$ outside a compact set, and 
    \begin{align*}
\norm{\tilde g_j - \gm}_{C^2(\R^{1+n})} \lesssim \eps.
    \end{align*}
By Lemma \ref{lem_globhyp_near_Min}, $\tilde g_j$ is globally hyperbolic for small $\eps > 0$. By Lemma \ref{lem_cauchy_to_plane} and Corollary \ref{cor_main} there is a cylinder $Q_T$ containing $M$ and a diffeomorphism $\Psi$ from $\mR^{1+n}$ to itself satisfying $\Psi = \id$ outside $Q_T$ and $g = \Psi^* g'$ everywhere. By Proposition \ref{prop_conformal_identity} and the assumption that $\mR^{1+n} \setminus M$ is connected, $\Psi = \id$ outside $M$.
\end{proof}

\appendix

\section{Integration on Lorentzian manifolds} \label{lorentzian_integration}

Let $X$ be a smooth oriented manifold of dimension $1+n$ with corners, and let $g$ be a smooth Lorentzian metric on $X$ with signature $(-,+,\ldots,+)$. Since $(X, g)$ is oriented, there is a canonical volume form $dV_g$ such that $dV_g(e_0, \ldots, e_n) = 1$ for any positively oriented basis $(e_0, \ldots, e_n)$ of $T_x X$ with $g(e_0, e_0) = -1$ and $g(e_j, e_k) = \delta_{jk}$. The $L^2$ inner product on functions is  
\[
(u, v) = \int_X u \bar{v} \,dV_g,
\]
and a corresponding (non-definite) bilinear form on $1$-forms is given by 
\[
(\alpha, \beta)_g = \int_X g(\alpha, \bar{\beta}) \,dV_g.
\]
The codifferential $\delta_g$ on $1$-forms is defined via 
\[
(\delta_g \alpha, v) = (\alpha, dv)_g, \qquad v \in C^{\infty}_c(X^{\mathrm{int}}),
\]
and the wave operator is $\Box_g = \delta_g d$. In local coordinates 
\begin{align*}
dV_g &= |g|^{1/2} \,dx, \\
\delta_g \alpha &= - |g|^{-1/2} \p_j( |g|^{1/2} g^{jk} \alpha_k), \\
\Box_g u &= -|g|^{-1/2} \p_j( |g|^{1/2} g^{jk} \p_k u),
\end{align*}
where $(g^{jk})$ is the inverse matrix of $(g_{jk})$ and $|g| = |\det(g_{jk})|$.

Next we discuss integration by parts and boundary terms. If $p \in \p X$ is such that the boundary is smooth near $p$, there is a boundary defining function $\rho \in C^{\infty}(X)$ such that near $p$ one has $X = \{ \rho \geq 0 \}$, $\p X = \{ \rho = 0 \}$, and $d\rho(p) \neq 0$. Then $-d\rho$ is an outer conormal of $X$ near $p$, and any other outer conormal is a positive multiple of $-d\rho$. At points where $g(d\rho, d\rho) \neq 0$, i.e.\ $\p X$ is noncharacteristic for $\Box_g$, one could use the Lorentzian metric $g$ to fix a unit conormal $-d\rho/|g(d\rho, d\rho)|^{1/2}$. However, to have a uniquely defined outer conormal also at characteristic points, it is convenient to introduce an auxiliary Riemannian metric on $X$. This was done in Lemma \ref{lemma_ht_riemannian}, which we now prove.

\begin{proof}[Proof of Lemma \ref{lemma_ht_riemannian}]
The decomposition $\xi = \lambda T + \xi_{\perp}$ with $g(T, \xi_{\perp}) = 0$ is obtained by taking $\lambda = g(T,\xi)/g(T,T)$. Since $T$ is timelike and $g(T, \xi_{\perp}) = 0$, we must have $g(\xi_{\perp}, \xi_{\perp}) > 0$ unless $\xi_{\perp} = 0$. Thus $g_T(\xi, \xi) > 0$ unless $\xi = 0$.

Since $(X,g)$ is oriented, we have $dV_{g_T} = f \,dV_g$ for some positive function $f$. Let $T^{\sharp}$ be the vector corresponding to $T$ via $g$, and let $(e_0, \ldots, e_n)$ be a positively oriented basis of $T_x X$ with $e_0 = T^{\sharp}/|g(T,T)|^{1/2}$ and $g(e_j, e_k) = \delta_{jk}$ if $(j, k) \neq (0, 0)$. From the definition of $g_T$ we see that $g_T(e_j, e_k) = \delta_{jk}$, which implies that $dV_{g_T}(e_0, \ldots, e_n) = 1 = dV_g(e_0, \ldots, e_n)$ so $f = 1$.
\end{proof}

At smooth points $p \in \p X$, we define $\nu$ to be the unique outer conormal with $|\nu|_{g_T} = 1$, and equip $\p X$ with the volume form $dS$ induced by $g_T$. In view of $dV_{g_T} = dV_g$, it is given by 
\[
dS(v_0, \ldots, v_{n-1}) = dV_g(v_0, \ldots, v_{n-1}, -\nu^{\sharp}), \qquad v_0, \ldots, v_{n-1} \in T_p \p X,
\]
where $\nu^{\sharp}$ is the vector corresponding to $\nu$ with respect to $g_T$. If $(x',x^n)$ are local coordinates so that $x^n$ is a local boundary defining function, we have $\nu = -dx^n/|dx^n|_{g_T}$. If $(\p_0, \ldots, \p_n)$ are the coordinate vector fields, then 
    \begin{align*}
dS(\p_0, \ldots, \p_{n-1}) 
&= 
dV_g(\p_0, \ldots, \p_{n-1}, \frac{1}{|dx^n|_{g_T}} g_T^{nj} \p_j) 
\\&= 
\frac{g^{nn}_T}{|dx^n|_{g_T}} dV_g(\p_0, \ldots, \p_{n-1}, \p_n) 
= 
|dx^n|_{g_T} |g(x',0)|^{1/2},
    \end{align*}
where we used $g^{nn}_T = |dx^n|_{g_T}^2$. Therefore 
\begin{equation} \label{lorentzian_ds_expression}
dS = |dx^n|_{g_T} |g(x',0)|^{1/2} \,dx'.
\end{equation}
A computation in the $(x',x_n)$ coordinates for $u, \alpha$ supported near $p$ gives 
\[
(du, \alpha)_g = \int_X g^{jk} \p_j u \bar{\alpha}_k |g|^{1/2} \,dx = -\int_{\p X} u g^{nk} \bar{\alpha}_k |g|^{1/2} \,dx' + (u, \delta_g \alpha)_g.
\]
This can be written invariantly as 
\begin{equation*} 
(du, \alpha)_g = (u, \delta_g \alpha)_g + (u \nu, \alpha)_{g, \p X}.
\end{equation*}
Therefore, with the notation  
\[
\tilde{\p}_{\nu} u := g(du, \nu),
\]
we also have 
\begin{equation*} 
(\Box_g u, v) = (du, dv)_g - (\tilde{\p}_{\nu} u, v)_{\p X}
\end{equation*}
where $(u, v)_{\p X} = \int_{\p X} u \bar{v} \,dS$.

\section{Degree theory facts} \label{degree_theory}

The following surjectivity result is used in this article.

\begin{Lemma} \label{lemma_idext_surjective}
Suppose that $\Psi: \mR^n \to \mR^n$ is a $C^1$ map with $\Psi = \id$ outside a compact set. Then $\Psi(\R^n) = \R^n$.
\end{Lemma}

We will also need the following version of stability of diffemorphisms \cite[\S1.6]{guillemin1974}.

\begin{Lemma} \label{lemma_diffeo_perturbation}
Let $U, V$ be open in $\mR^n$ and $\Phi: U \to V$ a $C^1$ diffeomorphism. Suppose that $K \subset V$ is compact and $W \subset \subset U$ is open with $K \subset \Phi(W)$. There is $\delta > 0$ such that any $C^1$ map $\Psi: W \to \mR^n$ with 
\begin{equation} \label{phi_delta_stability}
\norm{\Psi - \Phi}_{C^1(\ol{W})} \leq \delta
\end{equation}
is a $C^1$ diffeomorphism from $W$ onto an open set containing $K$. Moreover, 
\[
\|\Psi^{-1} - \Phi^{-1}\|_{C^1(K)} \to 0, \quad \delta \to 0.
\]
\end{Lemma}

Both results use basic degree theory. Let $W \subset \R^n$ be a bounded open set, $f \in C(\overline{W}, \R^n)$, and $y \notin f(\partial W)$. The topological degree $\deg(f, W, y) \in \mathbb Z$ is a homotopy invariant: if $f_t$, $t \in [0,1]$, is a continuous family of maps in $C(\overline{W};\R^n)$ satisfying
$y \notin f_t(\partial W)$ for all $t \in [0,1]$, then
$\deg(f_t, W, y)$ is constant in $t$, see e.g. \cite[Theorem~3.1]{deimling1985}. In the case that $f \in C^1(W)$ and $y$ is a regular value of $f$, the topological degree can be written 
    \begin{align*}
\deg(f, W, y) = \sum_{x \in f^{-1}(y)} \operatorname{sgn}
\det df(x),
    \end{align*}
see \cite[Definition~2.1]{deimling1985}. In particular, $\deg(\id, W, y) = 1$ for every $y \in W$. Moreover, $\deg(f, W, y) \neq 0$ implies that 
there is $x \in W$ with $f(x) = y$.

\begin{proof}[Proof of Lemma \ref{lemma_idext_surjective}]
Let $y \in \mathbb{R}^n$. Set $W = B(0, R)$ where $R > 0$ is large enough so that $y \in W$ and $\Psi = \id$ outside $W$.
Since $\Psi = \id$ on $\p W$, we have $y \notin \Psi(\partial W)$.
The homotopy $f_t = (1-t)\id + t\Psi$, $t \in [0,1]$,
satisfies $f_t = \id$ on $\partial W$,
so $y \notin f_t(\partial W)$. By homotopy invariance,
\[
\deg(\Psi, W, y)
= \deg(\mathrm{id}, W, y)
= 1.
\]
Since $\deg(\Psi, W, y) \ne 0$, there is $x \in W$ with $\Psi(x) = y$.
\end{proof}

\begin{proof}[Proof of Lemma \ref{lemma_diffeo_perturbation}]
By slightly enlarging $W$ if needed, we may assume that $W$ has smooth boundary (e.g.\ consider $W_1 = \{ \rho < \eps \}$ where $\rho$ is a smoothed out version of $\mathrm{dist}(\,\cdot\,, \ol{W})$). Since $\Phi$ is a diffeomorphism near $\ol{W}$, $D \Phi$ is invertible near $\ol{W}$ and by \eqref{phi_delta_stability} the same must be true for $\Psi$ when $\delta$ is small enough:
\[ 
|D\Psi(x)v| \geq c |v|, \qquad x \in \ol{W}, \ v \in \mR^n.
\] 
Thus $\Psi$ is a local diffeomorphism near any point of $\ol{W}$.

To see that $\Psi$ is injective near $\ol{W}$, for any $x,y \in \ol{W}$ we have 
\[
|\Psi(x)-\Psi(y)| \geq |\Phi(x)-\Phi(y)| - |(\Psi-\Phi)(x) - (\Psi-\Phi)(y)|.
\]
Now $\Phi$ is a diffeomorphism near $\ol{W}$ so 
\[
|\Phi(x)-\Phi(y)| \geq c|x-y| \text{ for $x,y \in \ol{W}$.}
\]
Moreover, since $W$ has smooth boundary, any points $x,y \in \ol{W}$ can be joined by a curve $\kappa$ in $\ol{W}$ with length $\leq C|x-y|$ (see e.g.\ \cite[(2.132)]{BrudnyiBrudnyi2012}). Therefore by \eqref{phi_delta_stability}
\begin{align*}
|(\Psi-\Phi)(x) - (\Psi-\Phi)(y)| &= \left| \int_0^1 \p_s \left[ (\Psi-\Phi)(\kappa(s)) \right] \,ds \right| \\
 &\leq (\sup_{\ol{W}} |D(\Psi-\Phi)|) C |x-y| \leq C \delta |x-y|.
\end{align*}
Combining these facts shows that $|\Psi(x)-\Psi(y)| \geq c|x-y|$ for $x,y \in \ol{W}$ when $\delta$ is small, and thus $\Psi$ is injective on $\ol{W}$.

Next we show that $K \subset \Psi(W)$ when $\delta$ is small enough. Fix $y \in K$. Since $\Phi$ is a $C^1$ diffeomorphism we have $y \notin \Phi(\p W)$. Consider the homotopy $f_t = (1-t)\Phi + t \Psi$, $t \in [0,1]$. Since $\p W$ has positive distance from $K$, the condition $\norm{\Psi - \Phi}_{C(\ol{W})} \leq \delta$ with $\delta$ small enough yields $y \notin f_t(\p W)$ for $t \in [0,1]$. By homotopy invariance,
\[
\deg(\Psi, W, y)
= \deg(\Phi, W, y).
\]
Since $\Phi$ is a $C^1$ diffeomorphism, $\deg(\Phi, W, y) = \pm 1$. Thus $\deg(\Psi, W, y) \neq 0$, so $\Psi(x) = y$ for some $x \in W$.

Finally, to show the convergence of inverses in $C^1(K)$, by writing $\Psi^{-1} = \Psi^{-1} \circ \Phi \circ \Phi^{-1}$ we can reduce matters to the case where $\Phi = \id$. If $y \in K$ and $y = \Psi(x)$ with $x \in W$, then 
\[
|\Psi^{-1}(y) - y| = |x - \Psi(x)| \leq \norm{\Psi-\id}_{C(\ol{W})}
\]
and 
\begin{align*}
|D(\Psi^{-1})(y) - \id| &= |(D\Psi)(x)^{-1} - \id| \\
 &= |(D\Psi)(x)^{-1}(\id - D\Psi(x))| \lesssim \norm{\Psi-\id}_{C^1(\ol{W})}. \qedhere
\end{align*}
\end{proof}

\section{Proof of Proposition \ref{prop_semiclassical_estimate}} \label{sec_carleman_appendix}

In this section we will prove Proposition \ref{prop_semiclassical_estimate}, which gives Proposition \ref{prop_carleman_psi} as a corollary. For the proof, we need several notations (see Appendix \ref{lorentzian_integration} for more details). Let $\delta_g$ be the formal adjoint of $d$ with respect to the bilinear form 
\[
(\alpha, \beta)_g = \int_Q g(\alpha, \bar{\beta}) \,dV_g.
\]
Then $\Box_g = \delta_g d$. Denote by $\nu$ the unique outer conormal of $\p Q$ (at points where $\p Q$ is smooth) that satisfies $|\nu|_{g_T} = 1$, and write 
\[
\tilde{\p}_{\nu} u|_{\p Q} = g(du, \nu)|_{\p Q}, \qquad \p_{\nu} u|_{\p Q} = g_T(du, \nu)|_{\p Q}.
\]
We will use frequently the integration by parts formula 
\[
(du, \alpha)_g = (u, \delta_g \alpha)_g + (u \nu, \alpha)_{g, \p Q}
\]
where $(\alpha, \beta)_{g, \p Q} = \int_{\p Q} g(\alpha, \bar{\beta}) \,dS$.

We will break the proof of Proposition \ref{prop_semiclassical_estimate} into a number of lemmas. First we decompose $P = e^{\psi/h} \circ h^2 \Box_g \circ e^{-\psi/h}$ into its self-adjoint and skew-adjoint parts.

\begin{Lemma} \label{lemma_p_decomposition_a_ib}
$P = A + i B$ with 
\begin{align*}
A &= h^2 \Box_g - g(d\psi, d\psi), \\
B &= 2 g(d\psi, hD\,\cdot\,) + i h (\Box_g \psi),
\end{align*}
where $D = (1/i) d$ and $A$ and $B$ are formally self-adjoint.
\end{Lemma}
\begin{proof}
Write in local coordinates $D_j = (1/i) \p_j$. Since $h D_j(e^{-\psi/h} u) = e^{-\psi/h}(h D_j + i \p_j \psi) u$, we have 
\begin{align*}
Pu &= e^{\psi/h} |g|^{-1/2} hD_j(|g|^{1/2} hD_k (e^{-\psi/h} u)) \\
 &= |g|^{-1/2} (hD_j + i \p_j \psi)(|g|^{1/2} g^{jk} (hD_k + i \p_k \psi) u ) \\
 &= h^2 \Box_g u - g(d\psi, d\psi) u + i \left[ \p_j \psi g^{jk} hD_k u + |g|^{-1/2} h D_j (|g|^{1/2} g^{jk} (\p_k \psi) u) \right] \\
 &= h^2 \Box_g u - g(d\psi, d\psi) u + i \left[ 2 g(d\psi, hD u) + ih (\Box_g \psi) u \right].
\end{align*}
Now $A$ is formally self-adjoint since $\Box_g = \delta_g d$, and $B$ is self-adjoint by a local coordinate computation.
\end{proof}

The next lemma gives an initial expression for $\norm{Pu}^2$ with $u \in C^{\infty}_c(X^{\mathrm{int}})$.

\begin{Lemma}
For any $u \in C^{\infty}_c(X^{\mathrm{int}})$, 
\begin{multline} \label{pu_square_first}
\norm{Pu}^2 = \norm{Au}^2 + \norm{Bu}^2 + ( i [A, B] u, u ) \\
 + h (i h\tilde{\p}_{\nu} Bu + b(x,\nu) Au, u)_{\p Q} - ih(Bu, h \tilde{\p}_{\nu} u)_{\p Q},
\end{multline}
where $b(x,\xi) = 2 g(d\psi, \xi)$.
\end{Lemma}
\begin{proof}
We have 
\begin{align*}
\norm{Pu}^2 &= ((A+iB)u, (A+iB)u) \\
 &= \norm{Au}^2 + \norm{Bu}^2 + i (Bu, Au) - i (Au, Bu).
\end{align*}
By Lemma \ref{lemma_p_decomposition_a_ib} we can write 
\[
A = A_2 + A_0, \qquad A_2 = h^2 \delta_g d, \qquad A_0 = -g(d\psi,d\psi).
\]
Integration by parts gives 
\begin{align*}
i(Bu, A_2 u) &= i(hdBu, hdu) - ih(Bu, h \tilde{\p}_{\nu} u)_{\p Q} \\
 &= i(A_2 Bu, u) + ih (h\tilde{\p}_{\nu} Bu, u)_{\p Q} - ih(Bu, h \tilde{\p}_{\nu} u)_{\p Q}
\end{align*}
and, since $B$ is of first order with principal symbol $b(x,\xi)$,  
\[
-i(Au, Bu) = -i(BAu, u) + h (Au, b(x,\nu) u)_{\p Q}.
\]
This proves \eqref{pu_square_first}.
\end{proof}

Next we study the interior terms in \eqref{pu_square_first}, and for this we need some properties of the commutator $i[A, B]$. Below we use the gradient $\nabla$ and Hessian $D^2 \psi$ with respect to the Lorentzian metric $g$. In local coordinates (see \cite{oneill1983})
\[
D^2 \psi = (\p_{jk} \psi - \Gamma_{jk}^l \p_l \psi) \,dx^j \otimes dx^k.
\]

\begin{Lemma} \label{lemma_pu_commutator}
$i[A, B] = h S$, where $S$ is a second order semiclassical operator with principal symbol 
\[
s(x, \xi) = 4 D^2 \psi(\xi^{\sharp}, \xi^{\sharp}) + 4 D^2 \psi( \nabla \psi, \nabla \psi )
\]
and $\xi^{\sharp}$ is the vector corresponding to $\xi$ via $g$. For any $u \in C^{\infty}_c(X^{\mathrm{int}})$, 
\[
(Su, u) = -4 h^2 (D^2 \psi(\nu^{\sharp}, \nu^{\sharp}) \p_{\nu} u, u)_{\p Q} + O(\norm{u}^2 + \norm{hdu}^2).
\]
\end{Lemma}
\begin{proof}
We refer to \cite{dos-santos-ferreira2009} for a derivation 
of the expression for $s(x,\xi)$. There the derivation is given in the Riemannian case, but the same proof works in the Lorentzian case.

The formula for $(Su, u)$ is obtained by integrating by parts once. By using cutoff functions, it is enough to consider the boundary terms for $u$ supported close to $\Gamma$ or $\Gamma_R$.
Suppose now that $u$ is supported close to $\Gamma$, and $(x',x^n)$ are coordinates near $\Gamma$ such that $\Gamma = \{ x^n = 0 \}$ and $dx^n|_{\Gamma} = -\nu$. Then near $\Gamma$, $s = s^{jk} \xi_j \xi_k + f$ with $s^{nn}|_{\Gamma} = 4 D^2 \psi(\nu^{\sharp}, \nu^{\sharp})$, and  
\[
Su = s^{nn} (hD_n)^2 u + Q_1 hD_n u + Q_2 u
\]
where $Q_j$ are tangential semiclassical operators of order $j$, depending smoothly on $x_n$. Since  $u$ has zero Cauchy data on $\p X$, integration by parts in the terms containing $Q_j$, $j=1,2$, gives
\[
(Su, u) = -h^2 (s^{nn} \p_n^2 u, u) + O(\norm{u}^2 + \norm{hdu}^2).
\]
On $\Gamma$, the choice $dx^n = -\nu$ implies $|dx^n|_{g_T} = 1$ and $(g_T)^{nj} = \delta^{nj}$. Hence
    \begin{align*}
\p_{\nu} u|_{\Gamma} = -g_T(du, dx^n) = -(g_T)^{nj} \p_j u = -\p_n u,
    \end{align*}
and the induced volume form satisfies $dS = |g|^{1/2} dx'$ in view of \eqref{lorentzian_ds_expression}. Now
\begin{align*}
(s^{nn} \p_n^2 u, u) &= -(s^{nn} \p_n u, u)_{\Gamma} + O(\norm{u}^2 + \norm{hdu}^2) \\
 &= (s^{nn} \p_{\nu} u, u)_{\Gamma} + O(\norm{u}^2 + \norm{hdu}^2).
\end{align*}
This proves the result for $u$ supported close to $\Gamma$. The case of $u$ supported close to $\Gamma_R$ is analogous.
\end{proof}

The following result gives an expression for the interior terms in \eqref{pu_square_first}.

\begin{Lemma} \label{lemma_pu_interior_terms}
For any $u \in C^{\infty}_c(X^{\mathrm{int}})$, 
\begin{multline} \label{pu_interior_terms_first}
\norm{Au}^2 + \norm{Bu}^2 + ( i [A, B] u, u) \\
 = \norm{m_T^2 u}^2 + 2 (m_T^2 hdu, hdu) + \norm{A_2 u}^2 - 2 h \mathrm{Re}( h \tilde{\p}_{\nu} u, A_0 u)_{\p Q} \\
  - 4 h^3 (D^2 \psi(\nu^{\sharp}, \nu^{\sharp}) \p_{\nu} u, u)_{\p Q}  + O(h \norm{u}^2 + h \norm{hdu}^2).
\end{multline}
\end{Lemma}
\begin{proof}
Since $A = A_2 + A_0$ with $A_2 = h^2 \delta_g d$, $A_0 = m_T^2$, we have 
\[
\norm{Au}^2 + \norm{Bu}^2 = \norm{A_2 u}^2 + \norm{A_0 u}^2 + (A_2 u, A_0 u) + (A_0 u, A_2 u) + \norm{Bu}^2.
\]
Integrating by parts in the cross terms yields 
\begin{align*}
 &(A_2 u, A_0 u) + (A_0 u, A_2 u) \\
 &= (hd u, hd(A_0 u))_g + (hd(A_0 u), hdu)_g - h( h \tilde{\p}_{\nu} u, A_0 u)_{\p Q} - h(A_0 u, h \tilde{\p}_{\nu} u)_{\p Q} \\
 &= 2 (A_0 hdu, hdu)_g - 2 h \mathrm{Re}( h \tilde{\p}_{\nu} u, A_0 u)_{\p Q} + O(h \norm{g(dA_0, hdu)} \norm{u}).
\end{align*}
Moreover, by the formula for $B$ and the definition of $g_T$ we have 
\begin{align*}
 &2 (A_0 hdu, hdu)_g + \norm{Bu}^2 \\
 &= 2 (m_T^2 hdu, hdu)_g + 4 \norm{g(T, hdu)}^2 + O(h \norm{g(T,hdu)} \norm{u} + h^2 \norm{u}^2) \\
 &= 2 (m_T^2 hdu, hdu) + O(h \norm{g(T,hdu)} \norm{u} + h^2 \norm{u}^2).
\end{align*}
For the commutator term, we use Lemma \ref{lemma_pu_commutator} to obtain 
\[
( i [A, B] u, u) = -4 h^3 (D^2 \psi(\nu^{\sharp}, \nu^{\sharp}) \p_{\nu} u, u)_{\p Q} + O(h \norm{u}^2 + h \norm{hdu}^2)
\]
Combining these facts gives the estimate \eqref{pu_interior_terms_first} for the interior terms.
\end{proof}

Next we study the combined boundary terms in \eqref{pu_square_first}--\eqref{pu_interior_terms_first}. They are given by 
\begin{multline} \label{pu_boundary_terms_combined}
h (i h\tilde{\p}_{\nu} Bu + b(x,\nu) Au, u)_{\p Q} - ih(Bu, h \tilde{\p}_{\nu} u)_{\p Q} \\
 - 2 h \mathrm{Re}( h \tilde{\p}_{\nu} u, A_0 u)_{\p Q} - 4 h^3 (D^2 \psi(\nu^{\sharp}, \nu^{\sharp}) \p_{\nu} u, u)_{\p Q}.
\end{multline}
Note that $u$ has zero Cauchy data on $\p X$, so it is enough to study boundary terms over $\Gamma$ and $\Gamma_R$ separately.

We first focus on the terms on $\Gamma_R$.

\begin{Lemma} \label{lemma_pu_gammar_terms}
On $\Gamma_R$, the boundary terms in  \eqref{pu_boundary_terms_combined} take the form 
\begin{multline} \label{pu_gammar_terms}
- 2h \norm{m_T^{1/2} h d u}_{\Gamma_R}^2 - 2 h \norm{m_T^{3/2} u}_{\Gamma_R}^2  \\
 + 2 h \mathrm{Re}( h \p_{\nu} u, m_T^2 u)_{\Gamma_R} + O(h^2 \norm{u}_{\Gamma_R} \norm{h d u}_{\Gamma_R}).
\end{multline}
\end{Lemma}
\begin{proof}
On $\Gamma_R$, the outer conormal $\nu$ is proportional to $T$. 
Since $|T|_{g_T} = m_T$,  
    \begin{align}\label{Gamma_R_nu}
\nu = m_T^{-1} T.
    \end{align}
The operator $ih \tilde{\p}_{\nu} B + b(x,\nu) A$ has principal symbol 
\begin{equation*} 
-g(\nu, \xi) 2 g(T, \xi) + 2 g(T, \nu) (g(\xi, \xi) - g(T,T)).
\end{equation*}
Since $\nu = m_T^{-1} T$,  using the notation of Lemma \ref{lemma_ht_riemannian} with $\xi = \lambda T + \xi_{\perp}$ and $g(T, \xi_{\perp}) = 0$, we have 
\begin{align*}
&-g(\nu, \xi) 2 g(T, \xi) + 2 g(T, \nu) g(\xi, \xi) = 2 g(T, \nu)(g(\xi, \xi) - \frac{g(T,\xi)^2}{g(T,T)}) \\
 &\qquad=2 g(T, \nu) g(\xi_{\perp}, \xi_{\perp}) 
= -2 m_T g_T(\xi_{\perp}, \xi_{\perp}).
\end{align*}
It follows from $T^* \Gamma_R = \{ \xi_\perp \in T^* X \mid g(T, \xi_\perp) = 0\}$ that the principal part of $ih \tilde{\p}_{\nu} B + b(x,\nu) A$ is a tangential operator, and 
\begin{multline*}
h (i h\tilde{\p}_{\nu} Bu + b(x,\nu) Au, u)_{\Gamma_R} \\
 = -2 h (m_T h d_{\Gamma_R} u, h d_{\Gamma_R} u)_{\Gamma_R} - 2 h (m_T^3 u, u)_{\Gamma_R} + O(h^2 \norm{u}_{\Gamma_R} \norm{h d u}_{\Gamma_R}).
\end{multline*}

Furthermore, using \eqref{Gamma_R_nu} and
$\tilde{\p}_{\nu} u = g(\nu, du) = -g_T(\nu, du) = -\p_{\nu} u$
we have 
\[
Bu = 2 g(T, hDu) + ih(\Box_g \psi) u = -2 m_T (h/i) \p_{\nu} u + ih(\Box_g \psi) u.
\]
Thus on $\Gamma_R$,
the second term in \eqref{pu_boundary_terms_combined} is  
\[
 -ih(Bu, h \tilde{\p}_{\nu} u)_{\Gamma_R} = - 2 h (m_T h \p_{\nu} u, h \p_{\nu} u)_{\Gamma_R} + O(h^2 \norm{u}_{\Gamma_R} \norm{h \p_{\nu} u}_{\Gamma_R}).
\]
The third term in \eqref{pu_boundary_terms_combined} is 
\[
2h \mathrm{Re}(h \p_{\nu} u, m_T^2 u)_{\Gamma_R},
\]
and the last term in \eqref{pu_boundary_terms_combined} is $O(h^2 \norm{u}_{\Gamma_R} \norm{h \p_{\nu} u}_{\Gamma_R})$. This proves \eqref{pu_gammar_terms}.
\end{proof}

It remains to study the boundary terms \eqref{pu_boundary_terms_combined} on $\Gamma$. The main point is that $\Gamma$ is characteristic, i.e.\ 
\[
g(\nu, \nu)|_{\Gamma} = 0.
\]
We will employ suitable coordinates $(x', x^n)$ near $\Gamma$ to study subprincipal terms. Let $x' = (x^0, \ldots, x^{n-1})$ be coordinates on $\Gamma$ and let $x^n = \varphi-s$. Thus $\Gamma = \{ x^n = 0 \}$ and, writing $\mu = |dx^n|_{g_T}^{-1}$, we have $\nu = -\mu \,dx^n|_{\Gamma}$.
By \eqref{lorentzian_ds_expression} the volume form on $\Gamma$ is 
\begin{equation} \label{lorentzian_ds_expression_second}
dS = \mu^{-1} |g|^{1/2} \,dx'.
\end{equation}
Since $g(d\varphi, d\varphi) = 0$, we have the important property that 
\[
g^{nn} = 0 \text{ near $\Gamma$}.
\]
We raise and lower indices with respect to $g$ and let $\alpha$ and $\beta$ run from $0$ to $n-1$. For instance, on $\Gamma$,
    \begin{align}\label{tilde_p_nu_Gamma}
\tilde{\p}_{\nu} = g^{jk} \nu_j \p_k = -\mu \p^n = -\mu g^{n\alpha} \p_{\alpha}.
    \end{align}

First we note an integration by parts identity.

\begin{Lemma} \label{lemma_tildepnu_gamma_tangential}
The operator $\tilde{\p}_{\nu}$ is tangential on $\Gamma$, and for $v, w \in H^1_0(\Gamma)$ we have the integration by parts formula 
\begin{equation*}
(\tilde{\p}_{\nu} v, w)_{\Gamma} =  -(v, \tilde{\p}_{\nu} w)_{\Gamma} + (\mu \theta^n v, w)_{\Gamma}
\end{equation*}
where $\theta^k = |g|^{-1/2} \p_j(|g|^{1/2} g^{jk}) = -\Box_g x^k$.
\end{Lemma}
\begin{proof}
Tangentiality follows from \eqref{tilde_p_nu_Gamma}. By \eqref{lorentzian_ds_expression_second} we have 
\begin{equation*}
(\tilde{\p}_{\nu} v, w)_{\Gamma} =  - \int_{\Gamma} g^{n\alpha} \p_{\alpha} v \bar{w} |g|^{1/2} \,dx' = -(v, \tilde{\p}_{\nu} w)_{\Gamma} + (\mu \theta^n v, w)_{\Gamma}
\end{equation*}
where $\theta^n$ has the required form.
\end{proof}

The next result describes the boundary terms on $\Gamma$.

\begin{Lemma} \label{lemma_pu_gamma_terms}
On $\Gamma$, the boundary terms in  \eqref{pu_boundary_terms_combined} take the form 
\begin{multline} \label{pu_gamma_terms}
\sqrt{2} h (\norm{m_T^{3/2} u}_{\Gamma}^2 + \norm{m_T^{1/2} h d_{\Gamma} u}_{\Gamma}^2) + 2 \sqrt{2} h \norm{m_T^{1/2} g_T(\hat{T}, hd_{\Gamma} u)}_{\Gamma}^2 \\
 + O(h^2 \norm{u}_{\Gamma}^2 + h^2 \norm{hd_{\Gamma} u}_{\Gamma}^2).
\end{multline}
\end{Lemma}
\begin{proof}
Lemma \ref{lemma_tildepnu_gamma_tangential} implies that the first two terms in \eqref{pu_boundary_terms_combined} may be written as 
\begin{multline*}
h (i h\tilde{\p}_{\nu} Bu + b(x,\nu) Au, u)_{\p Q}  -ih(Bu, h \tilde{\p}_{\nu} u)_{\Gamma} \\
 = h (2 i h\tilde{\p}_{\nu} Bu + b(x,\nu) Au, u)_{\p Q}  - ih^2 (Bu, \mu \theta^n u)_{\Gamma}.
\end{multline*}
The operator $2 ih \tilde{\p}_{\nu} B + b(x,\nu) A$ on $\Gamma$ has the form  
\begin{align*}
 &\frac{1}{2 h^2} (2 ih \tilde{\p}_{\nu} B u + b(x,\nu) A u)  \\
 &= -\mu \p^n(2 T^j \p_j u -  (\Box_g \psi) u) - \mu T^n(-g^{jk} \p_{jk} u - \theta^k \p_k u + h^{-2} m_T^2 u)  \\
 &= \mu (T^n g^{jk} - 2 T^j g^{nk}) \p_{jk} u + \mu (T^n \theta^k  - 2 \p^n T^k + g^{nk} \Box_g \psi)\p_k u  \\
 & \qquad -h^{-2} \mu T^n m_T^2 u + fu
\end{align*}
where $f = \mu \p^n \Box_g \psi$. Since the $\p_{\alpha n} u$ terms cancel out, and since the $\p_{nn} u$ terms disappear due to $g^{nn} = 0$, this reduces to 
\[
E u + \mu (T^n \theta^n - 2 \p^n T^n) \p_n u - h^{-2} \mu T^n m_T^2 u + Q u
\]
where $E = \mu (T^n g^{\alpha \beta} - 2 T^{\alpha} g^{n\beta}) \p_{\alpha \beta}$ and $Q$ is a first order tangential operator independent of $h$. 
Moreover, 
    \begin{align*}
- ih^2 (Bu, \mu \theta^n u)_{\Gamma}
&= -h^3(2 T^j \p_j u - (\Box_g \psi) u, \mu \theta^n u)_{\Gamma}
\\&= -2h^3(\mu T^n \theta^n \p_n u, u)_{\Gamma} + h^3(Q' u, u)_\Gamma,
    \end{align*}
where $Q'$ is another first order tangential operator independent of $h$. Summarizing,
\begin{align} \label{pu_bt_gamma_intermediate}
&h (i h\tilde{\p}_{\nu} Bu + b(x,\nu) Au, u)_{\p Q}  -ih(Bu, h \tilde{\p}_{\nu} u)_{\Gamma}
\\\notag&\qquad= 2h^3 (Eu - 2 \mu (\p^n T^n) \p_n u - h^{-2} \mu T^n m_T^2 u + \tilde{Q} u, u)_{\Gamma},
\end{align}
where the two terms containing $\theta^n$ canceled out and $\tilde Q = Q + Q'/2$.

Importantly, $E$ is an elliptic operator on $\Gamma$. Indeed, writing $N = \mu \,dx^n$, the symbol of $E$ is a quadratic form in $\xi \in T^* \Gamma$ given by 
\[
e(x,\xi) = -g(N, T) g(\xi, \xi) + 2 g(T, \xi) g(N, \xi).
\]
It is a variant of the quadratic form in \cite[Lemma 24.1.2]{hormander}. Since 
    \begin{align*}
\xi \in T^* \Gamma = \{ \eta \,:\, g_T(N,\eta) = 0 \},
    \end{align*}
we have $g(N, \xi) = 2 g(T, \xi) g(T, N) / g(T,T)$. Thus 
\begin{align*}
e(x,\xi) 
&= 
 - g(N, T) g(\xi, \xi) + 
\frac{4 g(T, N) g(T, \xi)^2}{g(T,T)} 
\\
 &= -g(N, T) g_T(\xi,\xi) + \frac{2 g(T, N)}{g(T,T)} g(T, \xi)^2.
\end{align*}
On $\Gamma$ we have $g_T(N,N) = 1$ and $g(N,N) = 0$, which gives 
\[
1 = g_T(N,N) = -2 \frac{g(N,T)^2}{g(T,T)} \quad \implies \quad g(N,T) = -\frac{1}{\sqrt{2}} m_T.
\]
Thus 
\[
e(x,\xi) = \frac{1}{\sqrt{2}} m_T g_T(\xi, \xi) + \sqrt{2} m_T g_T(\hat{T}, \xi)^2
\]
with $\hat{T} = T/|T|_{g_T}$. Quantizing this expression gives that 
\begin{multline*}
h^2 (Eu, u)_{\Gamma} = \frac{1}{\sqrt{2}} \norm{m_T^{1/2} hd_{\Gamma} u}_{\Gamma}^2 + \sqrt{2} \norm{m_T^{1/2} g_T(\hat{T}, hd_{\Gamma} u)}_{\Gamma}^2 \\
 + O(h \norm{hd_{\Gamma} u}_{\Gamma} \norm{u}_{\Gamma}).
\end{multline*}
Going back to \eqref{pu_bt_gamma_intermediate} and using 
    \begin{align*}
\mu T^n = g(N, T) = -m_T/\sqrt{2}
    \end{align*}
we have 
\begin{multline} \label{pu_bt_gamma_intermediate_second}
h (i h\tilde{\p}_{\nu} Bu + b(x,\nu) Au, u)_{\p Q}  -ih(Bu, h \tilde{\p}_{\nu} u)_{\Gamma} \\
 = \sqrt{2} h \norm{m_T^{1/2} h d_{\Gamma} u}_{\Gamma}^2 + \sqrt{2} h \norm{m_T^{3/2} u}_{\Gamma}^2 + 2 \sqrt{2} h \norm{m_T^{1/2} g_T(\hat{T}, hd_{\Gamma} u)}_{\Gamma}^2  \\
  - 4 h^3 (\mu (\p^n T^n) \p_n u, u)_{\Gamma} + O(h^2  \norm{u}_{\Gamma}^2 + h^2 \norm{hd_{\Gamma} u}_{\Gamma}^2).
\end{multline}

Let us rewrite the term $(\mu (\p^n T^n) \p_n u, u)_{\Gamma}$. In general, a function $A$ from $X$ to invertible matrices satisfies $\p_j(A^{-1}) = -A^{-1} (\p_j A) A^{-1}$. Hence  
\begin{align*}
\p^n T^n = g^{nj} \p_j(g^{nk} \p_k \psi) = g^{nj} \p_{jk} \psi g^{kn} - g^{nj} g^{np} \p_j g_{pq} g^{qk} \p_k \psi.
\end{align*}
On the other hand, since $g^{pq} g_{qr} = \delta^p_r$ and $g^{nn} = 0$ near $\Gamma$, we have  
\[
g^{nj} g^{nk} \p_m g_{jk} = -g^{nj} (\p_m g^{nk}) g_{jk} = -(\p_m g^{nk}) \delta^n_k = -\p_m g^{nn} = 0
\]
and therefore 
\[
g^{nj} g^{nk} \Gamma_{jk}^l = \frac{1}{2} g^{nj} g^{nk} g^{lm} (\p_j g_{km} + \p_k g_{jm} - \p_m g_{jk}) = g^{nj} g^{nk} g^{lm} \p_j g_{km}.
\]
It follows that 
\[
\p^n T^n = g^{nj} (\p_{jk} \psi - \Gamma_{jk}^l \p_l \psi) g^{kn} = D^2 \psi((dx^n)^{\sharp}, (dx^n)^{\sharp}) = \mu^{-2} D^2 \psi(\nu^{\sharp}, \nu^{\sharp}).
\]
We also have 
\begin{align*}
\p_{\nu} u|_{\Gamma} = g_T(du, -\mu \,dx^n) = -\mu (g_T)^{nj} \p_j u = - \mu^{-1} \p_n u
\end{align*}
since $(g_T)^{nj} = \mu^{-2} \delta^{nj}$ on $\Gamma$. After these simplifications \eqref{pu_bt_gamma_intermediate_second} becomes 
\begin{multline*}
h (i h\tilde{\p}_{\nu} Bu + b(x,\nu) Au, u)_{\p Q}  -ih(Bu, h \tilde{\p}_{\nu} u)_{\Gamma} \\
 = \sqrt{2} h \norm{m_T^{1/2} h d_{\Gamma} u}_{\Gamma}^2 + \sqrt{2} h \norm{m_T^{3/2} u}_{\Gamma}^2 + 2 \sqrt{2} h \norm{m_T^{1/2} g_T(\hat{T}, hd_{\Gamma} u)}_{\Gamma}^2  \\
  + 4 h^3 ((D^2 \psi)(\nu^{\sharp}, \nu^{\sharp}) \p_{\nu} u, u)_{\Gamma} + O(h^2  \norm{u}_{\Gamma}^2 + h^2 \norm{hd_{\Gamma} u}_{\Gamma}^2).
\end{multline*}
This is our final expression for the first two terms in \eqref{pu_boundary_terms_combined}.

The third term in \eqref{pu_boundary_terms_combined} is relatively harmless, since $\tilde{\p}_{\nu}$ is a tangential operator and we may use Lemma \ref{lemma_tildepnu_gamma_tangential} to obtain 
\[
- 2 h \mathrm{Re}( h \tilde{\p}_{\nu} u, A_0 u)_{\p Q} = h^2 \left[ ( u, (\tilde{\p}_{\nu} A_0)u )_{\Gamma} - (\mu \theta^n u, A_0 u)_{\Gamma} \right] = O(h^2 \norm{u}_{\Gamma}^2 ).
\]
The fourth term in \eqref{pu_boundary_terms_combined} is 
\[
- 4 h^3 (D^2 \psi(\nu^{\sharp}, \nu^{\sharp}) \p_{\nu} u, u)_{\p Q}.
\]
Combining the last three displayed equations proves the lemma.
\end{proof}

We may now prove the main estimate of this section.

\begin{proof}[Proof of Proposition \ref{prop_semiclassical_estimate}]
Combining \eqref{pu_square_first}--\eqref{pu_gammar_terms} and \eqref{pu_gamma_terms} gives
\begin{multline*}
\norm{m_T^2 u}^2 + 2 (m_T^2 hdu, hdu) + \norm{A_2 u}^2 + O(h \norm{u}^2 + h \norm{hdu}^2) \\
+ \sqrt{2} h (\norm{m_T^{3/2} u}_{\Gamma}^2 + \norm{m_T^{1/2} h d_{\Gamma} u}_{\Gamma}^2) + 2 \sqrt{2} h \norm{m_T^{1/2} g_T(\hat{T}, hd_{\Gamma} u)}_{\Gamma}^2 \\
 + O(h^2 \norm{u}_{\Gamma}^2 + h^2 \norm{hd_{\Gamma} u}_{\Gamma}^2) \\
 = \norm{Pu}^2 + 2h \norm{m_T^{1/2} h d u}_{\Gamma_R}^2 + 2 h \norm{m_T^{3/2} u}_{\Gamma_R}^2  \\
 - 2 h \mathrm{Re}( h \p_{\nu} u, m_T^2 u)_{\Gamma_R} + O(h^2 \norm{u}_{\Gamma_R} \norm{h d u}_{\Gamma_R}).
\end{multline*}
We drop the nonnegative $\norm{A_2 u}^2$ and $2 \sqrt{2} h \norm{m_T^{1/2} g_T(\hat{T}, hd_{\Gamma} u)}_{\Gamma}^2$ terms and use Cauchy-Schwarz in the last terms on the right to get the inequality 
\begin{multline*}
(1-Ch) \left[ \norm{m_T^2 u}^2 + 2 (m_T^2 hdu, hdu) + \sqrt{2} h \norm{m_T^{3/2} u}_{\Gamma}^2 +  \sqrt{2} h \norm{m_T^{1/2} h d_{\Gamma} u}_{\Gamma}^2 \right] \\
 \leq  \norm{Pu}^2 + (3+Ch) \left[ h \norm{m_T^{1/2} h d u}_{\Gamma_R}^2 + h \norm{m_T^{3/2} u}_{\Gamma_R}^2 \right].
\end{multline*}
This is the required statement.
\end{proof}

We will also phrase the above estimate in terms of the large parameter $\sigma = 1/h$.

\begin{Proposition} \label{prop_carleman_sigma_general}
Let $X \in L^{\infty}(Q, TQ)$ and $q \in L^{\infty}(Q)$. There is $\sigma_0 \geq 1$, only depending on $\norm{X}_{L^{\infty}}$, $\norm{q}_{L^{\infty}}$, $\inf_{Q} m_T$, $g$ and $\psi$, such that when $\sigma \geq \sigma_0$ we have the estimate 
\begin{align*}
\sigma^4 \norm{e^{\sigma \psi} m_T^2 u}^2 + \sigma^2 \norm{e^{\sigma \psi} m_T d u}^2 + \sigma^3 \norm{e^{\sigma \psi} m_T^{3/2} u}_{\Gamma}^2 + \sigma \norm{e^{\sigma \psi} m_T^{1/2} d_{\Gamma} u}_{\Gamma}^2 \\
 \leq 2 \norm{e^{\sigma \psi} (\Box_g + X + q) u}^2 + 8 \sigma^3 \norm{e^{\sigma \psi} m_T^{3/2} u}_{\Gamma_R}^2 + 8 \sigma \norm{e^{\sigma \psi} m_T^{1/2} d u}_{\Gamma_R}^2
\end{align*}
for any $u \in H^2(Q)$ with Cauchy data vanishing on $\p X$.
\end{Proposition}
\begin{proof}
This is obtained from Proposition \ref{prop_semiclassical_estimate} upon replacing $u$ by $e^{\sigma \psi} u$, adding and subtracting $X+q$ to $\Box_g$, using the triangle inequality, and absorbing errors by taking $\sigma \geq \sigma_0$ with $\sigma_0$ sufficiently large.
\end{proof}

Proposition \ref{prop_carleman_psi} follows immediately from Proposition \ref{prop_carleman_sigma_general}.

 \begin{Remark}
The estimate in Proposition \ref{prop_carleman_psi} is stronger than a typical Carleman estimate in terms of powers of $\sigma$. In fact, this estimate is semiclassically elliptic instead of being subelliptic. To explain this difference, let $p(x,\xi) = g(\xi, \xi)$ be the principal symbol of $\Box_g$, so that $p(x,\xi+i \sigma d\psi)$ is the semiclassical principal symbol of the conjugated operator $e^{\sigma \psi} \Box_g e^{-\sigma \psi}$ (with large parameter $\sigma$). As observed in Section \ref{sec_carleman}, $p(x,\xi+i \sigma d\psi)$ is never zero when $\sigma \neq 0$. 
A typical subellipticity condition for the conjugated operator would include the positivity of a certain Poisson bracket at points where $p(x,\xi+i \sigma d\psi)$ vanishes, and this is vacuously true in our case. In particular, the necessary conditions in \cite[Theorems 28.2.1 and 28.2.1']{hormander} hold.
\end{Remark}

\section{Proof of Theorem \ref{th_diffeo}} \label{sec_diffeo_appendix}

As shown in Section \ref{sec_directional_simplicity}, Proposition \ref{prop_simplicity_detection} follows from Theorem \ref{th_diffeo}. We will give a proof of this theorem after establishing a number of lemmas that follow by adapting the arguments in \cite[Section 3]{oksanen2024}. For notational simplicity, we suppose that $\omega = (1,0,\dots,0)$ in \eqref{def_Sigma_minus}.

\begin{Lemma}\label{lem_Jacobi}
Suppose that $g$ satisfies \eqref{assumption1}.
Then using $(t,y) \in \R \times \R^{n-1}$, with $y = (y^1,\dots,y^{n-1})$, as coordinates for a point $(t,-1,y) \in \Sigma_-$
we have
    \begin{align*}
g(\p_t \Phi^g, \p_r \Phi^g) = -1,
\quad
g(\p_{y_j} \Phi^g, \p_r \Phi^g) = 0.
    \end{align*} 
\end{Lemma}
\begin{proof}
We omit writing $g$ as superscript in the proof. 
Let $(t, y) \in \R \times \R^{n-1}$, set $z = (t, -1, y)$ and define the vector field $J(r) = \p_t \Phi(z, r)$ along the geodesic $\gamma(r) = \gamma_{z}(r)$. Since $D_r \p_t = D_t \p_r$, where $D$ denotes the covariant derivative, and since $\gamma$ is a null geodesic, we have  
\[
\p_r (g(J, \dot{\gamma})) = g(D_r \p_t \Phi, \dot{\gamma}) + g(J, D_r \dot{\gamma}) = g(D_t \p_r \Phi, \dot{\gamma}) = \frac{1}{2} \p_t ( g(\dot{\gamma}, \dot{\gamma}) ) = 0.
\]
Thus $g(J, \dot \gamma) = g(J, \dot \gamma)|_{r=0} = g(\p_t, N) = -1$ where $N = (1,e_1)$. A similar argument shows that $g(\p_{y_j} \Phi, \dot{\gamma}) = 0$.
\end{proof}

Recall that $\Sigma_{+,\sigma}^g$ is defined by \eqref{def_Sigma_plus_sigma} for $g$ that is simple in direction $\omega$.

\begin{Lemma}\label{lem_Sigma_splike}
Suppose that $g$ is simple in direction $\omega$. Let $\sigma \in \R$. Then $\Sigma_{+,\sigma}^g$ is spacelike
and the geodesics $\gamma_z$, $z \in \Sigma_{-,\sigma}$, are normal to $\Sigma_{+,\sigma}^g$.
\end{Lemma}
\begin{proof}
We omit writing $g$ as superscript in the proof. 
Define
    \begin{align*}
\Gamma = \{\gamma_z(s) \mid z \in \Sigma_{-, \sigma},\ s < r_+(z) + \epsilon\},
    \end{align*}
where $\eps > 0$ is as in Lemma \ref{lem_r_plus}. 
It follows from Lemma \ref{lem_r_plus} that $\Gamma$ is a smooth manifold. 
Moreover, 
    \begin{align*}
\Sigma_{+,\sigma} = \Gamma \cap \Sigma_+
    \end{align*}
and the intersection is transversal.
Let $x \in \Gamma \cap \Sigma_+$,  and write $x = \gamma_z(r)$ and $\gamma = \gamma_z$ for some $z \in \Sigma_{-, \sigma}$ and $r > 0$. 

Let us show that $T_x \Gamma$ is lightlike. Observe that $T_x \Gamma$ is spanned by $\dot \gamma(r)$ and $\p_{y^j} \Phi^g(z, r)$, $j=1,\dots,n-1$, where $y^j$ is as in Lemma \ref{lem_Jacobi}. By that lemma we have the orthogonality $\dot \gamma \perp \p_{y^j} \Phi^g$. In particular, $\dot \gamma \perp v$ for all $v \in T_x \Gamma$, and $g$ is degenerate on $T_x \Gamma$. In other words, $T_x \Gamma$ is lightlike.

By \cite[Lemma 28, p.\ 142]{oneill1983} the lightlike subspace $T_x \Gamma$ does not contain any timelike vectors and all null vectors in $T_x \Gamma$ are of the form $c \dot \gamma$ for some $c \in \R \setminus 0$. These null vectors are not in $T_x \Sigma_+$. Thus for all $v \in T_x \Gamma \cap T_x \Sigma_+$ it holds that $v$ is spacelike. Moreover, $\dot \gamma \perp v$ as $v \in T_x \Gamma$.
\end{proof}

Let $g$ be a time-oriented Lorentzian metric on $\R^{1+n}$.
We write $J^\pm_g(S)$ and $I^\pm_g(S)$ for the causal and chronological future and past of a set $S$, see \cite[p.\ 402]{oneill1983} for the definitions of these four sets.

For the rest of the section, whenever considering $g$ satisfying \eqref{assumption1}--\eqref{assumption2}, we fix the time-orientation so that $dt$ is future-directed outside $M$. Moreover, we fix $g'$ satisfying \eqref{assumption1}--\eqref{assumption2} and assume that $g'$ is simple in direction $\omega$. To simplify the notation, we omit writing $g'$ as a superscript. In particular, $\Sigma_{+,\sigma} = \Sigma_{+,\sigma}^{g'}$ in what follows. A key step in the proof of Theorem \ref{th_diffeo} is given by the following proposition.

\begin{Proposition} \label{prop_jminus}
Let $g$ be as in Theorem \ref{th_diffeo} and let $\sigma \in \R$.
Then 
\[
J_{g}^-(\Sigma_{+,\sigma}) \cap \Sigma_- = \{t \le \sigma\} \cap \Sigma_-.
\]
\end{Proposition}

The proof is based on several auxiliary lemmas. We begin by introducing some notation. For each $\sigma \in \R$ we choose a point  
    \begin{align*}
q_\sigma = (t_\sigma, x_\sigma^1, 0) \in \Sigma_{+,\sigma} \cap (\R^2 \times \{0\}).
    \end{align*}
Applying Lemma \ref{lemma_idext_surjective} to the map $\Psi(y) = \pi(\Phi(r_+(z), z))$, where $z = (\sigma, -1, y)$, $y \in \R^{n-1}$, and $\pi(t,x^1,\dots,x^n) = (x^2,\dots,x^n)$, we see that such a point exists. We choose large enough $R_0 \ge 3$ so that for all $\sigma \in \R$ the set 
    \begin{align*}
A_0 = \Sigma_{+,\sigma} \cap \{|x| \ge R_0\}, 
    \end{align*}
satisfies $A_0 = (\{\sigma + 2 + \eps\} \times \{2 + \eps\} \times \R^{n-1}) \cap \{|x| \ge R_0\}$, where $\eps > 0$ is as in Lemma \ref{lem_r_plus}.
In particular, $\max_{\sigma \in \R} x_\sigma^1 < R_0$. 
We write
\[
\mathcal{C} = \R \times B,
\]
and recall the following lemma from \cite{oksanen2024}.

\begin{Lemma} \label{lemma_c_bd}
Let $g$ satisfy \eqref{assumption1}--\eqref{assumption2} and let $p_j = (t_j, x_j) \in \p \mathcal{C}$. Then $t_1 \leq t_2-\pi$ implies $p_1 \in J_g^-(\{p_2\})$. 
\end{Lemma}

\begin{Lemma}\label{lem_near_far_J}
Let $g$ satisfy \eqref{assumption1}--\eqref{assumption2}, let $\sigma \in \R$ and let 
    \begin{align*}
R \geq \max(R_0, \pi - t_\sigma + x_\sigma^1 + \sigma + 2 + \eps).
    \end{align*}
Define
\[
A = \Sigma_{+,\sigma} \cap \{|x| \ge R\}, 
\quad
K = \Sigma_{+,\sigma} \cap \{|x| \le R\}.
\]
Then
\[
J_{g}^-(\Sigma_{+,\sigma}) = (J_{\gm}^-(A) \setminus \mathcal C) \cup J_{g}^-(K).
\]
\end{Lemma}
\begin{proof}
We omit writing $g$ as a subscript within this proof and write $\le$ for the causal relation with respect to $g$, see \cite[p.\ 402]{oneill1983} for the definition of $\le$.
Let $p \in J^-(\Sigma_{+,\sigma})$
and let us show that $p \in (J_{\gm}^-(A) \setminus \mathcal C) \cup J^-(K)$.
There is $q \in \Sigma_{+,\sigma}$ and a past-directed causal curve $\gamma$ from $q$ to $p$. If $q \in K$, then $p \in J^-(K)$, so we suppose that $q \in A$. Then $q = (\sigma + 2 + \eps, y)$ for some $y$.
If $\gamma$ does not intersect $\mathcal C$, 
then $\gamma$ is causal with respect to $\gm$ and 
$p \in J_{\gm}^-(A) \setminus \mathcal C$.
Suppose now that $\gamma$ intersects $\mathcal C$ and let $\tilde p = (t_1, x_1) \in \p \mathcal C$ be the first intersection point with $\p \mathcal C$. 
There holds $t_1 \leq \sigma + 2 + \eps - (R-1)$ as $\gamma$ is causal with respect to $\gm$ outside $\mathcal C$. Using the bound for $R$, we have, writing $t_2 = t_\sigma + 1 - x_\sigma^1$, 
    \begin{align*}
t_1 \leq  t_2 - \pi.
    \end{align*}
Writing $\tilde q = (t_2, 1, 0)$ we have $\tilde q \le q_\sigma = (t_\sigma, x_\sigma^1, 0)$ since $t_\sigma - t_2 = x_\sigma^1 - 1 > 0$. Lemma \ref{lemma_c_bd} implies $\tilde{p} \leq \tilde{q}$. Thus $p \le \tilde p \le \tilde q \le q_\sigma$. Since $q_\sigma \in K$, we have $p \in J^-(K)$.

Let now $p \in (J_{\gm}^-(A) \setminus \mathcal C) \cup J^-(K)$
and let us show that $p \in J^-(\Sigma_{+,\sigma})$. Observe that 
\[
J^-(\Sigma_{+,\sigma}) = J^-(A \cup K) = J^-(A) \cup J^-(K).
\]
If $p \in J^-(K)$, then $p \in J^-(\Sigma_{+,\sigma})$. On the other hand, suppose that 
$p \in J_{\gm}^-(A) \setminus \mathcal C$.
There is $q \in A$ and a curve $\gamma$ from $q$ to $p$ that is causal with respect to $\gm$. 
If $\gamma$ does not intersect $\mathcal C$, then $\gamma$ is causal with respect to $g$ and $p \in J^-(A)$.
Suppose now that $\gamma$ intersects $\mathcal C$ and let $\tilde p = (t_1,x_1)$ be the last intersection point with $\p \mathcal C$. As $\gamma$ is causal with respect to $\gm$, it holds that 
    \begin{align*}
t_1 \leq \sigma + 2 + \eps - (R-1) \leq t_2 - \pi.
    \end{align*}
We have again $\tilde p \le \tilde q$ for $\tilde q = (t_2, 1, 0) \in \Sigma_{+,\sigma}$. As $\tilde p$ is the last intersection point of $\gamma$ and $\p \mathcal C$, the segment of $\gamma$ from $\tilde p$ to $p$ is outside $\mathcal C$, and $p \le \tilde p$.
Since $q_\sigma \in K$ and $\tilde p \le \tilde q \le q_\sigma$, we have $p \in J^-(K)$.
\end{proof}

\begin{Lemma}\label{lem_cl_bdd}
Let $g$ satisfy \eqref{assumption1}--\eqref{assumption2} and let $\sigma \in \R$.
The set $J_{g}^-(\Sigma_{+,\sigma})$ is closed, and 
for any $z \in \Sigma_-$ there is $r > 0$ such that 
$z + r e_0 \not\in J_{g}^-(\Sigma_{+,\sigma})$ where $e_0 = (1,0) \in \R^{1+n}$. 
\end{Lemma}
\begin{proof}
We omit writing $g$ as a subscript within this proof. 
Let $R$, $A$ and $K$ be as in Lemma \ref{lem_near_far_J}.
Due to the properties of the Minkowski geometry, $J_{\gm}^-(A)$ is closed. 
The set $J^-(K)$ is closed as $K$ is compact, see \cite[Theorem 2.1]{hounnonkpe2019}.  
As $\mathcal C$ is open, the closedness of $J^-(\Sigma_{+,\sigma})$ follows from Lemma~\ref{lem_near_far_J}. 

Let $z \in \Sigma_-$. The future $J^+(z)$ is closed, so $J^+(z) \cap J^-(K)$ is also closed. Moreover, it holds that 
\[
J^+(z) \cap J^-(K) \subset J^+(K \cup \{z\}) \cap J^-(K \cup \{z\})
\]
and the latter set is compact \cite[Proposition 2.3]{hounnonkpe2019}. Thus $J^+(z) \cap J^-(K)$ is compact, and there is $T > 0$ such that
    \begin{align*}
 J^+(z) \cap J^-(K) \subset \{t \le T\}.
    \end{align*}
We see that $z + r e_0 \not\in J_{\gm}^-(A) \cup J^-(K)$ for $r \gg 0$.
\end{proof}

\begin{Lemma}\label{lem_bd_null}
Let $g$ satisfy \eqref{assumption1}--\eqref{assumption2} and let $\sigma \in \R$. Then the boundary of $J_{g}^-(\Sigma_{+,\sigma})$ is covered by normal null geodesics from $\Sigma_{+,\sigma}$.
\end{Lemma}
\begin{proof}
We omit writing $g$ as a subscript within this proof. 
By Lemma \ref{lem_cl_bdd} and \cite[Lemma 6, p.\ 404]{oneill1983} we have 
    \begin{align*}
\p J^-(\Sigma_{+,\sigma}) = J^-(\Sigma_{+,\sigma}) \setminus I^-(\Sigma_{+,\sigma}).
    \end{align*}
Let $p \in \p J^-(\Sigma_{+,\sigma})$. Then there is a causal curve from $p$ to $\Sigma_{+,\sigma}$. This curve is a normal null geodesic, since otherwise we get the contradiction $p \in I^-(\Sigma_{+,\sigma})$ in view of \cite[Theorem 51, p.\ 298]{oneill1983} and Lemma \ref{lem_Sigma_splike}. 
\end{proof}

\begin{proof}[Proof of Proposition \ref{prop_jminus}]
By Lemma \ref{lem_Sigma_splike}, 
the $g'$-geodesics $\gamma_z$, $z \in \Sigma_{-,\sigma}$, are normal to $\Sigma_{+,\sigma}$. As $g = g'$ on $\Sigma_{+,\sigma}$, it follows from Lemma \ref{lem_bd_null} that 
$\p J_{g}^-(\Sigma_{+,\sigma})$ is covered by $g$-geodesics starting from $\Sigma_{+,\sigma}$ in the directions $-\dot \gamma_z(r_+(z))$, $z \in \Sigma_{-,\sigma}$.
At this point, it is not known if all such geodesics stay on $\p J_{g}^-(\Sigma_{+,\sigma})$. Nonetheless,
\eqref{eq_beta} implies that they meet $\Sigma_-$ in the set $\Sigma_{-,\sigma}$. Thus we have 
    \begin{align}\label{upper}
\p J_{g}^-(\Sigma_{+,\sigma}) \cap \Sigma_- \subset \Sigma_{-,\sigma}.
    \end{align}
Let $z \in \Sigma_{-,\sigma}$. It follows from Lemma \ref{lem_nontrapping} that $z \in J_{g}^-(\Sigma_{+,\sigma})$. Consider the curve $\mu(r) = z + r e_0$ where $e_0 = (1,0) \in \R^{1+n}$.
If $r \le 0$, then $\mu(r) \le z$ and therefore $\mu(r) \in J_{g}^-(\Sigma_{+,\sigma})$.
On the other hand, let $r > 0$. 
By Lemma~\ref{lem_cl_bdd}, there is $R > 0$ such that $\mu(R) \not\in J_{g}^-(\Sigma_{+,\sigma})$. In view of \eqref{upper} the curve $\mu$ can intersect $\p J_{g}^-(\Sigma_{+,\sigma})$ only at $\mu(0)$. Hence the set 
    \begin{align*}
\{ r > 0 \,:\, \mu(r) \notin \overline{J_{g}^-(\Sigma_{+,\sigma})} \}
    \end{align*}
is a nonempty open and closed subset of the interval $\{ r > 0 \}$. Since the interval is connected, $\mu(r) \not\in J_{g}^-(\Sigma_{+,\sigma})$ when $r > 0$.
\end{proof}

\begin{Lemma}\label{lem_local_inv1}
Let $g$ be as in Theorem \ref{th_diffeo}.
Then the derivative of $\Phi^{g}$ is invertible at $(z,r)$ for all $z \in \Sigma_-$ and $0 < r < r_+^{g}(z)$.
\end{Lemma}
\begin{proof}
We use coordinates $(t,y)$ on $\Sigma_-$ as in Lemma \ref{lem_Jacobi}.
To get a contradiction, suppose that $\Phi^{g}$ is singular at $(z_*, r_*)  \in \Sigma_- \times \R$ 
with $z_* = (t_*, -1, y_*)$ and $0 < r_* < r_+^{g}(z_*)$, that is, 
there is a direction $(\delta t, \delta y, \delta r) \in \R^{1 + n}$ such that at $(z_*, r_*)$ there holds
    \begin{align*}
\p_t \Phi^{g} \delta t + \p_{y_j} \Phi^{g} \delta y_j + \p_r \Phi^{g} \delta r = 0. 
    \end{align*}
Applying Lemma \ref{lem_Jacobi} and using the fact that $\dot \gamma^{g} = \p_r \Phi^{g}$ is a null vector, we have
\[
0 = g(\p_t \Phi^{g} \delta t + \p_{y_j} \Phi^{g} \delta y_j + \p_r \Phi^{g} \delta r, \dot \gamma^{g}) = -\delta t. 
\]
Thus $\delta t = 0$, and the map 
\[
\Phi_*^g :  \R^{n-1} \times \R \to \R^{1+n}, \quad \Phi_*^g(y,r) = \Phi^{g}((t_*, -1, y), r)
\]
is singular at $(y_*, r_*)$ in the sense that its differential (or pushforward) does not have full rank. 

The null geodesic $\gamma_{z_*}^g$ is normal to the spacelike submanifold $\Sigma_{-,t_*}$. 
Moreover, $\gamma_{z_*}^g(r_*)$ is a focal point of $\Sigma_{-,t_*}$ along $\gamma_{z_*}^g$ since $\Phi_*^g$ is singular at $(y_*, r_*)$, see \cite[Proposition 30, p.\ 283]{oneill1983}. By \cite[Proposition 48, p.\ 296]{oneill1983} there is a $g$-timelike curve from some point $z_- \in \Sigma_{-,t_*}$ to the point $\gamma_{z_*}^g(r_+^{g}(z_*)) \in \Sigma_{+,t_*}$.
But then $z_-$ is in the interior of $J_g^-(\Sigma_{+,t_*}) \cap \Sigma_-$, see \cite[Lemma 6, p.\ 404]{oneill1983}.
This is a contradiction since $z_-$ is on the boundary of $J_g^-(\Sigma_{+,t_*}) \cap \Sigma_-$ according to Proposition \ref{prop_jminus}.
\end{proof}

\begin{proof}[Proof of Theorem \ref{th_diffeo}]
We write $V = \Phi^{g'}(\mho_0)$. Then $\p V = \Sigma_+^{g'}$.
The definition \eqref{def_r_g} of $r_+^g$ implies that
$\gamma_z^g(r_+^g(z)) \in \p V$ for $z \in \Sigma_-$ and that $\gamma_z^g(r)$ is in the interior of $V$ for $z \in \Sigma_-$ and $r < r_+^g(z)$. In particular, writing
    \begin{align*}
U = \{ (z, r) \in \Sigma_- \times \R \mid r < r_+^{g}(z)\},
    \end{align*}
it holds that $\Phi^{g} : U \to V$. The claim follows from Lemma \ref{lem_local_inv1} and Hadamard's global inverse function theorem, see e.g.\ \cite[Theorem 6.2.8]{krantz2002}, once we show that $\Phi^{g} : U \to V$ is proper. Writing $z = (t, y)$ we have 
    \begin{align*}
\Phi^{g}(t, y, r) = (t + r, y + r, r)
    \end{align*}
when $|t| + |y|$ is large or $r < 0$. Hence it is enough to observe that if a bounded sequence $(t_j, y_j, r_j)$ escapes to infinity in the sense that $|r_j - r_+^{g}(z_j)| \to 0$ then $x_j = \Phi^{g}(t_j, y_j, r_j)$ escapes to infinity in the sense that it approaches the boundary $\p V$.
\end{proof}

\begin{proof}[Proof of Lemma \ref{lem_sc}]
For the compact case, it suffices to show that 
\[
  L_\alpha = \{x \in \mathbb{R}^n \mid f_K(x) \geq \alpha \}
\]
is closed for each $\alpha \in \mathbb{R}$.
Let $\alpha \in \mathbb{R}$ and let $x_k \in L_\alpha$, $k=1,2,\dots$, converge to some $x \in \mathbb{R}^n$.
Because $K$ is closed and bounded, the fiber
\[
  F(y) = \{t \in \mathbb{R} \mid y + t e_n \in K \}
\]
is compact for every $y \in \R^n$. Indeed, it is closed as the preimage of $K$
under the continuous map $t \mapsto y + t e_n$, and bounded since $K$ is bounded.
Therefore whenever $F(y)$ is non-empty, the supremum $f_K(y)$ is attained.
Hence for each $k$ there exists $t_k \in F(x_k)$ with $t_k = f_K(x_k) \geq \alpha$.
Since the sequence $x_k$ is bounded and $K$ is compact, 
the sequence $t_k$ is bounded.
Pass to a subsequence, still denoted $t_k$, such that $t_k \to t$ for some $t \geq \alpha$.
Then
\[
x_k + t_k e_n \to x + t e_n.
\]
Since $K$ is closed and $x_k + t_k e_n \in K$, we have
$x + t e_n \in K$ so $f_K(x) \geq t \geq \alpha$.
Therefore $x \in L_\alpha$ and $L_\alpha$ is closed. 

We turn to the open case. 
It suffices to show that 
\[
  S_\alpha = \{ x \in \mathbb{R}^n \mid f_U(x) > \alpha \}
\]
is open for each $\alpha \in \mathbb{R}$.
Let $\alpha \in \mathbb{R}$ and let $x \in S_\alpha$.
By definition of the supremum, there exists $t > \alpha$ such that
$x + t e_n \in U$.
Since $U$ is open, there is $\varepsilon > 0$ such that $y \in U$ whenever $|x + t e_n - y| < \varepsilon$. In particular, if $|\tilde x - x| < \varepsilon$ then $\tilde x + t e_n \in U$ and $f_U(\tilde x) \geq t > \alpha$. Therefore $\tilde x \in S_\alpha$ and 
$S_\alpha$ is open. 
\end{proof}

\bibliographystyle{alpha}
\bibliography{master}

\end{document}